\documentclass[11pt, a4paper]{article}
\usepackage{amsmath,amssymb,amsfonts}
\usepackage{amsthm}
\usepackage[greek,french,english]{babel}

\usepackage[utf8]{inputenc}
\usepackage{verbatim}
\usepackage{a4wide}
\usepackage{epsfig,epic,eepic,graphicx}
\usepackage{graphicx}
\usepackage[active]{srcltx}
\usepackage{enumerate}
\usepackage[hidelinks]{hyperref}

\newcommand{\LU}{L_U^{-1}}
\newcommand{\cLU}{\mathcal L_U}
\newcommand{\p}{\partial}
\newcommand{\eps}{\epsilon}
\newcommand{\cD}{\mathcal D}
\newcommand{\cR}{\mathcal R}
\newcommand{\Uapp}{U^{\text{app}}}
\newcommand{\uapp}{u^{\text{app}}}
\newcommand{\be}{\begin{equation}}
\newcommand{\ee}{\end{equation}}
\newcommand{\ba}{\begin{aligned}}
\newcommand{\ea}{\end{aligned}}
\newcommand{\R}{\mathbb R}
\newcommand{\N}{\mathbb N}

\newcommand{\tb}{\tilde b}

\newcommand{\uF}{\underline{F}}
\newcommand{\uW}{\underline{W}}

\newcommand{\cN}{\mathcal N}
\newcommand{\lot}{\mathrm{l.o.t.}}
\newcommand{\cC}{\mathcal C}

\numberwithin{equation}{section}

\newtheorem{theorem}{Theorem}
\newtheorem{lemma}{Lemma}[section]
\newtheorem{proposition}[lemma]{Proposition}

\newtheorem{remark}[lemma]{Remark}
\newtheorem{corollary}[lemma]{Corollary}

\title{Separation for the stationary Prandtl equation}

\author{Anne-Laure Dalibard\footnote{Sorbonne Université, Université Paris-Diderot SPC, CNRS,  Laboratoire Jacques-Louis Lions, LJLL, F-75005 Paris} \and Nader Masmoudi\footnote{Department of mathematics, New York University in Abu Dhabi, Saadyiat Island, Abu Dhabi, UAE. E-mail:
		masmoudi@cims.nyu.edu}}

\date{\today}
\begin{document}

\bibliographystyle{amsplain}
\maketitle
\begin{abstract}
In this paper, we prove that separation occurs for the stationary Prandtl equation, in the case of adverse pressure gradient, for a large class of boundary data at $x=0$.
We justify the Goldstein singularity: more precisely, we prove that under suitable assumptions on the boundary data at $x=0$, there exists $x^*>0$ such that $\p_y u_{y=0}(x)\sim C \sqrt{x^* -x}$ as  $x\to x^*$ for some positive constant $C$, where $u$ is the solution of the stationary Prandtl equation in the domain $\{0<x<x^*,\ y>0\}$. Our proof relies on three main ingredients: the computation of a ``stable'' approximate solution, using modulation theory arguments;  a new formulation of the Prandtl equation, for which we derive energy estimates, relying heavily on the structure of the equation; and maximum principle techniques to handle nonlinear terms.

\end{abstract}

\tableofcontents
\section{Introduction}

One of the main open problems in the mathematical analysis of fluid flows  is 
the understanding of the inviscid limit in the presence of boundaries.  In the case of a fixed 
bounded domain, it is an open problem to know whether 
solutions to the Navier-Stokes system  with no slip boundary condition 
(zero Dirichlet  boundary condition)  do   converge
to a solution to the Euler system  when the viscosity 
goes to zero.  The main problem here comes from the fact that 
we cannot impose a no slip boundary  condition for the Euler system.   
To recover 
a zero Dirichlet condition, Prandtl proposed to introduce a
boundary layer \cite{Prandtl04} in a small neighborhood of the boundary
in which viscous effects are still present. It turns out that  the system that 
governs the flow in this  small neighborhood, namely the Prandtl system 
has many mathematical difficulties.  One of the outcome is that 
the justification of the approximation of the  Navier-Stokes system  by  
the Euler system  in 
the interior and the  Prandtl system  in a boundary layer is still mainly open.  
We refer to Sammartino and Caflisch \cite{SC98a,SC98b} 
for this justification   in the analytic case. There is  also a well known 
convergence criterion  due  to Kato \cite{Kato84} that states 
that the convergence from Navier-Stokes to Euler holds 
as long as there is no viscous dissipation in a small 
layer around the boundary
(see also \cite{Masmoudi98arma}).

Let us  also mention that 
when the no slip  boundary condition is replaced by a  Navier type condition or 
an inflow  condition,  
the situation gets much better: 
Bardos \cite{Bardos72} proved  that the convergence holds for some 
special type of boundary condition (vorticity equal 
to zero on the boundary)  which does not require the 
construction of any boundary layer.  For  Navier boundary conditions, a boundary layer 
can be constructed and controlled 
(see for instance   \cite{CMR98,Xin,Kelliher06,IS11,BC11,KTW11,MR12,MR16arma}).  

We are interested in the present paper in the stationary version of the Prandtl equation, namely
\be\label{P-general}
\ba
u u_x + v u_y - u_{yy}=-\frac{dp_E(x)}{dx},\quad x>0, \ y>0,\\
u_x + v_y=0,\quad x>0, \ y>0,\\
u_{|x=0}=u_0,\quad u_{|y=0}=0, \ \lim_{y\to \infty} u(x,y)= u_E (x),
\ea
\ee
where $y=0$ stands for the rigid wall, $x$ (resp. $y$) is the tangential (resp. normal) variable to the wall. The functions $u_E$, $p_E$ are given by the outer flow: more precisely $u_E$ (resp. $p_E$) is the trace at the boundary of the   tangential velocity (resp. of the pressure) of a flow satisfying the Euler equations. The functions $u_E, p_E$ are linked by the relation
$$
u_E u_E' = -\frac{dp_E(x)}{dx}.
$$

Existence results for \eqref{P-general} were first obtained by Oleinik  (see \cite[Theorem 2.1.1]{OS}). Indeed, as long as $u$ is positive (i.e. when there is no recirculation within the boundary layer), \eqref{P-general} can be considered as a non-local transport-diffusion equation in which the tangential variable $x$ plays the role of ``time". The function $u_0$, which is the input flow, is then considered as an ``initial data". However, this point of view breaks down as soon as $u$ takes negative values. Physical experiments and numerical simulations show that such behavior may occur; in this case, the boundary layer seems to detach itself from the boundary. This phenomenon is therefore referred to as ``boundary layer separation" (see figure \ref{fig:stewartson}).
\begin{figure}
	\includegraphics[width=\textwidth]{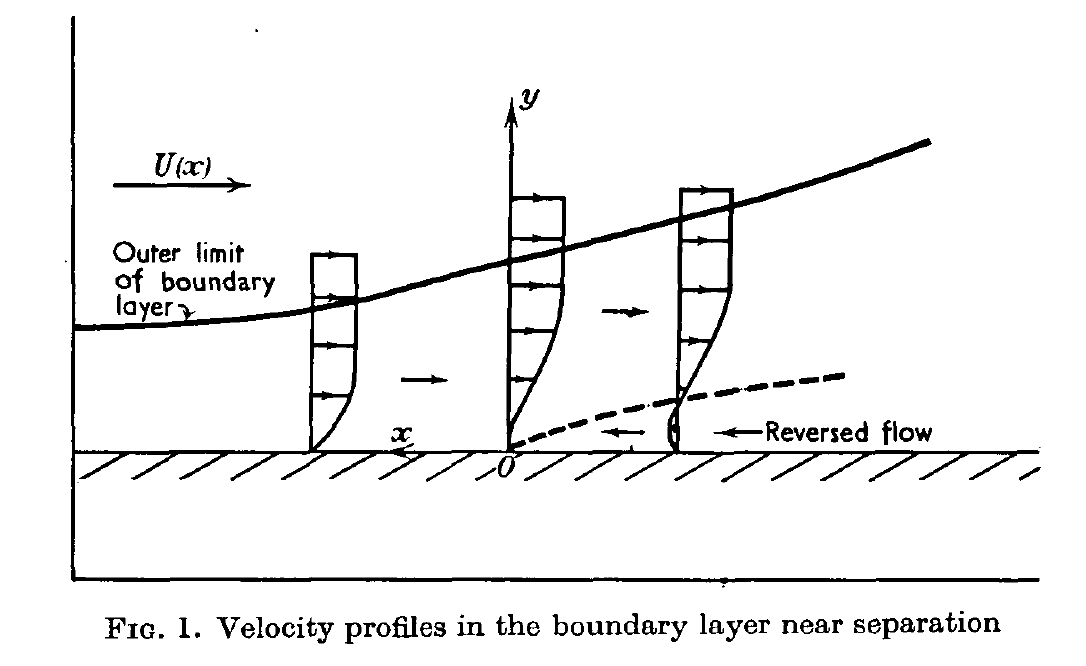}
	\caption{From Stewartson \cite{Stewartson}.}
\label{fig:stewartson}
\end{figure}

The goal of this paper is to prove that separation does occur for the stationary Prandtl model \eqref{P-general}, and to give a quantitative description of the solution close to (but on the left of) the separation point. In particular, we will justify rigorously the ``Goldstein singularity" (see \cite{Goldstein}). Note that a shorter version of this work was published in \cite{DM16}.

\subsection{Setting of the problem and state of the art}

The first mathematical study of the stationary Prandtl equation was performed by Oleinik (see \cite{OS}):

\begin{proposition}[Oleinik]\label{prop:Oleinik}
	Let $\alpha>0$, $X\in ]0, + \infty]$. Let $u_0\in    \mathcal C^{2,\alpha}_b(\R)$ such that $u_0(0)=0$, $u_0'(0)>0$, $\lim_{y\to \infty}u_0(y)=u_E(0)>0$, and such that $u_0(y)>0$ for $y>0$. Assume that $dp_E/dx\in \mathcal C^1([0,X])$, and that for $y\ll1$ the following compatibility condition is satisfied
	\be\label{compatibilite}
	u_0''(y) - \frac{dp_E(0)}{dx}=O(y^2).
	\ee
	
	Then there exists $x^*\leq X$ such that equation \eqref{P-general} admits a solution $u\in \mathcal C^1([0, x^*[\times \R_+)$ enjoying the following properties:
	\begin{itemize}
		\item Regularity: $u$ is bounded and continuous in $[0, x^*]\times \R_+$, $\p_y u, \p_y^2 u$ are bounded and continuous in  $[0, x^*[\times \R_+$, and $\p_x u$, $v$ and $\p_y v$ are locally bounded and continuous in $[0, x^*[\times \R_+$;
		
		\item Non-degeneracy: $u(x,y)>0$ for all $y>0$ $x\in [0, x^*[$, and for all $\bar x<x^*$ there exists $y_0>0$, $m>0$ such that $\p_y u(x,y)\geq m$ for all $(x,y)\in [0,\bar x]\times [0,y_0]$.
		
		\item Sufficient condition for global solutions: if $\frac{dp_E(x)}{dx}\leq 0$, then the solution is global, i.e. $x^*=X$.
	\end{itemize}
	
\end{proposition}

In this paper, we are interested in the case where the solution of \eqref{P-general} is not global: more precisely, we consider the equation \eqref{P-general} with $dp_E/dx=1$, i.e.
\be\label{Prandtl}\tag{P}\ba
u u_x + v u_y - u_{yy}=-1,\quad x\in (0, x_0), \ y>0,\\
u_x + v_y=0,\quad x\in (0, x_0), \ y>0,\\
u_{|x=0}=u_0,\quad u_{|y=0}=0, \ \lim_{y\to \infty} u(x,y)=u_E (x),\ea
\ee
with $u_E(x)=\sqrt{2(x_0-x)}$, for some $x_0>0$, and $u_0$ satisfies the assumptions of  Proposition \ref{prop:Oleinik}. Hence it is known that local solutions (in $x$) of \eqref{Prandtl} do exist. However, heuristically, it can be expected that the negative source term will diminish the value of the tangential velocity $u$, and that there might exist a point $x^*$ beyond which  the result of Proposition \ref{prop:Oleinik} cannot be used to extend the solution. More precisely, it can be checked easily that the compatibility condition \eqref{compatibilite} is propagated by equation \eqref{Prandtl}. As a consequence, we have $x^*<x_0$ if and only if one of the following two conditions is satisfied:
\begin{enumerate}[(i)]
	\item  $u_y(x^*,0)=0$;
	\item There exists $y^*>0$ such that $u(x^*,y^*)=0$.
\end{enumerate}

In order to simplify the mathematical analysis, we will work with solutions of \eqref{Prandtl} that are increasing in $y$. This property is propagated by the equation, and ensures that situation (ii) above never occurs. Consequently, for solutions which are increasing in $y$, we have $x^*<x_0$ if and only if
\be\label{separation}
\frac{\p u}{\p y}(x^*,0)=0.
\ee
In the Physics literature (see for instance the seminal work of Goldstein \cite{Goldstein}, followed by the one of Stewartson \cite{Stewartson}), this condition is used as a characterization of the  ``separation point''. 

The first computational works on this subject go back to  Goldstein \cite{Goldstein} and Landau \cite[Chapter 4, \S 40]{Landau}. In particular, Goldstein uses an asymptotic expansion in self-similar variables to compute the profile of the singularity close to the separation point. These computations are later extended by Stewartson \cite{Stewartson}. However, these calculations are formal; furthermore, some of the coefficients of the asymptotic expansion cannot be computed by either method. Independently, Landau proposes another characterization of the separation point, and gives an argument suggesting that $\p_y u_{|y=0}\sim \sqrt{x^*-x}$ close to the separation point.

On the other hand, in the paper \cite{E-prandtl} Weinan E announces a result obtained in collaboration with Luis Caffarelli. This result states, under some structural assumption on the initial data, that the existence time $x^*$ of the solutions of \eqref{Prandtl} in the sense of Oleinik is finite, and that the family $u_\mu(x,y):=\frac{1}{\sqrt{\mu}} u(\mu (x^*-x), \mu^{1/4} y)$ is compact in $\mathcal C(\R_+^2)$.
Moreover, the author states two technical Lemmas playing a key role in the proof. However, to the best of our knowledge, the complete proof of this result was never published.

Let us also mention recent works by Guo and Nguyen \cite{GN} and by Iyer \cite{Iyer1,Iyer2,Iyer3}, in which the authors justify the Prandtl expansion either over a moving plate or over a rotating disk. Note that in these two cases, the velocity of the boundary layer on the boundary is non zero, which somehow prevents recirculation and separation. 

In the time-dependent framework, boundary layer separation has also been tackled recently by Kukavica, Vicol and Wang \cite{kukavica2017van}, extending computations by Engquist and E \cite{EE}: starting from an analytic initial data, for a specific Euler flow, the authors prove that some Sobolev norm blows up in finite time. This is known as the van Dommelen and Shen singularity. Note that in this time-dependent context, separation is defined as the apparition of a singular behaviour, which is a somewhat  different notion from the one we are describing in the present paper. This is related to the bad mathematical properties of the time-dependent Prandtl equation, which is known to be locally well-posed in analytic or Gevrey spaces \cite{SC98a,LCS,KV,KMVW,GVM}, but ill-posed in Sobolev spaces \cite{Grenier00,GVM}.

\subsection{Main result}

Our main result states that for a suitable class of initial data $u_0$, the maximal existence ``time''  $x^*>0$ of the solution given by Oleinik's Theorem is finite: in other words, setting
$$
\lambda(x):= \p_y u_{|y=0},
$$
there exists $x^*\in ]0, + \infty[$ such that $\lim_{x\to x^*} \lambda(x)=0$. Furthermore, for this class of initial data, we are able to quantify the rate of cancellation of $\lambda(x)$. 

Let us now explicit our assumptions on the initial data $u_0$:
\begin{enumerate}[(H1)]
\item $u_0\in \mathcal C^7(\R_+)$, $u_0$ is increasing in $y$ and $\lambda_0:= u_0'(0)>0$;
\item There exists a constant $C_0>0$ such that 
$$
\ba
\forall y \geq 0,\quad -C_0\inf (y^2, 1)\leq u_0''(y)-1\leq 0,\\
C_0^{-1}\leq -u_0^{(4)}(0)\leq C_0,\\
\|u_0\|_{W^{7,\infty}} \leq C_0.
\ea
$$
\item $u_0=\uapp_0 + v_0$, where
\begin{eqnarray*}
\uapp_0&=& \lambda_0 y + \frac{y^2}{2} + u_0^{(4)}(0) \frac{y^4}{4!}\\&& - c_7  (u_0^{(4)}(0))^2 \frac{y^7}{\lambda_0} + c_{10}(u_0^{(4)}(0))^3\frac{y^{10}}{\lambda_0^2} + c_{11} (u_0^{(4)}(0))^3\frac{y^{11}}{\lambda_0^3}\quad \text{for } y \leq \lambda_0^{3/7},\\
|\uapp_0|&\leq &C_0 \quad \text{for } y \geq \lambda_0^{3/7},
\end{eqnarray*}
and
$$
|v_0|\leq C_0\left(\lambda_0^{-\frac{3}{2}} (\lambda_0y^7 + c_8 y^8) + \lambda_0^{-2} y^{10} + \lambda_0^{-3} y^{11}\right) \quad \text{for } y \leq \lambda_0^{3/7}.
$$
In the expressions above, the constants $c_i$ are universal and can be computed explicitely.

\end{enumerate}

\begin{remark}
\begin{itemize}
\item These assumptions are actually not optimal: in fact, condition (H3) merely ensures that some energy-like quantities are small enough. However, the actual condition we need is complicated to state at this stage: we refer to the statement of Theorem \ref{thm:rescaled}, in  rescaled variables, for a less stringent condition. 

\item Notice that $|v_0|\ll \uapp_0$ if $\lambda_0\ll 1$: the term $v_0$ is the initial data for the corrector term $v=u-\uapp$. The main issue of the paper is to have a good control of $v$ close to $y=0$.

\item The monotony assumption  on $u_0$ ensures that separation occurs at $y=0$. The monotony is preserved by the Prandtl equation for $x>0$.

\item Notice that we prescribe the Taylor expansion of $u_0$ up to order 7. In other words, we impose a high order compatibility condition on the initial data, because we need to derive estimates on derivatives of $u$.

\end{itemize}

\end{remark}

\begin{theorem}
\label{thm:main}
Consider the Prandtl equation with adverse pressure gradient \eqref{Prandtl} and with an initial data $u_0\in \mathcal C^7(\R_+)$ satisfying (H1)-(H3). Then for any $\eta>0$, $C_0>0$, there exists $\epsilon_0>0$ such that if $\lambda_0<\epsilon_0$, the ``existence time'' $x^*$ is finite, and $x^*=O(\lambda_0^2)$. Furthermore, setting $\lambda(x):=\p_y u_{|y=0}(x)$, there exists a constant $C>0$, depending on $u_0$, such that
$$
\lambda(x)\sim C \sqrt{x^*-x}\quad \text{as } x\to x^*.
$$

\end{theorem}
The proof of Theorem \ref{thm:main} relies on several ingredients: the first step is to perform a self-similar change of variables, using $\lambda(x)$ as a scaling factor. Then the issue is to control the variations of $\lambda$, or more precisely, of $b:=-2 \lambda_x \lambda^3$. The method thanks to which we construct an approximate solution and find the ideal ODE on $b$ is inspired from the theory of modulation of variables, which was initiated formally by Zakharov and Shabat (see \cite{ZS} and the presentation in the book  by Sulem and Sulem \cite{SulemSulem}) and rigorously applied by Merle and Rapha\"el to blow-up phenomena in the nonlinear Schr\"odinger equation \cite{MerleRaphael1,MerleRaphael2}.

Once the approximate solution is constructed, the whole problem amounts to controlling the remainder $v$. To that end, we exhibit a  transport-diffusion structure of equation \eqref{Prandtl} (or of its rescaled version, see equation \eqref{eq:U-LU}). Let us emphasize that this structure, to our knowledge, is entirely new. We perform energy estimates that rely strongly on the structure of the equation. In order to handle nonlinearities, we will also need to control $u$ in $L^\infty$. Therefore we derive pointwise estimates on $u$ and its derivatives by constructing sub and super-solutions and using the maximum principle. 

Let us point out that in order to carry these estimates, we will use three different versions of the equation. The first one is merely a rescaling of equation \eqref{Prandtl} (see \eqref{rescaled}). It will be used to compute explicitly the approximate solution and find the ODE on $b$. The second one is a transport equation with a non local diffusion term (see \eqref{eq:V-2}). Its purpose is to perform energy estimates, and the major difficulty will be to find good coercivity inequalities on the diffusion. Eventually, we will use a change of variables to transform \eqref{rescaled} into a nonlinear transport diffusion equation of porous medium type (see \eqref{VM}). This last form was already used by Oleinik in \cite{OS} and will be suitable for the maximum principle and will help us prove the $L^\infty$ estimates

In the next section, we present our scheme of proof and state our main intermediate results. The reader that is not interested in the technical details of the proof can focus on section 2, that gives an overall idea of the main arguments involved.
The third section is devoted to the construction of sub and super solutions. In section \ref{sec:proof-statements-algebraic}, we introduce several tools that play an important role in the energy estimates: coercivity of the diffusion term, commutator Lemma, computation of the remainder... Eventually, we prove the energy estimates in section \ref{sec:proof-statements-energy}.

\begin{remark}
	Our result actually gives much more information on $u$: in fact, we construct an approximate solution $\uapp$, which contains the main order terms in the Taylor expansion of $u$, and we control $v=u-\uapp$. As a corollary, we find that  the sequence of functions $(u^\mu)_{\mu>0}$ from the statement of Luis Caffarelli and Weinan E converges towards $z^2/2$ in the zone $z\leq \mu^{-1/12} \xi^{1/6}$, $\xi\lesssim 1$ (see Remark \ref{rem:ECaf} for more details). Hence our result holds under more stringent assumptions on the initial data, but on the other hand it gives a much more quantitative and precise description of the asymptotic behaviour.

\end{remark}

\section{Strategy of proof}

\subsection{Self-similar change of variables}

Let us first recall that equation  \eqref{Prandtl} has a scaling invariance: indeed, if $(u,v)$ is a solution of  \eqref{Prandtl}, then for any $\mu>0$, the couple $(u_\mu, v_\mu)$ defined by
$$
u_\mu= \frac{1}{\sqrt{\mu}} u(\mu x, \mu^{1/4} y),\quad
v_\mu= \mu^{1/4} v(\mu x, \mu^{1/4} y),
$$
is still a solution of  \eqref{Prandtl}. This scaling invariance has been used by Goldstein \cite{Goldstein} and Stewartson \cite{Stewartson} to compute exact solutions of  \eqref{Prandtl} close to the separation point. These special solutions were sought as infinite series in some rescaled variables.

In the present article, the idea is to perform a change of variables which relies on this scaling invariance and 
which depends on the solution itself.  It  incorporates information on the ``separation rate'', i.e. on the speed of cancellation of $\p_y u_{|y=0}$. This type of idea was used  by Franck Merle and Pierre Raphaël in the context of singularity analysis for the nonlinear Schrödinger equation \cite{MerleRaphael1,MerleRaphael2}. More precisely, define
$$
\lambda(x):=\p_y u_{|y=0}\quad \text{and} \quad Y= \frac{y}{\lambda (x)}.
$$
We also change the tangential variable and define the variable $s$ by
\be\label{def:x-s}
 \frac{ds}{dx}=\frac{1}{\lambda^4(x)}.
\ee
Then the new unknown function is
\be\label{change-var}
U(s,Y):=\lambda^{-2}(x(s)) u(x(s), \lambda(x(s)) Y).
\ee
It can be easily checked that $U$ is a solution of the equation
\be\label{rescaled}
 UU_s - U_Y\int_0^Y U_s-b U^2 + \frac{3b}{2} U_Y \int_0^Y U - U_{Y Y}=-1,
\ee 
where
\be\label{def:b}
b=-2\lambda_x \lambda^3= - 2 \frac{\lambda_s}{\lambda}.
\ee
The boundary conditions become 
\be\label{cond:U-1}U_{|Y=0}=0,\quad \lim_{Y\to \infty} U(s,Y)= U_\infty(s),\ee
 where $ U_\infty$ satisfies $ U_\infty  U'_\infty - b U_\infty^2=-1$. Moreover, thanks to the definition of  $\lambda$, we have 
 \be \label{cond:U-2}\p_Y U_{|Y=0}=1.\ee

 From now on, we will work with equation \eqref{rescaled} only. The goal is to construct an approximate solution of \eqref{rescaled},  together with $b(s)$ and $\lambda(s)$,  having  nice stability properties as $s\to\infty$. Note that the limit $s\to \infty$  corresponds to the limit $x\to x^*$ in the original variables. As we will see in the next paragraph, the stability properties of the approximate solution are intimately connected to the asymptotic law of $b$ as $s\to \infty$. 
 Eventually, the asymptotic behavior of $b$ will dictate the rate of cancellation of $\lambda$ close to $x=x^*$. We  prove that the behavior $b(s)\sim s^{-1}$ is stable. This asymptotic law corresponds to the separation rate announced in  Theorem \ref{thm:main}, namely $\lambda(x)\sim C \sqrt{x^*-x}$.
 
 In the next paragraphs, we explain how we construct the approximate solution, and which energy estimates are used to prove its stability. We deal with nonlinearities in the equation by using the maximum principle, together with Sobolev embeddings. Let us recall that we will in fact use three different forms of equation \eqref{rescaled}:
 \begin{itemize}
\item Due to its polynomial form, equation \eqref{rescaled} itself is very useful to construct the approximate solution and find the correct asymptotic law for $b$;
\item In order to perform energy estimates, we will transform \eqref{rescaled} into a transport-diffusion equation (with a non-local diffusion term), see \eqref{eq:U-LU} and \eqref{eq:V-2};
\item Eventually, in order to use the maximum principle, we rely on a third version of \eqref{rescaled}, that uses von Mises variables. The equation then becomes a nonlinear local transport-diffusion equation.
 \end{itemize}

\subsection{Construction of an approximate solution}

The heuristic idea behind the construction of stable approximate solutions is the following: we look for an approximate solution $\Uapp$ of \eqref{rescaled} with a remainder as small as possible. In particular, the remainder for $\Uapp$ should have the lowest possible growth at infinity. This implies  that the function $\Uapp$ itself should have the  lowest possible growth as $Y\to \infty$, as we shall see in a moment. As in the work of Merle and Raphaël in the context of the nonlinear Schrödinger equation, this low growth condition has an immediate impact on the asymptotic behavior of the function $b$.

We decompose the definition of approximate solutions into three zones: the main zone goes from $0$ to $s^\alpha$, for some $\alpha>0$ to be defined later on. In this zone, we compute a Taylor expansion of $U(s,Y)$ for $Y$ close to zero, and we try to push the expansion as far as possible, which amounts to the ``low growth condition'' explained above. In the second zone, we only keep the largest term in the Taylor expansion, namely ${Y^2/2}$. It can be checked that  $Y^2/2$ is a stationary solution of \eqref{rescaled}. This stationary solution corresponds to a solution of \eqref{Prandtl} which is independent of $x$ and scaling invariant, namely $(x,y)\mapsto y^2/2$. In the third zone, we connect $Y^2/2$ to an asymptotic profile $\Uapp_\infty(s)$. Notice that if $b(s)=s^{-1} + O(s^{-\eta-1})$ for some $\eta>0$, then $U_\infty(s)=s+1+o(1)$, and therefore we also take $\Uapp_\infty(s)\sim s$.

Throughout this paragraph, we will rely on the polynomial form on the rescaled Prandtl equation, namely \eqref{rescaled}.

\vskip2mm
$\bullet$\textit{ Taylor expansion of $U$ for $Y$ close to zero:}

Let us first recall that thanks to the change of variables \eqref{change-var}, we have
$$
U(s,0)=0,\quad \p_Y U(s,0)=1.
$$
It then follows from \eqref{rescaled} that
$$
\p_{YY} U(s,0)=1.
$$
The first terms of the Taylor expansion of $U$ for $Y$ close to zero are therefore $Y + \frac{Y^2}{2}$.

The first natural idea is to define a sequence of polynomials in $Y$ with coefficients depending on $s$ thanks to the induction relation
\be\label{recurrence-1}
\begin{aligned}
U_1(s,Y):=Y + \frac{Y^2}{2},\\
\p_{YY} (U_{N+1}-U_N):=1 + U_N\p_s U_N- \p_Y U_N\int_0^Y \p_s U_N - b U_N^2 + \frac{3b}{2} \p_Y U_N \int_0^Y U_N-\p_{YY } U_N.
\end{aligned}
\ee
We obtain easily that
\be\label{def:U2}
U_2(s,Y):=Y + \frac{Y^2}{2}-a_4 b Y^4,\quad \text{with } a_4=\frac{1}{48}.
\ee
Let us now compute the error terms generated by $U_2$. We have
\begin{eqnarray*}
&&U_2\p_s U_2- \p_Y U_2\int_0^Y \p_s U_2 - b U_2^2 + \frac{3b}{2} \p_Y U_2 \int_0^Y U_2 - \p_{YY} U_2 + 1\\
&=&-a_4\left( \frac{4}{5} b_s + \frac{13}{10} b^2\right)Y^5 - \frac{3}{10}a_4 \left(b_s + {b^2}\right) Y^6+a_4^2 \frac{b}{5} \left(b_s + {b^2}\right)Y^8.
\end{eqnarray*}
Let us recall that we expect that $b(s)= O(s^{-1})$ as $s\to \infty$. Therefore the coefficient of the last term in the right-hand side is one order of magnitude smaller than the first two terms. We thus focus on the comparison between the first two terms in the right-hand side.
As explained above, the goal is to choose the approximate solution with the smallest growth at infinity. Note that the remainder term $ \left(b_s + {b^2}\right) Y^6$ would yield in $U_3$ a term proportional to $ \left(b_s + {b^2}\right) Y^8$, whereas the remainder term $\left( \frac{4}{5} b_s + \frac{13}{10} b^2\right)Y^5$  would yield  a term proportional to $\left( \frac{4}{5} b_s + \frac{13}{10} b^2\right)Y^7$. Consequently, we choose to ``cancel out'' the term $ \left(b_s + {b^2}\right) Y^8$ in $U_3$. In other words, in the induction formula \eqref{recurrence-1} defining the sequence $(U_N)_{N\geq 1}$, we replace every occurrence of  $b_s$ by $-b^2$. The polynomial $U_3$ is therefore defined by
$$
U_3(s,Y):= Y + \frac{Y^2}{2} - a_4 b Y^4 - a_7 b^2 Y^7,\quad \text{with }a_7=\frac{1}{84} a_4.
$$
It follows that for $Y\ll 1$,
\be\label{V-0}\ba
U(s,Y)\simeq U_3(s,Y) + V_3(s,Y),\\
\text{where} 
\quad
V_3(s,Y)=-a_7\frac{8}{5}(b_s+b^2) Y^7 - a_4\frac{3}{10\times 7\times 8}\left(b_s + {b^2}\right) Y^8 + O(Y^{10}).\ea\ee

\begin{remark}
In the work of   Merle and  Rapha\"{e}l, the choice of the parameters $\lambda$ and $b$ stems from orthogonality properties of the quantity $U-\Uapp$ on some well chosen functions. 
 In the present case, these orthogonality properties can be seen as a cancellation at high enough order of $U-\Uapp$ at $Y=0$.

\end{remark}

For technical reasons, it is necessary to push further the expansion of $U$. We thus compute $U_4$. We find that
\begin{eqnarray*}
&&U_3\p_s U_3- \p_Y U_3\int_0^Y \p_s U_3 - b U_3^2 + \frac{3b}{2} \p_Y U_3 \int_0^Y U_3 - \p_{YY} U_3 + 1\\
&=&(b_s+ b^2)\left[-\frac{4}{5}a_4 Y^5 - \frac{3}{10} a_4 Y^6 - \frac{a_4}{60} b Y^8 -\frac{3}{4} a_7 b Y^9 + \frac{3}{5} a_4 a_7 b^2 Y^{11} + \frac{1}{4} a_7^2 b^3 Y^{14}\right]\\
&&-\frac{27}{16} a_7 b^3 Y^8 - \frac{3}{16} a_7 b^3 Y^9 + \frac{11}{16}a_4 a_7 b^4 Y^{11} + \frac{3}{8}a_7^2
 b^5 Y^{14}.\end{eqnarray*}

It follows that
$$
U_4=U_3 - a_{10} b^3 Y^{10} - a_{11} b^3 Y^{11} + a_{13} b^4 Y^{13} + a_{16} b^5 Y^{16},
$$
where 
\be\label{def:a10a11}
a_{10}=\frac{27}{16\times 90} a_7,\ a_{11}=\frac{3}{16\times 110} a_7 ,\ a_{13}=\frac{11 a_4 a_7}{13\times 12\times 16},\ a_{16}=\frac{a_7^2}{16\times 5\times 8}.
\ee
\vskip2mm

$\bullet$ \textit{Definition of the approximate solution:}

We now define the approximate solution $\Uapp$ in the following way: let $\Theta\in \mathcal C^2(\R_+)$  be such that $\Theta(\xi)= \frac{\xi^2}{2}$ for $\xi\leq c_0$ for some $c_0>0$, $\Theta$ strictly increasing, and $\Theta(\xi)\to 1$ as $\xi\to \infty$. Let $\chi \in \mathcal C^\infty_0(\R_+)$  be such that $\chi\equiv 1$ in a neighbourhood of zero. We take
\be\label{def:Uapp}
\Uapp(s,Y):=\chi\left(\frac{Y}{s^{2/7}}\right) \left[ Y- a_4 b Y^4 - a_7 b^2 Y^7 -a_{10} b^3 Y^{10} - a_{11} b^3 Y^{11}\right] + \frac{1}{b} \Theta \left( \sqrt{b} Y\right).
\ee
Notice that $\Uapp\simeq U_4$ as long as $Y\lesssim s^{2/7}$ (the highest order terms have been removed, mainly because they do not lower the size of the remainder while making the computations heavier), and that $\Uapp \to \frac{1}{b}$ as $Y\to \infty$. Therefore we do not require that $\Uapp - U(s,Y)\to 0$ as $Y\to \infty$. But this is not an issue, since we will measure the distance between $U$ and $\Uapp$ in weighted Sobolev spaces, with weights decreasing polynomially (with a large power) after $s^{\beta}$, for some $\beta<{2/7}$.

\begin{remark}
The zone after which we cut-off the first part of the approximate solution is irrelevant: we could have used any cut-off $\chi(\cdot/s^\alpha)$ as long as $\alpha \in ]1/4, 1/3[$. The choice $\alpha=2/7$ simplifies some of the statements on $\Uapp$ since it ensures that $Y$ and $-a_4 b Y^4$ are the largest terms in $Y- a_4 b Y^4 - a_7 b^2 Y^7 -a_{10} b^3 Y^{10} - a_{11} b^3 Y^{11}$.

\end{remark}

We also set, in the rest of the paper, $V:= U-\Uapp$. The computations above and in particular \eqref{V-0} show that
$$
V(s,Y)= -a_7\frac{8}{5}(b_s+b^2) Y^7 - a_4\frac{3}{10\times 7\times 8}\left(b_s + {b^2}\right) Y^8 + O(Y^{10})\quad \text{for }0<Y\ll 1.
$$
In particular, let $\cN$ be a (semi-)norm on functions $W\in\mathcal C^8(\R_+)$ such that $W=O(Y^7)$ for $Y$ close to zero. Assume that there exists a constant $C_\cN$ such that
$$
\cN(W)\geq C_\cN |\p_Y^7 W_{|Y=0}|.
$$
Then $\cN(U-\Uapp)\geq C_\cN |b_s+b^2|$. Therefore the goal of the paper is to use the structure of the equation \eqref{rescaled} in order to find a semi-norm $\cN$ which satisfies the assumptions above, and to prove that
$$
\cN(U-\Uapp)\leq C s^{-2-\eta},
$$
for some positive constant $\eta$ and for $s$ sufficiently large, or alternatively, that
\[
\int_{s_0}^\infty s^{3+2\eta} \cN(U(s)-\Uapp(s))\:ds<+\infty.
\]
Indeed, we have the following result:
\begin{lemma}
Let $s_0:= b_0^{-1}$.
Assume that the variables $x,s$ and the parameters $\lambda, b$ are related by the formulas \eqref{def:x-s}, \eqref{def:b} with the initial conditions $\lambda_{|s=s_0}= \lambda_0$, $b_{|s=s_0}= b_0= s_0^{-1}$, and that there exists a constant $c_0$ such that
$$
c_0^{-1}\leq \frac{b_0}{\lambda_0^2} \leq c_0.
$$

Assume furthermore that there exist   constants  $\eta>0$ and $\eps\in (0,1)$ such  that for all  $s\geq s_0$,
\be\label{hyp:b-mod-rate}
\ba
\int_{s_0}^{\infty} s^{3+2\eta} |b_s+b^2|^2\:ds<\infty,\\
\frac{1-\eps}{s}\leq b(s)\leq \frac{1+\eps}{s}.\ea
\ee

Then there exists $x^*>0$ such that $\lambda(x)\to 0$ as $x\to x^*$. Furthermore, if $\lambda_0\ll 1$, then $x^*= O(\lambda_0^2)$ and
there  exists a constant $C$ such that
$$
\lambda(x)\sim C (x^*-x)^{1/2}\quad \text{as }x\to x^*.
$$

\label{lem:modulation}
\end{lemma}
\begin{proof}
First, setting
\[
J:=\int_{s_0}^{\infty} s^{3+2\eta} |b_s+b^2|^2\:ds
\]
and
using Lemma \ref{lem:b} in the Appendix, we know that
$$
b(s)=\frac{1}{s} + r(s),
$$
where 
$$
\forall s\geq s_0,\quad |r(s)|\leq \eps\frac{1+\eps}{1-\eps} \frac{s_0}{s^2} + J^{1/2} \frac{1+\eps}{(1-\eps)^2 } \frac{1}{s^{1+\eta}}.
$$
As a consequence,
$$
\int_{s_0}^\infty |r(s)|\:ds\leq \eps\frac{1+\eps}{1-\eps} + J^{1/2} \frac{1+\eps}{(1-\eps)^2 \eta s_0^\eta} <\infty. 
$$
From \eqref{def:x-s} and \eqref{def:b}, it follows that
$$
\lambda(s)= \lambda(s_0) \exp\left(-\frac{1}{2}\int_{s_0}^s b(s')\:ds'\right)=\lambda(s_0)\left(\frac{s_0}{s}\right)^{1/2} \exp\left(-\frac{1}{2}\int_{s_0}^s r(s')\:ds'\right).
$$
We have $\lambda(s_0)=\lambda_0= O(s_0^{-1/2})$ by assumption. Moreover, according to the estimate of $r$ above, the function
$
\psi(s):= \exp\left(-\frac{1}{2}\int_{s_0}^s r(s')\:ds'\right)
$ has a finite, strictly positive limit $\psi_\infty$ as  $s\to \infty$. As a consequence $\lambda(s)=\left(\lambda_0 s_0^{1/2}\right)\psi(s) s^{-1/2}$ for all $s\geq s_0$. According to \eqref{def:x-s}, we have
$$
x^*:=\int_{s_0}^\infty \lambda(s)^4\:ds=\left(\lambda_0 s_0^{1/2}\right)^4\int_{s_0}^\infty \psi(s)^4 s^{-2}\:ds<\infty.
$$
Thus separation occurs at a finite $x^*$.  Moreover, 
$$
x^*-x(s)=\left(\lambda_0 s_0^{1/2}\right)^4\int_{s}^\infty \psi(s')^4 s'^{-2}\:ds'\sim \left(\lambda_0 s_0^{1/2}\right)^4 \psi_\infty^4 s^{-1}\quad \text{as } s\to \infty.
$$
Going back to the original variables, we deduce that
$$
\lambda(x)\sim \frac{1}{\left(\lambda_0 s_0^{1/2}\right) \psi_\infty} (x^*-x)^{1/2}\quad \text{as }x\to x^*.
$$
Using the above formulas, we also infer that if $s_0\gg 1$, $x^*= O(s_0^{-1})= O(\lambda_0^2)$.

\end{proof}
\begin{remark}
Notice that  the precise value of the separation point $x^*$ depends on the whole function $u_0(y)$ (and not only on its derivatives at $y=0$) through the function $\psi(s)$. This intricate dependance might explain why the coefficients in Goldstein's expansion were undetermined.

\end{remark}

\begin{remark}
Let us now give some examples of norms $\cN$ such that
\be\label{hyp:cN}
\cN(W)\geq C_\cN |\p_Y^7 W_{|Y=0}|
\ee
for $W\in \mathcal C^8(\R_+)$ with $W=O(Y^7)$ for $Y$ close to zero. We can take for instance
$$
\cN(W)^2:=\int_0^\infty (\p_Y^7 W)^2 + (\p_Y^8 W)^2,
$$
or
$$
\cN(W)^2:=\int_0^\infty \left(\frac{W}{Y^7}\right)^2 + \left(\p_Y\frac{W}{Y^7}\right)^2 .
$$
More generally, we can use any norm $\cN$ such that
$$
\cN (W)\gtrsim \left\| Y^{k-7} \p_Y^k W\right\|_{H^1(0,Y_0)},
$$
for some fixed $Y_0>0$ and for any $k\in \{0,\cdots, 7\}$. The norm $\cN$ we will use eventually will be  equivalent to a linear combination of such norms in a neighbourhood of $Y=0$.
\end{remark}

\subsection{Error estimates}

In this paragraph, we explain roughly how estimates on  $V:=U-\Uapp$ are derived. More details will be given in sections \ref{sec:proof-statements-algebraic} and \ref{sec:proof-statements-energy}. We emphasize that all energy estimates written in this paper are new. The first step is to compute an evolution equation of transport-diffusion type on $V$. To that end, let us consider equation \eqref{rescaled}, and set, for $W_1, W_2\in \mathcal C(\R_+)$,
$$
L_{W_1} W_2 := W_1 W_2 - \p_Y W_1 \int_0^Y W_2,
$$
so that equation \eqref{rescaled} can be written as
$$
L_U U_s - b U^2 + \frac{3b}{2} U_Y \int_0^Y U - \p_{YY} U =-1.
$$

Notice that since $U(s,Y)>0$ for all $s,Y>0$,
$$
\frac{L_U W}{U^2} = \p_Y\left(\frac{\int_0^Y W}{U}\right).
$$
Hence we can define the inverse of  the operator $L_U$, for functions $f$ such that $f(Y)/Y^2$ is integrable in a neighbourhood of zero: we have
\be\label{LU}
\LU f= \left( U \int_0^Y \frac{f}{U^2}\right)_Y= U_Y \int_0^Y \frac{f}{U^2} + \frac{f}{U}.
\ee
As a consequence, the equation on $U$ can be written as
$$
\p_s U + b \LU \left( \frac{3}{2} U_Y \int_0^Y U - U^2\right) - \LU(\p_{YY} U -1)=0.
$$
It follows immediately from the definition that
$$
\LU (U^2)= \p_Y \left( U \int_0^Y 1\right)= (YU)_Y.
$$
On the other hand,
\begin{eqnarray*}
\LU \left(U_Y \int_0^Y U\right) &=& \LU \left(U_Y \int_0^Y U - U^2\right) + (YU)_Y\\
&=&-\LU L_U U + (YU)_Y\\
&=& -U + (YU)_Y= YU_Y.
\end{eqnarray*}
We infer that the equation on $U$ becomes
\be\label{eq:U-LU}
\p_s U - b U + \frac{b}{2} Y \p_Y U - \LU(\p_{YY} U -1)=0.
\ee
The whole non-linearity of the equation is now encoded in the diffusion term  $\LU(\p_{YY} U -1)$.

Setting $\cLU:= \LU \p_{YY}$,  the equation on $V=U-\Uapp$ becomes
\be\label{eq:V-2}
\p_s V - b V + \frac{b}{2} Y \p_Y V - \cLU V =\cR,
\ee
where the remainder $\cR$ is defined by
$$
\cR:=
-\left(\p_s \Uapp - b \Uapp + \frac{b}{2} Y \p_Y \Uapp\right) + \LU(\p_{YY}\Uapp -1).$$
Equation \eqref{eq:V-2} is the second form we will be using for the rescaled Prandtl equation. It will be handy for the derivation of energy estimates.

We have the following result, which is proved in section \ref{sec:proof-statements-algebraic}:
\begin{lemma}\label{lem:reste}
The remainder term $\cR$ can be decomposed as
\begin{eqnarray*}
\cR&=& \left( b_s + {b^2}\right)\chi\left(\frac{Y}{s^{2/7}}\right) \left[ a_4 Y^4 + 2 a_7bY^7 + 3 a_{10}b^2 Y^{10} + 3 a_{11} b^2 Y^{11} \right ]\\
&+& \chi\left(\frac{Y}{s^{2/7}}\right) \left[ a_{10}  b^4 Y^{10} + a_{11}  3b^4 Y^{11}/2\right]\\
&+&  \frac{b}{2}\LU (L_VY) + \frac{a_7 b^3}{2} \LU \left( \chi \left( \frac{Y}{s^{2/7}} \right)\left( L_V Y^7 + L_{-a_4 b Y^4 - a_7 b^2 Y^7 + a_{10} b^3 Y^{10} + a_{11 } b^3 Y^{11} } Y^7\right) \right)\\
&+& P_1(s,Y) + \LU (P_2(s,Y))
\end{eqnarray*}
 where $P_1, P_2 \in \mathcal C^0([s_0, \infty), \mathcal C^\infty(\R_+))$ are such that $P_i$ has at most polynomial growth in $s$ and $Y$ and $P_i(s,Y)=0 $ for $Y \leq c s^{{2/7}}$ for some $c>0$.

\end{lemma}
\begin{remark}
Following the decomposition of Lemma \ref{lem:reste}, we write $\cR=\sum_{i=1}^4 \cR_i$. Each of the remainder terms $\cR_i$ will play a different role and will be treated separately. More precisely:
\begin{enumerate}
\item Cancellations will occur in the remainder term  $\cR_1$;

\item The size of the term $\cR_2$ dictates the final rate of convergence of the energy. This is where the choice of the approximate solution plays an important role;

\item The term $\cR_3$ can be treated as perturbation of the zero order term $bV$ and of the transport term $bY\p_Y V$ as soon as $Y\gg 1$, and as a perturbation of the diffusion $\cLU V$ if $Y\ll s^{1/4}$. Indeed, think of $\LU$ as a division by $U$, and of a derivation with respect to $Y$ as a division by $Y$. Then
$$
|b \LU (L_V Y)| \lesssim b \frac{Y^2}{U}  |V_Y|\lesssim b \frac{1}{1+Y} |Y \p_Y V|.
$$
Thus if $Y\gg 1$, this term is small compared to $bY \p_Y V $. On the other hand, heuristically, $|\cLU V | \gtrsim  U^{-1} |\p_Y^2 V|\gtrsim (YU)^{-1} |\p_Y V|$ (think for instance of a Hardy inequality). Thus as long as $Y^2 U \ll b^{-1}$, i.e. $Y\ll s^{1/4}$, the diffusion term $\cLU V$ dominates $bY \p_Y V$, and therefore $\cR_3$.

\item The last term $\cR_4$ will not play any role in the energy estimates: indeed, we will choose weights with a strong polynomial decay for $Y\geq s^{\beta}$ for some $\beta<2/7$, so that the error stemming from $\cR_4$ can be made $O(s^{-P})$ for any $P>0$ by an appropriate choice of the weight.

\end{enumerate}

\label{rem:decomposition-D}

\end{remark}

The idea is now to perform weighted energy estimates on equation \eqref{eq:V-2}, with the help of a norm $\cN$ satisfying assumption \eqref{hyp:cN}. These estimates rely on the following ideas:
\begin{enumerate}
\item Let $\cN$ be a norm satisfying \eqref{hyp:cN}, and define an energy $E(s)$ by 
$$
E(s)= \cN(V(s))^2.
$$
In order to prove that $b_s+b^2=O(s^{-2-\eta})$ (or that $\int_{s_0}^{+\infty} s^{3+2\eta} (b_s+b^2)^2ds<+\infty$) for some $\eta>0$, it is enough to show that
\be\label{est:energie-generale}
\frac{dE}{ds} + \frac{\alpha}{s} E(s) \leq \rho(s)\quad \forall s\geq s_0
\ee
with $ 4+\eta\leq \alpha $, and with a right-hand side $\rho(s)$ such that $\int_{s_0}^\infty s^{\alpha} \rho(s)\:ds<+\infty$.  Indeed, integrating \eqref{est:energie-generale} between $s_0$ and $s$ yields
$$
E(s)\leq s^{-\alpha}\left( E(s_0) s_0^\alpha +\int_{s_0}^\infty s^{\alpha} \rho(s)\:ds\right).
$$
Assuming additionally that $E(s_0)\leq C_0 s_0^{-\alpha}$ for some constant $C_0$ independent of $s_0$, we are led to
$$
c_\cN|b_s+b^2|^2\leq E(s)\leq C s^{-4-\eta}\quad \forall s\geq s_0,
$$
and using Lemma \ref{lem:modulation}, we obtain the desired result.

\item The property $\alpha>4$ in \eqref{est:energie-generale} is derived thanks to algebraic manipulations on \eqref{eq:V-2}. Schematically, if we only keep the transport part in \eqref{eq:V-2} and if we consider the model equation
$$
\p_s f - \frac{1}{s} f + \frac{1}{2s} Y \p_Y f = r,
$$
we see that the $k$-th derivative of $f$ satisfies
$$
\p_s \p_Y^k f + \left(\frac{k}{2}-1\right) \frac{1}{s} \p_Y^k f + \frac{1}{2s} Y \p_Y^{k+1} f = \p_Y^k r.
$$
Hence
$$
\frac{d}{ds} \| \p_Y^k f \|_{L^2}^2 +\left(k-\frac{5}{2}\right) \frac{1}{s} \| \p_Y^k f \|_{L^2}^2= \int \p_Y^k r\;  \p_Y^k f
$$
Taking $k=7$ and $k=8$ and summing the two estimates, 
we  obtain the desired result with $\alpha=\frac{9}{2}$ and  $\cN(f):=\| \p_Y^7 f\|_{H^1}$.

\item The fact that the remainder is  integrable with a weight $s^\alpha$ in \eqref{est:energie-generale} stems from our choice of approximate solution. In particular, if we modify the algorithm of construction of $(U_N)_{N\geq 1}$ described in the previous paragraph and replace every occurrence of $b_s$ by $-cb^2$ for some constant $c\neq 1$, the estimate is no longer true.

\end{enumerate}

 Let us now explain the main steps in the derivation of estimate \eqref{est:energie-generale}. The difficulties lie  in the complex structure of the diffusion operator $\cLU$, and in the estimation of some commutator terms. The idea is to apply several times the operator $\cLU$ to equation \eqref{eq:V-2}.  This requires:
\begin{enumerate}
\item computing  the commutator of $\cLU$ with $\p_s+ \frac{b}{2} Y \p_Y$;
\item understanding the action of $\cLU$ on the remainder term $\cR$;
\item obtaining energy estimates on transport-diffusion equations of the type
\be\label{transport-diff}
\p_s f + C b f + \frac{b}{2} Y\p_Y f - \cLU f=r.
\ee
\end{enumerate}
Let us now explain how we deal with each of the points above.

\subsubsection{Commutator of \texorpdfstring{$\cLU$}{operator} with \texorpdfstring{$\p_s+ \frac{b}{2} Y \p_Y$}{transport}}

The  commutator result is stated in the following
\begin{lemma}[Computation of the commutator]
For any function $W\in W^{1,1}_{loc}((s_0,\infty)\times\R_+)$ such that $W=O(Y^2)$ for $Y$ close to zero,
\begin{eqnarray*}\left[\LU, \p_s + \frac{b}{2} Y \p_Y\right] W&=& b \LU W-\left(\cD \int_0^Y \frac{W}{U^2}\right)_Y + 2 \left( U \int_0^Y \frac{W }{U^3} \cD\right)_Y,
\end{eqnarray*}
where $\cD:=L_U^{-1}(\p_{YY} U -1)$.

\label{lem:commutator}
\end{lemma}
In the rest of the article, we define the commutator operator
$$
\mathcal C[W]:=-\left(\cD \int_0^Y \frac{W}{U^2}\right)_Y + 2 \left( U \int_0^Y \frac{W }{U^3} \cD\right)_Y.
$$

The quantity $\cD$ involved in the commutator can then be written as
$$
\cD= \cLU V + \LU (\p_{YY}\Uapp-1),
$$
where the second term can be developed using the explicit expression $\Uapp$.
Notice the commutator $\cC[\p_Y^2 V]$ contains some quadratic  terms in $V$. In order to estimate these quadratic terms, we will need both preliminary estimates and estimates in $L^\infty$ on the function $V$. The $L^\infty$ estimates are derived in detail in section \ref{sec:infty}, and rely on a careful use of the maximum principle. 
 
 \subsubsection{Action of \texorpdfstring{$\cLU$}{LU} on the remainder term \texorpdfstring{$\cR$}{R}}

  We will use the decomposition of the remainder given in Remark \ref{rem:decomposition-D}. The first two terms, namely $\cR_1$ and $\cR_2$, are essentially polynomials. Therefore, in order to deal with them, we will need to get explicit formulas for terms such as $\cLU (Y^4)=12\LU (Y^2)$, and more generally, to understand the asymptotic behavior of $\LU (Y^k)$ for $Y\gg 1$ and $k\geq 2$. 
  
  In order to get explicit formulas, we will use in several instances the following trick: for $k\in \N$, write
  $$
  Y^k= \LU L_U (Y^k)= \LU (L_{\Uapp} (Y^k) + L_V Y^k).
 $$
 Now, since $\Uapp$ is a polynomial, $L_{\Uapp} (Y^k)$ can be easily computed, and is also a polynomial in $Y$.   The term $\LU L_V (Y^k)$ is expected to be of lower order. For instance, taking $k=1$, we observe that $L_{\Uapp} (Y)= \frac{Y^2}{2} + O(bY^5)$ for $Y\ll s^{2/7}$. Hence we obtain a formula for $\LU(Y^2)$ (up to some remainder terms).

 Concerning the asymptotic behavior of $\LU (Y^k)$ for $Y\gg 1$ and for $k\geq 4$,  notice  that the operator $\LU$ acts roughly like a division by $U$, as can be seen from the formula \eqref{LU}. 
Furthermore, the $L^\infty$ estimates (see Proposition \ref{prop:max-princ}) will ensure that there exists a constant $C$ such that
 $$
C^{-1}(Y+Y^2)\leq  U(s,Y) \leq C(Y+Y^2) \quad \forall Y\lesssim s^{2/7}.
 $$
   Therefore, $\LU$ behaves differently for $Y\ll 1$ and for $Y\gg 1$: for $Y$ close to zero, applying $\LU$ amounts to dividing by $Y$, while for $Y\gg 1$, it amounts to dividing by $Y^2$. We obtain that for $Y\ll s^{1/2}$ and $k\geq 4$,
   $$
   \LU (Y^k)=\left\{
   \begin{array}{ll}
 O(Y^{k-1})& \text{ if } Y \ll 1,\\
 O(Y^{k-2})& \text{ if } Y\gtrsim 1.
   \end{array}
   \right.
   $$
 
As explained in Remark \ref{rem:decomposition-D}, the term $\cR_3$ is treated as a perturbation of the dissipation coming from the transport and the diffusion term. Eventually, since $\cR_4$ is supported in $Y\gtrsim s^{2/7}$, while we use weights that have a strong polynomial decay for $Y\gtrsim s^{\beta}$ for some $\beta<2/7$, the size of $\cR_4$ in our energy norms will be smaller that that of $\cR_1+ \cR_2 + \cR_3$.

\subsubsection{Energy estimates on transport-diffusion equations of the type \texorpdfstring{\eqref{transport-diff}}{}}

The most difficult part is proving coercivity and positivity estimates for the diffusion. We will rely on the diffusion Lemma  \ref{lem:diff1}, which makes an extensive use of weighted Hardy inequalities, see \cite{masmoudi2011hardy}. They also rely on the fact that if $U= Y + \frac{Y^2}{2} + O(bY^4)$, then $\p_{YY} \LU$ is ``almost'' a local differential operator (see the formulas in Lemma \ref{lem:p3LU} in the Appendix).

We will also often use the observation that if $Y\lesssim s^{1/4}$, then the diffusion term dominates, while for $Y\gtrsim s^{1/4}$, then the transport part becomes preponderant. Indeed, for $k\geq 6$, if $b\sim s^{-1}$
$$
\ba
C b Y^k + \frac{b}{2} Y \p_Y Y^k\sim \left(\frac{k}{2} + C\right) \frac{1}{s} Y^k,\\
\text{and } \cLU Y^k= O(Y^{k-4})\quad \text{for } Y\gg 1.
\ea
$$
It is easily checked that both terms are of the same order for $Y\sim s^{1/4}$, and diffusion (resp. transport) is dominant below (resp. above) that threshold.

\vskip3mm

We now turn towards the sequence of estimates on $V$. In the end, we seek to obtain estimates on $\p_Y \cLU^2 V$ and $\p_Y^2 \cLU^2 V$. We recall that $V\sim -C(b_s+b^2) Y^7$ for $Y$ close to zero, and therefore $\p_Y \cLU^2 V \sim -C(b_s+b^2)$ for $Y$ close to zero.

Using Lemma \ref{lem:commutator}, we infer that $\cLU V$ satisfies the following equation:

\be\label{eq:cLUV}
\p_s \cLU V + b \cLU V + \frac{b}{2} Y \p_Y \cLU V - \cLU^2 V = \cLU \mathcal R + \mathcal C[\p_{YY} V].
\ee
We get the following result:
\begin{proposition}\label{prop:est-cLUV-1}
Assume that:
\begin{itemize}
\item There exists a constant $J>0$ such that
$$
\int_{s_0}^{s_1} s^{13/4} |b_s+b^2|^2 ds \leq J;
$$
\item $\frac{1-\bar \eps}{s} \leq b \leq \frac{1+\bar \eps}{s}$ for $s\in [s_0, s_1]$ and for some small universal constant $\bar \eps$ (say $\bar \eps=1/50$);
\item There exist  constants  $M_1,M_2,c$ independent of $s$ such that for all $Y$,
$$\ba
-M_1\leq U_{YY}\leq 1\quad\forall Y \geq 0,\\
1- M_2 b Y^2 \leq  U_{YY}\quad \forall Y \in [0, c s^{1/3}].\ea
$$

\end{itemize}
Let $w_1:=Y^{-a}(1+ s^{-\beta_1} Y)^{-m_1}$ for some $\beta_1\in ]1/4, 2/7[$, $m_1\in \N$.

Let
$$\ba
E_1(s):=\int_0^\infty (\p_Y^2 \cLU V )^2  w_1 ,\\
D_1(s):=\int_0^\infty \frac{(\p_Y^3 \cLU V)^2} U w_1 + \int_0^\infty \frac{(\p_Y^2 \cLU V)^2}{U^2} w_1 .\ea
$$
Then there exist universal constants $\bar a, \bar c>0$, such that for all $a\in ]0, \bar a[$, for all  $\alpha<6-(11-a)\beta_1$, for $m_1$ large enough, there exists $S_0, H_1>0$ depending on $M_1, M_2, c, \beta_1, m_1$, and $a$
such that if $s_0\geq \max (S_0, J^4)$,
$$
E_1(s ) \leq H_1 (1 + E_1(s_0) s_0^\alpha ) s^{-\alpha},\  \int_{s_0}^{s_1} s^\alpha D_1(s) ds\leq   H_1(1 + E_1(s_0) s_0^\alpha )\quad \forall s\in [s_0, s_1].
$$

\end{proposition}
\begin{remark}
The weight $Y^{-a}$ in $w_1$ has two different roles. On the one hand, we gain  a bit of decay in the remainder terms. On the other hand, we are able to control, through a simple Cauchy-Schwarz inequality, quantities of the type
$$
\int_0^Y \frac{\p_Y^2 \cLU V}{U^2}
$$
by the diffusion term.
\end{remark}

\begin{remark}
Notice that if we take $\beta_1$ such that
$$
\frac{1}{4}< \beta_1 < \frac{1}{4} \frac{11}{11-a},
$$
then we can choose $\alpha$ so that $\alpha>13/4$. We will make this choice in the final energy estimates, and we will use the corresponding decay of $E_1$ when we apply the maximum principle.
\label{rem:E_1-13/4}
\end{remark}

Differentiating \eqref{eq:cLUV} with respect to $Y$ and taking the trace of the equation at $Y=0$, we obtain in particular
\begin{lemma}
\label{lem:trace-cLU2V}
For all $s\ge s_0$, we have
$$
\p_Y \cLU^2 V_{|Y=0}= -\frac{1}{2}(b_s+b^2).
$$

\end{lemma}

We then derive an equation on $\cLU^2 V$. Applying once more the commutator result of Lemma \ref{lem:commutator}, we deduce that $\cLU^2 V$ satisfies the following equation:
\be
\label{eq:cLU2V}
\p_s \cLU^2 V + 3 b \cLU^2 V + \frac{b}{2} Y \p_Y \cLU^2 V - \cLU^3 V =\cLU^2 \mathcal R + \mathcal C[\p_Y^2 \cLU V] + \cLU \mathcal C[\p_Y^2 V].
\ee
In order to have a trace estimate, and also to have nice positivity properties for the diffusion term, we will also need to use estimates on $\p_Y^2 \cLU^2 V$. 
We therefore define the energy
$$
\ba
E_2(s):=\int_0^\infty (\p_Y^2 \cLU^2 V)^2 w_2,
 \ea
$$
together with the dissipation terms
$$
\ba
D_2(s):=\int_0^\infty \frac{(\p_Y^3 \cLU^2 V)^2}{U}  w_2 + \int_0^\infty\frac{(\p_Y^2 \cLU^2 V)^2 }{U^2}  w_2,
\ea
$$
where the weight $w_2$ is defined by
$ w_2=Y^{-a}(1+  s^{-\beta_2} Y)^{-m_2}
$
for some parameters  $ \beta_2\in ]1/4,\beta_1[$, $m_2\gg m_1$  sufficiently large. The parameter $a$ is the same as the one in Proposition \ref{prop:est-cLUV-1}.

We then claim that we have the following estimate:
\begin{proposition}
Assume that the hypotheses of Proposition \ref{prop:est-cLUV-1} are satisfied. 
Let $C_0:=\max(E_1(s_0) s_0^{13/4+\eta}, E_2(s_0) s_0^5)$ for some $\eta>0$ such that $13/4 + \eta<6-(11-a)\beta_1$.
Then there exists a universal constant $\bar a$, such that for all $a\in ]0, \bar a[$, for a suitable choice of $\beta_2, m_2$, there exist $S_0, H_2>0$ (depending on $a, \beta_i, m_i, M_1, M_2$) 
such that if $s_0\geq \max (S_0, J^4, C_0^8)$ ,
\be\label{est:E_3}
E_2(s)\leq H_2 (1+C_0)\exp(H_2 (1+C_0)) s^{-5}\quad \forall s\in [s_0, s_1].
\ee

\label{prop:est-cLU2V}
\end{proposition}

Let us now go back to the definition of the semi-norm $\cN$. We need a new type of trace estimate, taking advantage of the fact that $E_2$ has a stronger decay than $E_1$ (notice that the sole decay of $E_1$ is not sufficient to close the bootstrap argument, since $E_1\lesssim s^{-13/4-\eta}$ and we need $|b_s + b^2| \lesssim s^{-2-\eta/2}$, while $13/4<4$).

We will use the following trace estimate, which is proved in Appendix:
\begin{lemma} There exists a universal constant $\bar C$, such that for all $L\geq 1$, for any smooth function $f$,
	$$
	|f(0)|^2\leq \bar C \left( L^{1+a} \int_0^L (\p_Y f)^2 Y^{-a}\: dY+ \frac{1}{L^{3-a} }\int_0^L |f(Y)|^2 (Y+Y^2) Y^{-a}\: dY\right).
	$$
	In particular, taking $f=\p_Y \cLU^2 V$ and $L=s^{1/4}$, under the assumptions of Proposition \ref{prop:est-cLU2V},
	$$
	|b_s+b^2|^2\leq \bar C \left( s^{\frac{1+a}{4}} E_2 + s^{-\frac{3-a}{4}}\int_0^{s^{1/4} }U \left(\p_Y \cLU^2 V\right)^2\right) .
	$$
	\label{lem:trace}
\end{lemma}

Let us now go back to the definition of the semi-norm $\cN$. According to the above Lemma, we can take for instance
$$
\cN(V):= \left(s^{\frac{1+a}{4}} E_2 + s^{-\frac{3-a}{4}}\int_0^\infty U \left(\p_Y \cLU^2 V\right)^2 \tilde w_1\right)^{1/2},
$$
where $\tilde w_1:= Y^{-a} (1+ s^{-\beta_1} Y)^{-m_1-2}= w_1  (1+ s^{-\beta_1} Y)^{-2}$, so that $
\cN(V) \geq \bar C |b_s+ b^2|
$
for some universal constant $\bar C$.

However there remains to prove that with this definition, $\cN(V)$ is sufficiently small.
According to the above Lemma, we also need to find a bound for $\int_0^\infty U \left(\p_Y \cLU^2 V\right)^2 \tilde w_1$. We claim that we have the following estimate:
\begin{lemma} Assume that the assumptions of Proposition \ref{prop:est-cLUV-1} are satisfied. 
Let $P>0$ be arbitrary. 
There exist $S_0, H_1>0$ depending on $M_1, M_2, \beta_1, m_1$, and $a$, and a function $\rho$ such that $\int_{s_0}^{s_1} \rho(s) s^{-1/2} ds\leq 1$, such that if $m_1$ is large enough (depending on $P$)
and  if $s_0\geq \max (S_0, J^4)$,
\be\label{in:coerciv}
\int_0^\infty U \left(\p_Y \cLU^2 V\right)^2 \tilde w_1\leq H_1 s^{(3-a)\beta_1} D_1 + s^{-1} E_1 + s^{-P} \rho(s).
\ee
\label{lem:coerciv}
\end{lemma}

Gathering Lemmas \ref{lem:trace} and \ref{lem:coerciv} and using Proposition \ref{prop:est-cLUV-1} and Proposition \ref{prop:est-cLU2V}, we find that if $s_0\geq \max (S_0, J^4, C_0^8)$,
$$
|b_s + b^2|^2 \leq H_2(1+C_0)  \exp(H_2(1+C_0))s^{-\frac{19-a}{4}}+ H_1  s^{(3-a)(\beta_1 - \frac{1}{4})}  D_1 + s^{-P} \rho(s).
$$
Recall that $\int_{s_0}^{s_1} s^\alpha  D_1(s)\:ds \leq H_1(1+C_0)$ for $s_0\geq \max(S_0, J^4)$ and for $\alpha=13/4 + \eta < 6-(11-a) \beta_1$. A short computation\footnote{It is enough to choose 
	$$
	\beta_1< \frac{1}{4}\left(1 + \frac{a}{14-2a}\right).
	$$} shows that we can choose $\beta_1$ and $a$ so that 
$$
\frac{13}{4} + (3-a)\left(\beta_1 - \frac{1}{4}\right) <6-(11-a) \beta_1. 
$$
We obtain eventually that
$$
|b_s + b^2| \leq \bar C \cN(V),
$$
and
\[
\int_{s_0}^{s_1} s^{13/4} \cN(V(s))^2 \:ds\leq H(1+C_0)\exp(H(1+C_0)), 
\]
for some constant $H$ depending on  $a, \beta_i, m_i, M_1, M_2$, provided $s_0\geq \max (S_0, J^4, C_0^8)$.
Thus $b$ satisfies the assumptions of Lemma \ref{lem:b} with $\gamma=13/4>3$.

We gather the estimates of Propositions \ref{prop:est-cLUV-1} and \ref{prop:est-cLU2V} and Lemma \ref{lem:trace} in the following Theorem:
\begin{theorem}
\label{thm:estimates-gal}
Let $a\in ]0,\bar a[$, and choose the parameters $\beta_1$ and $\beta_2$ such that 
\be\label{def:beta}\frac{1}{4}<\beta_2<\beta_1<\frac{1}{4}\inf\left(\frac{11}{11-a},\frac{14-a}{14-2a}\right)\ee
and $m_2\gg m_1\gg 1$.

Let $\eta=\eta(a,\beta_1)$ such that $0<\eta< (3-a) (\beta_1-1/4)$.

Assume that the following assumptions are satisfied:
\begin{itemize}
 \item There exists a constant $J>0$ such that
$$
\int_{s_0}^{s_1} s^{13/4} |b_s+b^2|^2 ds \leq J;
$$
\item $(1-\bar \eps)/s \leq b \leq (1+\bar\eps)/s$ for $s\in [s_0, s_1]$ for some small enough constant $\bar \eps$ (say $\bar \eps=1/50$);
\item There exist  constants $M_1, M_2,c$ independent of $s$, such that
$$\ba
-M_1\leq U_{YY}\leq 1\quad\forall Y \geq 0,\\
1- M_2 b Y^2 \leq  U_{YY}\quad \forall Y \in [0, c s^{1/3}].\ea
$$
\item There exists a constant $C_0$, independent of $s_0$ and $\lambda_0$, such that for some $\eta>0$
$$
  E_2(s_0)s_0^{5} + E_1(s_0) s_0^{\eta+13/4}\leq C_0.
$$
\end{itemize}

Then there exists   constants $H, S_0$, depending on $a, m_1, m_2, \beta_1, \beta_2, M_1$ and $M_2$, such that for all $ s_0\geq \max(S_0, J^4, C_0^8)$,
$$
\int_{s_0}^{s_1} s^{13/4}|b_s+b^2|^2\leq \exp(H(1+C_0)).
$$
In particular, setting $J':=\exp(H(1+C_0))$, we have
$$
s_0\geq \max(S_0, J^4, C_0^8) \Rightarrow \int_{s_0}^{s_1} s^{13/4}|b_s+b^2|^2\leq J',
$$
and the constant $J'$ is independent of $J$.

\label{thm:mod-rate}
\end{theorem}

\subsection{Construction of sub- and super-solutions}
The other ingredient in the proof of Theorem \ref{thm:main} is the use of the maximum principle in order to control the growth of $U$ and its derivatives on the one hand, and the size of some non-linear terms on the other hand. Indeed, the assumptions of Propositions \ref{prop:est-cLUV-1} and \ref{prop:est-cLU2V} require estimates on $\p_Y^2 U$. These estimates are obtained by careful applications of a comparison principle. We emphasize that this principle is not applied to equation \eqref{rescaled} directly, but rather to an equation derived from \eqref{rescaled} after a non-linear change of variables. More precisely, we use the von Mises variables
\be\label{VM}
\psi:=\int_0^Y U,\quad W:= U^2.
\ee
The tangential variable remains $s$, the normal variable is now $\psi$ (instead of $Y$), and the new unknown function is $W$. This change of variables transforms \eqref{rescaled} into a non-linear transport-diffusion equation which is more 
suited for maximum principle techniques, namely
\be\label{eq:vM}\p_s W -2 b W + \frac{3b}{2}\psi \p_\psi W - \sqrt{W}\p_\psi^2 W = -2\ee
Equation \eqref{eq:vM} is the third and last form of equation \eqref{rescaled}.
 We refer to \cite{OS} and to section \ref{sec:infty} of the present paper for more details. Since the new equation is local and  parabolic, it enjoys maximum principle properties. Therefore we construct sub- and super-solutions for $W$ and its derivatives and thereby derive estimates on $W$. These estimates are then translated in terms of the former variables $s,Y$. 
 
 One of the key points lies in the construction of a sub-solution for $W$ (see Lemma \ref{lem:sub-U}). Actually, the Sobolev estimates of Proposition \ref{prop:est-cLUV-1} provide a very good pointwise control of $U$ up to $Y\sim s^{\beta_1}\gg s^{1/4}$, but this control degenerates for $Y\gtrsim s^{\beta_1}$. Hence it is sufficient to construct sub-solutions for $Y\gtrsim s^{1/4}$, or equivalently (since $\psi\sim Y^3/6$ for $Y\gg 1$) for $\psi \gtrsim s^{3/4}$. On this zone, the sub-solutions will be linear combinations of powers of $\psi$.
Furthermore, we define a \textit{regularized modulation rate $\tb$}, whose role is to remove some oscillations from $b$ while keeping the same asymptotic behavior. We take
 $$
  \tb_s + b \tb=0,\quad \tb_{|s=s_0}= \frac{1}{s_0}.
 $$
 The sub-solutions for $W$ are defined by
 $$
 \uW:=\frac{(6\psi)^{4/3}}{4} - A \psi^{7/3} \tb^{5/4},
 $$
 where $A$ is chosen sufficiently large. Notice that the main order term $(6\psi)^{4/3}/4$ is the same as in the original solution. It corresponds to the main order term $Y^2/2$ in $U(s,Y)$. 
 
 The regularized modulation rate $\tb$ is also used in the construction of a sub-solution for $U_{YY}-1$ (see Lemma \ref{lem:U_YY-final}).

The final result is the following (we refer to Remark \ref{rem:E_1-13/4} regarding the assumption on $E_1$):

 \begin{proposition}\label{prop:max-princ}
Assume that there exist constants $J>0$, $\eta>0$ and $\eps>0$ such  that for all  $s\in [s_0, s_1]$, assumption \eqref{hyp:b-mod-rate} is satisfied. 
Assume furthermore that there exists a constant $M_0$ such that
$$\ba
- M_0 \inf (1,  s_0^{-1} Y^2) \leq U_{YY}(s_0, Y)-1\leq 0\quad \forall Y>0,\\
\lim_{Y\to \infty} U(s_0, Y)\leq M_0 s_0,
\ea
$$
and that there exists $C_1>0$ such that
$$
E_1(s)\leq C_1 s^{-13/4}\quad \forall s\in [s_0, s_1].
$$
Then there exist universal constants $\bar M, \bar C>0$, and $S_0$ depending on $C_1, \beta_1, m_1$, such that if
$s_0\geq \max (S_0, \bar C (J\eps^{-2})^{1/2\eta})$, then, setting $M'= \bar M \max(1, M_0)$,
$$\ba
- M' b Y^2 \leq  U_{YY}(s, Y)-1\leq 0 \quad \forall s\in [s_0, s_1],\ \forall Y \in [0, s^{1/3}],\\
\text{and } -M'\leq U_{YY}(s, Y)-1\leq 0\quad \forall s\in [s_0, s_1],\ \forall Y \geq s^{1/3}.
\ea
$$

 \end{proposition}

Notice that the above estimates are precisely the ones that are required in Proposition \ref{prop:est-cLUV-1}.

\subsection{Bootstrap argument}
The bootstrap argument consists in bringing together Theorem \ref{thm:mod-rate} on the one hand, and  Proposition \ref{prop:max-princ} on the other. In the rest of this section, we will assume that $U(s_0)$ satisfies
\be\label{hyp:well-prep}\ba
E_1(s_0)\leq C_0 s_0^{-13/4-\eta/2}, \quad E_2(s_0)\leq C_0 s_0^{-5} ,\\
 \left|b(s_0)-\frac{1}{s_0}\right| \leq \frac{\bar \eps}{2 s_0},\\
\text{ and }- M_0 \inf (1,  s_0^{-1} Y^2) \leq U_{YY}(s_0, Y)-1\leq 0\quad \forall Y>0, \quad \lim_{Y\to \infty} U(s_0, Y)\leq M_0 s_0.
\ea \ee
where $C_0, M_0$ are constants independent of $s_0$, and $\eta$ is such that $\eta-3/4<2-(11-a)\beta_1$. Without loss of generality, we also assume that $M_0\geq 1$.
Such an initial data is ``well-prepared" in the sense that it is close to the blow-up profile.

Assumption \eqref{hyp:well-prep} involves three different types of estimates. In order to propagate these estimates, we will apply three different results:
\begin{enumerate}
\item The energy estimates from Propositions \ref{prop:est-cLUV-1} and \ref{prop:est-cLU2V}, which are gathered in Theorem~\ref{thm:mod-rate};
\item The maximum principle estimates from Proposition \ref{prop:max-princ};
\item Lemma  \ref{lem:b} on the modulation rate $b$.
\end{enumerate}
Note that the maximum principle will propagate the third estimate of \eqref{hyp:well-prep} without improving it (in fact, we will change the constant $M_0$ into $\bar M M_0$); however the energy estimates will transform $\eta/2$ into $\eta$, and will therefore improve the estimates on $E_1$.

The argument goes as follows: let $S_0, H$ be the constants from Theorem \ref{thm:mod-rate} with $\alpha=13/4 + \eta$ and $M_1=M_2=\bar M M_0$ (recall that $S_0$ depends in particular on $\beta_1,\beta_2, \eta$   and $a$). 
Let $J:=2 \exp (H(1+C_0))$.
Assume that $s_0\geq \max (S_0, \bar CJ^4, C_0^8)$ for the large universal constant $\bar C$ from Proposition \ref{prop:max-princ}.

If the initial data satisfies \eqref{hyp:well-prep}, by continuity,  there exists $s_1> s_0$ such that for all $s\in [s_0, s_1]$,
\be\label{bootstrap1}
E_1(s)\leq 2H_1(1+C_0)s^{-13/4-\frac{\eta}{2}}, \quad \left|b(s)-\frac{1}{s}\right| \leq \frac{\bar \eps}{s},\quad \int_{s_0}^{s_1}|b_s + b^2|^2 s^{13/4}\:ds\leq J.
\ee

Then for all $s\in [s_0, s_1]$, according to Proposition \ref{prop:max-princ} with $C_1:= 2H_1(1+C_0)$, we infer that  up to choosing a larger $S_0$ (depending on $H_1$ and $C_0$), 
$$\ba
-\bar M M_0\leq U_{YY}\leq 1\quad\forall Y \geq 0,\\
1- \bar M M_0 b Y^2 \leq  U_{YY}\quad \forall Y \in [0,  s^{1/3}].\ea
$$

The assumptions of Propositions \ref{prop:est-cLUV-1} and \ref{prop:est-cLU2V} are satisfied (with $M_1=M_2=\bar M M_0$), and we infer that if $s_0\geq \max(S_0, \bar C J^4, C_0^8)$, for all $s\in [s_0, s_1]$,
$$\ba
E_1(s)\leq H_1(1 + C_0 s_0^{\eta/2}) s^{-13/4 -\eta},\quad E_2(s)\leq H_2(1+C_0) \exp(H_2(1+C_0)) s^{-5}.
\ea
$$
We have in particular for all $s\in [s_0, s_1]$
$$
E_1(s)\leq H_1(1+ C_0) s^{-13/4-\eta/2},
$$
and using Lemmas \ref{lem:trace} and \ref{lem:coerciv},
\[
\int_{s_0}^{s_1} |b_s+b^2|^2 s^{13/4}ds \leq \exp(H(1+C_0))= \frac{J}{2}.
\]

Using Lemma \ref{lem:b} in the Appendix, we infer that for all $s\in [s_0, s_1]$,
$$
\left|b(s)-\frac{1}{s}\right| \leq  \frac{1+\bar\eps}{1-\bar \eps}\left| \frac{1}{s_0} - b(s_0)\right| \frac{s_0^2}{s^2} +  \frac{1+\bar\eps}{(1-\bar\eps)^2} s^{-9/8}\sqrt{\frac{2J}{7}}.
$$
Without loss of generality, we can always assume that $\bar \eps, \eta, s_0$ are such that
$$
 \frac{1+\bar\eps}{1-\bar \eps}\leq \frac{5}{4},\quad  \frac{1+\bar\eps}{(1-\bar\eps)^2} \sqrt{\frac{2J}{7}} s_0^{-1/8}\leq \frac{\bar \eps}{4}.
$$
Then for all $s\in [s_0, s_1]$, we have
\be\label{bootstrap2}
\left|b(s)-\frac{1}{s}\right| \leq \frac{7\bar\eps}{8s}.
\ee

Gathering \eqref{bootstrap1} and \eqref{bootstrap2}, we infer that
\begin{multline*}
s_1:=\inf\{s\geq s_0,\ E_1(s)=2H_1(1+C_0) s^{-13/4-\eta/2}\text{ or } |b(s)-1/s| = \bar\eps/s\\\text{ or } \int_{s_0}^s \tau^{13/4} |b_\tau+b^2|^2d\tau= J\}=+\infty.
\end{multline*}

As a consequence, we infer that for some constant $J$ depending only on $C_0$ and $M_0$, if $s_0\geq \max (S_0, \bar C J^4, C_0^8)$,
$$
\int_{s_0}^\infty s^{13/4} |b_s + b^2|^2 ds\leq J,\quad \left|b-\frac{1}{s}\right|\leq \frac{\bar \eps}{s}.
$$

We therefore obtain the following Theorem in the rescaled variables $(s,Y)$:

\begin{theorem}
Let $\eta$ be such that $0<\eta<(3-a)(\beta_1-1/4)$, where $\beta_1$ satisfies \eqref{def:beta}.
Assume that $U_0$ satisfies the hypotheses \eqref{hyp:well-prep}, and consider the solution $U$ of equation \eqref{rescaled} with $U(s_0, Y)=U_0$. Then there exists a constant $S_0>0$, depending on $\eta, C_0, M_0$, such that if $s_0\geq S_0$, then for all $s\geq s_0$,
\be\label{conclusion-b}
\int_{s_0}^\infty s^{13/4} |b_s + b^2|^2 ds<+\infty,\quad  \left|b-\frac{1}{s}\right|\leq \frac{\bar \eps}{s}.
\ee

\label{thm:rescaled}
\end{theorem}

Let us now go back to the original variables and prove Theorem \ref{thm:main}. First, we set
$$\lambda_0:=\p_y u_{0|y=0},\  b_0:= -\lambda_0^2\p_y^4 u_{0|y=0},\ s_0:=b_0^{-1}.
$$
The assumption (H1) entails that
$$
c_0^{-1} \lambda_0^2\leq s_0^{-1}\leq c_0 \lambda_0^2.
$$
Assumption (H2) implies
$$\ba- M_0 \inf (1,  s_0^{-1} Y^2) \leq U_{YY}(s_0, Y)-1\leq 0\quad \forall Y>0,\\
U(s_0,Y)\leq M_0 \lambda_0^{-2}\leq M_0c_0 s_0\ea
$$
for a suitable constant $M_0$. Of course, without loss of generality, we can assume that $M_0 \geq 1$.
Furthermore, assumption (H3) becomes, in the rescaled variables and after a few easy computations,
$$
V(s_0, Y)= O(s_0^{-9/4}) (Y^7 + c_8 Y^8) + O(s_0^{-3}(Y^{10}+ Y^{11}))
$$
for $Y\leq s_0^{2/7}$. Here the constant $c_8$ is defined so that
$$
\p_Y\cLU^2 V(s_0)= O(s_0^{-9/4}) + O(s_0^{-3}Y^2).
$$
It can be easily checked that these assumptions ensure that $U$ is a well-prepared initial data. As a consequence, if $s_0$ is large enough (i.e. if $\lambda_0$ is small enough), \eqref{conclusion-b} holds. Using Lemma \ref{lem:modulation}, we infer that $x^*<+\infty$ and that $\lambda(x)\sim C \sqrt{x^*-x}$. Theorem \ref{thm:main} follows.

Furthermore, we deduce from the maximum principle estimates some pointwise control on $u$. Indeed, we have
\[
Y + \frac{Y^2}{2} - \frac{M_2}{12} b Y^4\leq U(s,Y)\leq Y + \frac{Y^2}{2}\quad \text{for }0\leq Y \leq c s^{1/3}. 
\]
Going back to the original variables, we find that there exist constants $C,c$ such that
\[
\lambda(x)y + \frac{y^2}{2} - C y^4\leq  u(x,y)\leq \lambda(x)y + \frac{y^2}{2}\quad \forall y\leq (x^*-x)^{1/6}.
\]
 
\begin{remark}[Comparison with the result by Caffarelli and E]
	Let us now plug the change of variables in the result announced by Caffarelli and E  in \cite{E-prandtl} into the asymptotic expansion above. 
	We recall that
	$$
	u^\mu(\xi, z)= \frac{1}{\mu^{1/2}} u (x^*-\xi\mu, \mu^{1/4}z).
	$$
	It follows that in the zone  $z\leq  \mu^{-1/12} \xi^{1/6}$, $\xi\lesssim 1$,
\[
	u^\mu(\xi,z)=O (\mu^{1/4} \sqrt{\xi} )z + \frac{z^2}{2}+ O(\mu^{1/2} z^4 )\to \frac{z^2}{2}\quad \text{as } \mu\to 0.
\]

	\label{rem:ECaf}
	
\end{remark}

\subsection{Organization of the rest of the paper}
The rest of the paper is dedicated to the proof of Theorem \ref{thm:rescaled}, or more specifically, to the proofs of Proposition \ref{prop:est-cLUV-1}, Proposition \ref{prop:est-cLU2V}
 and Proposition \ref{prop:max-princ}. Since the maximum principle estimates are easier to derive than the energy estimates, we start with the proof of Proposition \ref{prop:max-princ} in section \ref{sec:infty}. We then lay the ground for the derivation of the energy estimates by proving several important intermediate results in section \ref{sec:proof-statements-algebraic}. Eventually, we prove Proposition  \ref{prop:est-cLUV-1} and Proposition \ref{prop:est-cLU2V} in section \ref{sec:proof-statements-energy}.

Let us also explain here the order in which the parameters are chosen. We first pick $a\in (0, \bar a)$, where $\bar a$ is the universal constant in Proposition \ref{prop:est-cLUV-1}. We then choose $\beta_1> \beta_2$ satisfying \eqref{def:beta}, and $\eta>0$ such that $
\eta-\frac{3}{4} < 2 - (11-a) \beta_1.$
We then pick $m_1,m_2$ large enough and such that $m_2\gg m_1$. Eventually, we take $s_0$ large, depending on all other parameters.

\subsubsection*{Notation} We will use indifferently $f_{Y}$ and $\p_Y f$ to denote the $Y$ derivative of an arbitrary function $f$. 
The constants with a bar ($\bar a, \bar C, \bar M, \bar \eps$) denote universal constants, that do not depend on any of the parameters. All constants with a zero subscript ($M_0$, $C_0$, $s_0$) refer to the initial data. Constants involving the letter $M$ ($M_0$, $M_1$, $M_2$, $\bar M$) are related to the maximum principle.

\section{Derivation of \texorpdfstring{$L^\infty$}{sup-norm} estimates and construction of sub and super solutions}

\label{sec:infty}
This section is devoted to the proof of Proposition \ref{prop:max-princ}, which consists in the derivation of pointwise estimates on $U, U_{Y}$ and $U_{YY}$, provided $b$ satisfies the assumptions of Lemma \ref{lem:b} and $E_1(s)= O(s^{-13/4})$. 
Throughout this section, we will use the von Mises formulation of the rescaled Prandtl equation, namely \eqref{VM}-\eqref{eq:vM}.
The idea is to use the maximum and comparison principles for these equations  (see Lemmas 2.1.3 and 2.1.4 in \cite{OS}), together with Sobolev estimates coming from the bound on $E_1(s)$.

Let us first recall  some useful formulas regarding the von Mises formulation of the equation in the original variables and in the rescaled variables. If $u$ is the solution of \eqref{Prandtl}, we set
$$
\phi(x,y)=\int_0^y u,\quad w=u^2.
$$
We recall that \eqref{Prandtl} is equivalent to the following equation, written in the variables $(x,\phi)$
\be\label{vM-original}
w_x -\sqrt{w} \p_{\phi}^2 w=-2.\ee
Furthermore, notice that if $W,\psi$ are defined by \eqref{VM}
$$
\psi(s,Y)=\lambda(x(s))^{-3} \phi(x(s), \lambda(x(s)) Y),\quad W(s,\psi)=\lambda(x(s))^{-4} w(x(s), \lambda^3 \psi).
$$
It follows that some qualitative properties of equation \eqref{eq:vM} (growth with respect to $\phi$, local bounds) can be inherited directly from  equation  \eqref{vM-original}. More precisely, we have the following result:
\begin{lemma}
Let $w_0\in \mathcal C^{3,\alpha} (\R_+)$ such that $w_0(0)=0$ and $w_0'(0)>0$. Assume  that $w_0$ is increasing. Then for all $x\in [0, x^*[$, $w(x,\phi)$ is increasing with respect to $\phi$. Furthermore, for any $X\in [0, x^*[$, there exists $C_X>0$ such that for all $x\in [0,X]$,
$$\ba
|w_\phi(x,\phi)|\leq C_X \quad\forall \phi\geq 0,\\
|\p_\phi^2 w(x,\phi)|, \ |\p_\phi^3 w(x,\phi)|\leq C_X \quad \forall \phi\geq 1.
\ea$$

As a consequence, $W$ is increasing in $\psi$ (or equivalently, $U$ is increasing in $Y$) for all $s\in [s_0, s_1]$, and 
$$
\lim_{\psi \to \infty} W_{\psi}= \lim_{\psi \to \infty} W_{\psi\psi}=0\quad \forall s\in [s_0, s_1].
$$

\label{lem:preliminary-bounds}
\end{lemma}

\begin{proof}
The bounds on $w_\phi$, $w_{\phi \phi}$ are explicitly written in \cite{OS} (see Lemmas 2.1.9 and 2.1.11). The bound on $\p_\phi^3 w$ follows from the same arguments as \cite[Lemma 2.1.11]{OS}, writing down the equation on $w_\phi$. Since $\lim_{\phi\to \infty} w(x,\phi)= \bar U(x)^2$, it follows that $\lim_{\phi\to \infty} w_\phi=0$ for all $x\in (0, x^*)$. Therefore we also have $\lim_{\phi\to \infty} w_{\phi\phi}=0$.
Whence $\lim_{\psi \to \infty} W_\psi = \lim_{\psi \to \infty} W_{\psi\psi}=0$.

Furthermore, the equation satisfied by $w_\phi$ is
$$
\p_x w_\phi - \frac{w_\phi}{2\sqrt{w}} \p_\phi w_\phi- \sqrt{w} \p_{\phi \phi} w_\phi=0,
$$
with boundary conditions $w_{\phi|x=0}= w_0'(\phi)\geq 0$, $\lim_{\phi\to \infty} w_\phi=0$, and $w_{\phi|\phi=0}= 2\lambda(x)>0$ for all $x\in (0, x^*)$. According to the maximum principle, we have $w_\phi \geq 0$ in $(0, x^*)\times (0,\infty)$. Hence $W_\psi\geq 0$, and $W$ is increasing in $\psi$.

\end{proof}

\subsection{Uniform bounds on \texorpdfstring{$\p_{YY} U$}{the second derivative of U}}

The first step of the proof of Proposition \ref{prop:max-princ} is the derivation of uniform $L^\infty$ bounds on $\p_{YY} U$. The result we prove in this paragraph is the following

\begin{lemma}
Let $U$ be a solution of \eqref{rescaled} on $(s_0, s_1)\times (0, + \infty)$ such that $U_{Y|Y=0}=1$ for all $s\in [s_0, s_1]$, and such that $U$ is strictly increasing in $Y$ for all $s$, with $\lim_{Y\to \infty}U(s,Y)=U_\infty(s)<+\infty$.

Assume that there exists $M_2$ such that
$$
-M_2\leq \p_{YY} U(s_0, Y)-1\leq 0\quad \forall Y>0,
$$
and that 
\be\label{compatibility}
\p_{YY}U(s_0, Y)= 1-12 a_4 b_0 Y^2 + O(Y^5) \text{ for } Y\ll 1.
\ee

Then
$$
-\max(M_2,1)\leq \p_{YY} U(s, Y)-1\leq 0\quad \forall Y>0\ \forall s\geq s_0.
$$

\label{lem:U_YY-prelim}
\end{lemma}
\begin{remark}
Assumption \eqref{compatibility} is a compatibility condition at a high order at $s=s_0$. It is propagated by the equation.
\end{remark}

\begin{proof}
We rely on the equation on $W$  in the $(s,\psi)$ variables. We recall that $\p_Y U = \p_\psi W/2$, and therefore
$$
\p_{YY} U(s,Y) = \frac{1}{2}\sqrt{ W(s,\psi(s,Y))} \p_{\psi \psi} W(s,\psi(s,Y))\quad \forall s\geq s_0,\ \forall Y>0.
$$ 
Therefore we derive estimates on the quantity
$$
F(s,\psi):=\sqrt{W} \p_{\psi\psi} W - 2.
$$
Notice that the assumptions on $U$ imply that 
$$
-2M_2\leq F(s_0,\psi)\leq 0\quad \forall \psi>0.
$$
On $\{Y=0\}$, we have $U_{YY}=1$, and therefore $F_{|\psi=0}=0$. Using Lemma \ref{lem:preliminary-bounds}, we also have $\lim_{\psi \to \infty} F(s,\psi)=-2$.

Furthermore, $F$ satisfies
$$
\p_s F= \frac{\p_s W}{2 \sqrt{W}}\p_{\psi\psi }W + \sqrt{W}\p_{\psi\psi} \p_s W.
$$
Using the equation on $W$ \eqref{eq:vM} and writing $\p_{\psi\psi } W=(F+2)/\sqrt{W}$, we infer that
$$
\p_s F =\frac{1}{2 W} F (F+2) + \frac{1}{2\sqrt{ W}} \p_{\psi }^2 W \left(2b W - \frac{3b}{2} \psi \p_{\psi} W\right) + \sqrt{W}\left( - b \p_{\psi}^2 W - \frac{3b}{2} \psi \p_\psi^3 W + \p_{\psi}^2F\right). 
$$
Gathering all the terms and using the formula
$$
\p_\psi F= \frac{1}{2\sqrt{W}} \p_\psi W \p_{\psi\psi} W + \sqrt{ W}\p_\psi^3 W,
$$
we obtain eventually
\be\label{eq:F}
\p_s F - \frac{1}{2W}F (F+2) + \frac{3b}{2}\psi \p_\psi F - \sqrt{W}\p_{\psi\psi} F=0.\ee

$\rhd$\textit{ First step: Lower bound on $F$ and consequences.}

We start with the lower bound, which is easier. Assume that $F$ has an interior minimum $F_{min}$ at some point $(s,\psi)$ for some $s\in (s_0, s_1]$, $\psi>0$. Then according to  equation \eqref{eq:F}, $F_{min} (F_{min}+2)\leq 0$, and therefore $F_{min}\in (-2,0)$. Thus $F(s,\psi)\geq \min (-2, \inf F(s_0))\geq \min(-2, -2M_2)$.

We infer from this lower bound on $F$ some non-degeneracy estimates for $W$ for $\psi$ close to zero. Indeed, it follows from the inequality $U_{YY}\geq -M_2'$ with $M_2'= \max (M_2-1, 0)$, that
$$
1-M_2' Y\leq U_Y\quad \forall Y>0.
$$
In particular, if $Y\leq Y_0:=(2M_2')^{-1}$, then $U_Y(s,Y)\geq 1/2$ and $U(s,Y)\geq Y/2$. As a consequence, if $\psi\leq \psi(s,Y_0)$, then $W_{\psi}\geq 1$. Now, the lower bound on $U_{YY}$ also entails that
$$
\psi(s, Y_0)=\int_0^{Y_0} U(s,Y)\:dY\geq \frac{Y_0^2}{4}= \frac{1}{16{M_2'}^2}.
$$
Hence in particular, for all $s\geq s_0$,
\be\label{non-deg-W}
\psi \leq \frac{1}{16{M_2'}^2}\Rightarrow W(s,\psi)\geq \psi\text{ and } W_\psi \geq 1.
\ee

$\rhd$\textit{ Second step: Upper bound on $F$.}

The derivation of the upper-bound is a little more involved. The main difficulty comes from the nonlinear term $F(F+2)/W$, which is also singular near $\psi=0$. In order to deal with it, we use a bootstrap type argument. 
%
Notice first that the preliminary bounds of Lemma \ref{lem:preliminary-bounds} entail that $W$ is Lipschitz continuous, uniformly in $\psi$ and locally uniformly in $s$ and that $\p_{\psi\psi} W$,  $\p_{\psi}^3 W$ are bounded (locally in $s$, uniformly in $\psi$) in $s\geq s_0$, $\psi\geq \delta$, for any $\delta>0$. Considering eventually equation \eqref{eq:F}, we deduce that $\p_s F$ is bounded in a neighbourhood of $s=s_0$, uniformly in $\psi$ for $\psi\geq \delta$. Furthermore, using  assumption \eqref{compatibility} on $U(s_0)$, both $F(s_0)/W(s_0)$ and  $\sqrt{W(s_0)} \p_{\psi \psi} F(s_0)$ are  bounded in a neighbourhood of $Y=0$, and therefore $\p_s F_{|s=s_0}$ is bounded in $L^\infty(\R_+)$.

We now set
$$
s_0':=\inf\{s\in [s_0,s_1],\ \exists \psi>0,\ F(s, \psi)\geq 1\}.
$$
It follows from the above arguments that $s_0'>s_0$. On the interval $[s_0, s_0']$, we have $F(s, \psi)\in [-2\max(M_2,1), 1]$. As a consequence, we multiply \eqref{eq:F} by $ F_+ p$, where $p\in \mathcal C^\infty(\R)$ is a non-increasing weight function such that $p\equiv 1$ for $\psi $ close to zero and $ p(\psi)=O(\psi^{-k})$ for some $k>1$ for $\psi>1$, with $|p'|/p\in L^\infty$, $p''/p\in L^\infty$. Since $F_{+|\psi=0}=0$, we obtain, as long as $s\leq s_0'$,
$$
\frac{d}{ds}\int_{\R_+}  F_+^2 p + \int_{\R_+}\sqrt{W} (\p_\psi F_+)^2 p - \frac{1}{2}\int_{\R_+} F_+^2 \p_\psi^2(\sqrt{W} p)\leq \frac{3}{2}\int_{\R_+}\frac{F_+^2}{W} p + \frac{3b}{4}\int_{\R_+} F_+^2 p.
$$
An easy computation gives
$$
 \p_\psi^2(\sqrt{W} p)= \frac{F+2}{2W} p- \frac{1}{4}\frac{W_\psi^2}{W^{3/2}}p + \frac{W_\psi}{\sqrt{W}} p' + \sqrt{W}p''.
$$
Using the assumptions on $p$, the upper-bound $\sqrt{W} \leq U_\infty(s)$ and the bound on $F$ for $s\leq s_0'$, we deduce eventually that
\begin{multline*}
\frac{d}{ds}\int_{\R_+}  F_+^2 p + \int_{\R_+}\sqrt{W} (\p_\psi F_+)^2 p +  \frac{1}{2}\int_{\R_+} F_+^2\frac{W_\psi}{\sqrt{W}}| p'|  + \frac{1}{8}\int_{\R_+} F_+^2 \frac{W_\psi^2}{W^{3/2}}p\\\leq C\int_{\R_+}\frac{F_+^2}{W} p + C(1+ U_\infty(s))\int_{\R_+} F_+^2 p.
\end{multline*}
The second term in the right-hand side will be handled thanks to a Gronwall type argument. The singularity of the first term in the right-hand side will be absorbed in the dissipation term. Indeed, let us first decompose the integral into two pieces depending on the value of $W$. First,
$$
\int_{\R_+}\mathbf 1_{W\geq \inf(1, (5M_2')^{-2})}\frac{F_+^2}{W} p \leq C \int_{\R_+} F_+^2 p,
$$
and as before, that part can be handled thanks to a Gronwall type argument. We thus focus of the values of $W$ below $\inf(1, (5M_2')^{-2})$. In that case, according to \eqref{non-deg-W}, we have $\psi \leq \inf(1, (5M_2')^{-2})=:\psi_0$ and
$$
\frac{W_\psi}{W}\geq W_\psi \geq 1.
$$
Let us choose $p$ so that $p(\psi)=1 $ for $\psi \in [0,1]$. We deduce  that there exists an explicit constant $C$ such that for all $\psi \in (0,\psi_0)$,
$$\sqrt{W} (\p_\psi F_+)^2  + \frac{1}{8}F_+^2 \frac{W_\psi^2}{W^{3/2}}\\
\geq  C\left(\p_\psi\left( W^{1/4}  F_+\right)\right)^2.
$$
Therefore, using the Hardy inequality, there exists a constant $C$ such that
$$
D(s):=\int_0^{\psi_0}\sqrt{W} (\p_\psi F_+)^2 p +  \frac{1}{2}\int_0^{\psi_0}F_+^2\frac{W_\psi}{\sqrt{W}}| p'|  + \frac{1}{8}\int_0^{\psi_0} F_+^2 \frac{W_\psi^2}{W^{3/2}}\geq C \int_0^{\psi_0}\frac{W^{1/2} F_+^2}{\psi^2}p.$$
Using once again the non-degeneracy of $W$ for $\psi$ close to zero (see \eqref{non-deg-W}), we infer that up to choosing a smaller $\psi_0$,
$$
\int_{\R_+}\mathbf 1_{W\leq \inf(1, (5M_2')^{-2})}\frac{F_+^2}{W} p\leq \frac{1}{2} D(s).
$$
Eventually, we obtain
$$
\frac{d}{ds}\int_{\R_+}  F_+^2 p\leq C (1+ U_\infty(s)) \int_{\R_+}  F_+^2 p \quad \forall \in [s_0, s_0'].
$$
Now, since $F_{|s=s_0}\leq 0$, we have $F_{+|s=s_0}\equiv 0$. The Gronwall Lemma implies that $F_+\equiv 0$ for $s\leq s_0'$. Therefore $F(s,\psi)\leq 0<1$ for all $s\leq s_0'$. It follows that $s_0'=s_1$, and thus $F(s,\psi)\leq 0$ for all $s\in [s_0, s_1]$ and for all $\psi>0$.

\end{proof}

Under the assumptions of Lemma \ref{lem:U_YY-prelim}, we therefore have
\be\label{est:U-prelim}\ba
\sup\left(Y-M_2'\frac{Y^2}{2} , 0\right)\leq U(s,Y) \leq Y+ \frac{Y^2}{2}, \\ \sup (1-M_2'Y, 0)\leq U_Y\leq 1+Y
\ea \qquad \forall Y>0,\ \forall s\geq s_0.
\ee
Notice that these estimates are independent of $s$, and that the constant $M_2'$ depends only on $M_2$.

\subsection{Construction of sub and super solutions for \texorpdfstring{$W$}{W}}
\label{ssec:subsupW}

We now derive pointwise estimates  on $U$, which will be used in the last paragraph of this section to obtained a refined lower bound on $\p_{YY} U$. We distinguish between different zones:
\begin{itemize}
\item On the zone $Y\ll s^{\beta_1}$, where $\beta_1>1/4$ is the parameter entering the definition of $w_1$ (see Proposition \ref{prop:est-cLUV-1}), the energy estimate $E_1(s)\lesssim s^{-13/4}$ actually provides a very good pointwise estimate of $U$. However, this estimate degenerates when $Y\gtrsim s^{\beta_1}$. Let us emphasize that we do need estimates on $U, U_Y, U_{YY}$ on the zone $Y\geq s^{\beta_1}$ in order to prove Proposition \ref{prop:est-cLUV-1} and therefore close the bootstrap argument.

\item On the zone $Y\geq C s^{1/4}$ for some large enough constant $C$, which corresponds to $\psi \gtrsim s^{3/4}$, we construct  sub and super solutions for $U$ (or rather, for $W$) by using maximum principle arguments. Note that this requires to have a good control of $W$ on the lower boundary of that zone, i.e. on the line $\psi= C' s^{3/4}$. This is achieved thanks to the pointwise control coming from the bound on $E_1$.

\end{itemize}

Let us start with the following Lemma:
\begin{lemma}
Assume that $U$ satisfies the assumptions of Lemma \ref{lem:U_YY-prelim} and that 
$$
E_1(s)\leq C_1 s^{-13/4}\quad \forall s\in [s_0, s_1],
$$
where $E_1$ is defined in Proposition \ref{prop:est-cLUV-1}.

Let $c>0$ be arbitrary. Then for all $Y\in [0, cs^{1/4}]$, provided $s_0$ is large enough (depending on $\beta_1, m_1$ and $c$),
$$\ba
U_{YY}(s,Y)=1-12 a_4 bY^2 +O(s^{-13/8} Y^{\frac{5+a}{2}} (1+Y)),\\
U(s,Y)= Y + \frac{Y^2}{2} - a_4 b Y^4   + O(s^{-13/8} Y^{\frac{9+a}{2}} (1+Y)).
\ea
$$
As a consequence, if $s_0$ is large enough (depending on $\beta_1, m_1, c $ and $C_1$),
$$
U_{YY}(s,Y)\geq 1-\frac{1}{2} b Y^2 \quad \forall Y\in [0, cs^{1/4}].
$$

\label{lem:approx-E1}
\end{lemma}

\begin{proof}
We recall that
$$
E_1(s)=\int_0^\infty \left( \p_Y^2 \cLU V\right)^2 w_1,
$$
with $w_1= Y^{-a} (1+ s^{-\beta_1} Y)^{-m_1}$. Therefore, choosing $s_0$ sufficiently large (depending on $\beta_1, m_1$ and $c$), we have, for all $Y\in [0, c s^{1/4}]$,
$$
w_1\geq Y^{-a} \left(1+ c s_0^{\frac{1}{4} - \beta_1}\right)^{-m_1} \geq \frac{1}{2} Y^{-a}.
$$
Using a simple Cauchy-Schwartz inequality, it follows that
$$
|\p_Y\cLU V(s,Y)| = \left| \int_0^Y \p_Y^2 \cLU V \right| \leq \sqrt{2} Y^\frac{1+a}{2} E_1(s)^{1/2}\quad \forall s\in [s_0, s_1],\ \forall Y\in [0, c s^{1/4}].
$$
Integrating twice, we obtain estimates on $\cLU V$ and $\int_0^Y \cLU V$.

Now
$$
\p_Y^2 V= L_U \cLU V = U \cLU V - U_Y \int_0^Y \cLU V.
$$
Using \eqref{est:U-prelim}, we infer that
$$
| \p_Y^2 V | \leq \bar C (1+Y) Y^\frac{5+a}{2} E_1^{1/2},
$$
where $\bar C$ is an explicit and computable constant. Therefore
$$
|V| \leq \bar C C_1^{1/2} (1+Y )Y^\frac{9+a}{2} s^{-13/8}  \quad \forall Y \in [0, cs^{1/4}].
$$
Writing $U=\Uapp + V$ and recalling the definition of $\Uapp$, we notice that $\Uapp= Y + Y^2/2 -a_4 b Y^4 +  O(s^{-2 } Y^7+ s^{-3} (Y^{10}+Y^{11} ) )$ for $Y\leq c s^{1/4}$, and we obtain the estimate announced in the statement of the Lemma. Notice that the remainder terms in $\Uapp$ (namely $O(s^{-2 } Y^7+ s^{-3} (Y^{10}+Y^{11} ) )$) are smaller than $(1+Y )Y^\frac{9+a}{2} s^{-13/8}$ in the region $Y\leq c s^{1/4}$.
\end{proof}

Let us now deduce from the above Lemma an asymptotic expansion of $W$ for $1\ll \psi \lesssim s^{3/4}$. Indeed, a precise pointwise estimate on $W$ is necessary in order to build sub and super solutions.

By definition of $W$ and $\psi$, we have, in terms of $Y$,
\begin{eqnarray*}
W&=& Y^2+ Y^3 + \frac{Y^4}{4}-a_4  b Y^6+O(s^{-1} Y^5 + s^{-13/8} Y^\frac{15+a}{2} )\\&=& \frac{Y^4}{4}\left(1+4  Y^{-1} + 4  Y^{-2} - 4 a_4 b Y^2  +O(s^{-1} Y + s^{-13/8} Y^\frac{7+a}{2} )\right)\\
\psi&=&\frac{1}{6}Y^3\left(1+ 3  Y^{-1}-\frac{6}{5} a_4 b Y^2  +O( s^{-13/8} Y^\frac{7+a}{2} )\right).
\end{eqnarray*}
Above, the notation $A=O(B)$ means the following: there exists a constant $C$, depending only on $C_1$, and there exists $S_0>0$ depending on $c, C_1, M_2, \beta_1, m_1$ such that for all $s\geq s_0\geq S_0$, for all $Y\in [1, c s^{1/4}]$, $|A| \leq C B$.

It follows that
\begin{eqnarray*}
W(s,\psi)&=&\frac{(6\psi)^{4/3}}{4}\left(1+4  Y^{-1} + 4  Y^{-2}- 4 a_4 \tb Y^2 +O(s^{-1} Y + s^{-13/8} Y^\frac{7+a}{2} )\right)\\&&\times\left(1+ 3  Y^{-1}-\frac{6}{5} a_4 b Y^2  +O(s^{-13/8} Y^\frac{7+a}{2} )\right)^{-4/3}.
\end{eqnarray*}
Performing an asymptotic expansion of the right-hand side for $1\ll Y \lesssim s^{1/4}$, we find that
$$
W(s,\psi)=\frac{(6\psi)^{4/3}}{4}\left(1+2 Y^{-2} -\frac{12}5 a_4 b Y^2+O(s^{-2} Y^5 + s^{-13/8} Y^\frac{5+a}{2} + Y^{-3})\right).
$$
Since $Y\sim (6\psi)^{1/3}$, we obtain eventually, for $1\ll \psi \lesssim s^{3/4}$,
\be\label{asympt-W}
W(s,\psi)=\frac{(6\psi)^{4/3}}{4}\left(1+2 (6\psi)^{-2/3} -\frac{12}5 a_4 b (6\psi)^{2/3} +O(s^{-1} \psi^{1/3} + s^{-13/8} \psi^\frac{7+a}{6} + \psi^{-1})\right).
\ee

We are now ready to construct a sub-solution for $W$ beyond $c s^{1/4}$, for some constant $c>0$ large but fixed, that will be determined later on. To that end, we introduce a\textit{ regularized modulation rate $\tb$}, that has the same asymptotic behavior as $b$, but whose role is to remove some time oscillations. More precisely, define $\tb$ by the ODE 
\be\label{def:tb}
\tb_s + b \tb=0,\quad \tb_{|s=s_0}=\frac{1}{s_0}.
\ee
We then have the following result (see Appendix B for a proof):
\begin{lemma}
	Assume that there exist constants $J>0$,  and $\eps>0$ such  that for all  $s\in [s_0, s_1]$
such  that for all  $s\in [s_0, s_1]$ 
\be\label{hyp:b-simple}
\ba
\int_{s_0}^{s_1}\left| b_s + b^2\right|^2 s^{13/4} ds \leq J,\\
\frac{1-\eps}{s}\leq b(s)\leq \frac{1+\eps}{s}.\ea
\ee

Then if $s_0$ is large enough (depending on $K$ and $\eps$), for all $s\geq s_0$,
$$
 \frac{1-2\eps}{s}\leq \tb(s)  \leq \frac{1+2\eps}{s}.
$$
\label{lem:tb}
\end{lemma}

We then define our subsolution and super-solutions in the following way:
\begin{lemma}
Assume that:
\begin{itemize}
	\item There exist constants $J>0$, $\eta\in (0,1)$ and $\eps>0$ such  that \eqref{hyp:b-simple} is satisfied;
	\item $U$ satisfies the assumptions of Lemma \ref{lem:U_YY-prelim};
	\item There exists a constant $M_0$ such that 
	\be\label{hyp:UYY-1}
	U_{YY}(s_0)-1 \geq -M_0 s_0^{-1} Y^2\quad \forall Y\geq 0\ee
	and such that $\lim_{Y\to \infty}U(s_0, Y)\leq M_0 s_0$;
	\item $E_1(s)\leq C_1 s^{-13/4}$ for all $s\in [s_0, s_1]$.
	
\end{itemize}

Then there exist a universal constant $\bar C$ and a  constant $A_0$, depending only on $M_0$, such that the following properties are satisfied:

$\bullet$ Sub-solution: For $A_->0$ define\footnote{The constant $C_+$ is such that $\uW(s,C_+ \tb^{-\frac{5}{4} })=0$. It can be determined explicitely, depending on $A_-$; however its precise value is irrelevant.}
$$
\uW(s,\psi):=\frac{(6\psi)^{4/3}}{4} - A_- \psi^{7/3} \tb^{\frac{5}{4}} \quad \forall s\in [s_0, s_1],\ \forall \psi\in \left[C_- \tb^{-3/4}, C_+ \tb^{-\frac{5}{4} }\right].
$$

If $A_-\geq A_0$, $C_-\geq \bar C$ and if $s_0$ is large enough, then
$$
\uW(s,\psi)\leq W(s,\psi)\quad \forall s\in [s_0, s_1],\ \psi \geq C_- \tb^{-3/4}.
$$

$\bullet$ Super-solution: For $A_+>0$, define
$$
\bar W(s,\psi):=\frac{(6\psi)^{4/3}}{4} + A_+ \psi^{10/3} \tb^{2} \quad \forall s\in [s_0, s_1],\ \forall \psi\geq C_- \tb^{-3/4}.
$$
If $A_-\geq A_0$, $C_-\geq \bar C$ and if $s_0$ is large enough, then
$$
 W(s,\psi)\leq \bar W(s,\psi)\quad \forall s\in [s_0, s_1],\ \psi \geq C_- \tb^{-3/4}.
$$

\label{lem:sub-U}
\end{lemma}
The proof of Lemma \ref{lem:sub-U} is postponed to the Appendix.

\subsection{Refined lower bound on \texorpdfstring{$\p_{YY} U$}{the second derivative of U}}

Lemma \ref{lem:sub-U} allowed us to extend the lower bound on $W$ coming from the estimation of $E_1$ beyond $\psi\simeq s^{3/4}$. Thanks to this extension, we now construct a sub-solution for $U_{YY}-1$ (or rather, for the function $F$ introduced in Lemma \ref{lem:U_YY-prelim}). Eventually, the lower bound on $U_{YY}-1$ will yield a finer lower bound on $U$.

\begin{lemma}\label{lem:U_YY-final}
Assume that the hypotheses of Lemma \ref{lem:sub-U} are satisfied. 
Then, setting $M_2:=\max(M_0, \bar M)$ for some universal constant $\bar M$,  there exists a constant $ c>0$ such that
$$
U_{YY}-1\geq -M_2b Y^2\quad \forall Y\in [0, c s^{1/3}].
$$

\end{lemma}
The proof is postponed to the Appendix.

Putting together the results of this section, we obtain Proposition \ref{prop:max-princ}.

\section{Main tools for the energy estimates}
\label{sec:proof-statements-algebraic}

This section is devoted to the derivation of several independent intermediate results which play an important role in the proof of energy estimates.  We first prove the result on the decomposition of the diffusion term and on the remainder term, namely Lemma \ref{lem:reste}. We then turn to the commutator Lemma \ref{lem:commutator}. We study the structure of the diffusion term (see Lemma \ref{lem:diff1}). Eventually,  we state some estimates allowing to perform a systematic treatment of some remainder terms.

For the sake of brevity, we   adopt the following notation, which we will use extensively in the next two sections: for any $\alpha>0$, and for quantities $A$ and $B$ that depend on $s$, we say that $A=O_\alpha(B)$ if there exists a constant $C$ and a function $Q=Q(s,Y)$ with at most polynomial growth in $s$ and $Y$, such that
\be\label{O-alpha}\ba
|A(s,Y)|\leq C | B(s,Y)| \quad\text{for } Y \leq s^{\alpha}, \\
|A(s,Y)|\leq Q(s,Y)\quad\text{for } Y \geq s^{\alpha}.\ea
\ee
This notation will be useful because we work with weights of the form $w(s,Y)= Y^{-a} (1 + s^{-\beta} Y)^{-m}$, where $m$ is an arbitrarily large integer. Therefore, the contribution of any function having at most polynomial growth in $s$ and $Y$ can be made as small as desired on the set $Y\geq s^{\alpha}$, in the following sense: if $\alpha>\beta$, for any integers $n,P\in \N$, if $m$ is large enough (depending on $\alpha, \beta, n$ and $P$),
$$
\int_{s^\alpha}^\infty (s^n + Y^n) w(s,Y)\:dY \leq s^{-P}.
$$
In other words, when we estimate functions in $L^2(w)$, their behavior for $Y\gg s^{\beta}$ is unimportant, as long as these functions are polynomially bounded (with an explicit and computable bound).

\subsection{Proof of Lemma \ref{lem:reste}}

Since $\Uapp$ is essentially a polynomial in $b$ and $Y$ (at least in the zone $Y\lesssim s^{2/7}$), the computation of the transport term $\p_s \Uapp - b \Uapp + bY/2 \p_Y \Uapp$ is straightforward. Difficulties stem from $\LU( \p_{YY} \Uapp -1)$, which is also present in the diffusion term $\cD$. Hence we start
 with a decomposition of the diffusion term 
$$
\cD:= \LU (\p_{YY} U-1),
$$
which will be useful in other occurrences. Writing $U= \Uapp + V$, we decompose $\cD$ into four parts: 
\begin{itemize}
\item the biggest term, which we compute explicitly, and which is equal to $-b/2 Y$. This term comes from $\Uapp$.
\item a second order term $\cLU V$;
\item a first order term $ \frac{b}{2}\LU L_V Y $;
\item additional error terms coming from $\Uapp$, which we will treat as perturbations in all occurrences.
\end{itemize}
%

Our precise result concerning the diffusion term is the following:

\begin{lemma}
We recall that $
\cD= \LU (\p_{YY} U-1).
$ Then
\begin{eqnarray*}
\cD&=& - \frac{b}{2} Y + \cLU V + \frac{b}{2}\LU L_{V} Y\\
&&+\LU \left(\left(\left(\frac{5}{4} a_7-90 a_{10} \right) b^3 Y^8 - 110 a_{11} b^3 Y^9 + 2 a_{10} b^4 Y^{11} + \frac{9}{4} a_{11} b^4 Y^{12}\right) \chi \left(\frac{Y}{s^{2/7}}\right)\right)\\
&&+\LU \left( P(s,Y) (1-\bar \chi) \left(\frac{Y}{s^{2/7}}\right)\right)
 \\&=& - \frac{b}{2} Y +\cD_{NL} + \tilde{\cD}
 \end{eqnarray*}
where $\chi, \bar \chi\in \mathcal C^\infty_0(\R)$ are cut-off functions such that $\chi, \bar \chi \equiv 1$ in a neighbourhood of zero,  and $P$ is a function that has at most polynomial growth in $s$ and $Y$.

We have set
\begin{eqnarray*}
\cD_{NL} &:=& \cLU V + \frac{b}{2}\LU L_V Y,\\
\tilde{\cD}&:=&\LU \left(\left(\left(\frac{5}{4} a_7-90 a_{10} \right) b^3 Y^8 - 110 a_{11} b^3 Y^9 + 2 a_{10} b^4 Y^{11} + \frac{9}{4} a_{11} b^4 Y^{12}\right) \chi \left(\frac{Y}{s^{2/7}}\right)\right)  \\
&+& \LU \left( P(s,Y) (1-\bar \chi) \left(\frac{Y}{s^{2/7}}\right)\right).
\end{eqnarray*}
\label{lem:cR}
\end{lemma}

\begin{remark}
The decomposition of Lemma \ref{lem:cR} will be used in two different occurrences: 
\begin{itemize}
\item First, we will use it to decompose the total diffusion term $\cD$ into a dissipation operator acting on the error term $V$, namely $\cLU V$, 
 and remainder terms, namely $- \frac{b}{2} Y $, $\frac{b}{2}\LU (L_V Y)$ and $\tilde \cD$. As we derive an equation on $V$, the diffusion term $\cLU V$ will be kept in the left-hand side of the equation, while the remainder terms will be added to the terms stemming from $\Uapp$ in the left-hand side.

\item Additionally, $\cD$ will appear in the commutator of $\cLU$ with $\p_s + \frac{b}{2} Y \p_Y$. We will then isolate the term in $\cD$ which bears the highest number of derivatives on $V$, namely $\cLU V$, which we will need to estimate separately in some instances.
\end{itemize}

\end{remark}

\begin{proof}
Throughout the proof, since we are not interested in the specific definitions of the functions $P, \chi, \bar \chi$, we keep uniform notations for these three objects, even though they are used to group together different terms. 

Recalling the definition of $\Uapp$ \eqref{def:Uapp}, we have
\begin{eqnarray*}
\cD&=&\LU(\p_{YY} U -1)= \LU(\p_{YY} \Uapp -1) + \cLU V\\
&=& -\LU \left((12 a_4 b Y^2+42 a_7 b^2 Y^5 + 90 a_{10} b^3 Y^8 +110 a_{11} b^3 Y^9) \chi\left(\frac{Y}{s^{2/7}}\right)\right) \\
&+&\LU (F''(\sqrt{b Y})- 1)  \\
&+& \LU \left((Y- a_4 bY^4 - a_7 b^2 Y^7 - a_{10} b^3 Y^{10} - a_{11} b^3 Y^{11}) \frac{1}{s^{{4/7}}} \chi''\left(\frac{Y}{s^{2/7}}\right)\right) \\
&+&2 \LU \left( (1-4 a_4 b Y^3 - 7 a_7 b^2 Y^6 -10 a_{10} b^3 Y^9 -11 a_{11} b^3 Y^{10})\chi'\left(\frac{Y}{s^{2/7}}\right)\right)\\
&+&\cLU V.
\end{eqnarray*}

We now examine each of the terms in the right-hand side separately. 
\begin{itemize}

\item The term
$
F''(\sqrt{b Y})- 1
$ and all the terms involving at least one derivative of $\chi$
are identically zero up to $Y\sim s^{2/7}$. Therefore they can all be written as 
$
P(s,Y) (1-\bar \chi) \left(\frac{Y}{s^{2/7}}\right).
$

\item We therefore focus on 
$$ \LU \left((12 a_4 b Y^2+42 a_7 b^2 Y^5 + 90 a_{10} b^3 Y^8 +110 a_{11} b^3 Y^9) \chi\left(\frac{Y}{s^{2/7}}\right)\right),$$ 
and in particular on the value of this term on $Y\leq s^{2/7}$. Indeed, all values for $Y\geq s^{{2/7}}$ can be written as $\LU (P(s,Y) (1-\bar \chi) \left(\frac{Y}{s^{2/7}}\right))$.

We recall that $12 a_4=1/4$, and that $42 a_7=a_4/2$. We compute
\begin{eqnarray}
L_U (Y)&=& U Y - U_Y \frac{Y^2}{2}\nonumber\\
&=& \Uapp Y - \Uapp_Y\frac{Y^2}{2} + L_{V} Y\nonumber\\
&=&  \left(\frac{Y^2}{2}+ a_4 b Y^5 + \frac{5}{2} a_7 b^2 Y^8 + 4 a_{10} b^3 Y^{11} + \frac{9}{2} a_{11} b^3 Y^{12}\right) \chi \left(\frac{Y}{s^{2/7}}\right)\label{LUY2}\\
&+& L_{V} Y+P(s,Y) (1-\bar \chi) \left(\frac{Y}{s^{2/7}}\right).\nonumber
\end{eqnarray}

Multiplying \eqref{LUY2} by $b/2$  and applying $L_U^{-1}$, we deduce that 
\begin{eqnarray*}
&&\LU \left((12 a_4 b Y^2+42 a_7 b^2 Y^5 + 90 a_{10} b^3 Y^8 +110 a_{11} b^3 Y^9) \chi\left(\frac{Y}{s^{2/7}}\right)\right)\\
&=&\frac{b}{2} Y - \frac{b}{2} \LU L_V Y\\
&+&  \LU \left(\left((90 a_{10} - \frac{5}{4} a_7) b^3 Y^8 + 110 a_{11} b^3 Y^9 - 2 a_{10} b^4 Y^{11} - \frac{9}{4} a_{11} b^4 Y^{12}\right) \chi \left(\frac{Y}{s^{2/7}}\right)\right)\\
&+ &\LU \left( P(s,Y) (1-\bar \chi) \left(\frac{Y}{s^{2/7}}\right) \right).
\end{eqnarray*}
Gathering all the terms, we obtain the  decomposition announced in the Lemma.

\end{itemize}
\end{proof}

\begin{corollary}
Assume that there exist constants $M,c$ such that
$$\ba
| \p_{YY} U - 1 |\leq M b Y^2 \quad \forall Y \in [0, c s^{1/3}],\ \forall s \in [s_0, s_1],\\
\frac{1}{2s}\leq b\leq \frac{2}{s},\\
|U_{YY}|\leq M \quad \forall Y>0,\ \forall s \in [s_0, s_1].
\ea
$$
Then 
$$
\cD= O_{1/3} (bY),\quad \p_Y \cD= \p_Y \cLU V + O_{2/7} (b).
$$

\label{cor:est-cD}
\end{corollary}

\begin{proof}
Note first that under these assumptions, 	
	\[
	\ba
 Y + \frac{Y^2}{2} - \frac{M}{12} b Y^4	\leq U(s,Y)\leq Y + \frac{Y^2}{2} + \frac{M}{12} b Y^4,\\
 1+Y- \frac{M}{4} b Y^3 \leq U_Y(s,Y)\leq 1+Y+ \frac{M}{4} b Y^3,
	\ea
	\qquad \forall Y\in [0, cs^{1/3}].
	\]
The  estimate on $\cD$ follows simply from writing
$$
\cD= U_Y \int_0^Y \frac{\p_{YY} U -1}{U^2} + \frac{\p_{YY } U -1}{U}
$$
and using the bounds on $\p_{YY } U, U_Y$ and $U$. As for the second one, notice that
$$
\p_Y \cD= \p_Y \cLU V - \frac{b}{2} + \p_Y \LU Z,
$$
where
\begin{eqnarray*}
Z&=& \frac{b}{2} (Y V - \frac{Y^2}{2} V_Y  )\\&&+ \left(\left(\frac{5}{4} a_7-90 a_{10} \right) b^3 Y^8 - 110 a_{11} b^3 Y^9 + 2 a_{10} b^4 Y^{11} + \frac{9}{4} a_{11} b^4 Y^{12}\right) \chi\left(\frac{Y}{s^{2/7}}\right) \\&&+ P(s,Y) (1-\bar \chi) \left(\frac{Y}{s^{2/7}}\right)
\end{eqnarray*}
where $P$ has at most polynomial growth in $s$ and $Y$. The estimate then follows from the bounds
$$
V= O_{2/7}( b Y^4),\ V_Y=O_{2/7}( b Y^3),\  V_{YY}=O_{2/7}( b Y^2)
$$
and from the formula giving $\p_Y \LU$ in Lemma \ref{lem:p3LU}.

\end{proof}

We deduce that $V$ is a solution of equation \eqref{eq:V-2}, i.e.
$$
\p_s V - b V + \frac{b}{2} Y\p_Y V  -\cLU V= \cR,
$$
 with a remainder
\begin{eqnarray*}
\mathcal{R}&:=& -\left(\p_s \Uapp - b \Uapp + \frac{b}{2} Y \p_Y \Uapp \right) - \frac{b}{2} Y\\
&+& \LU \left(\left(\left(\frac{5}{4} a_7-90 a_{10} \right) b^3 Y^8 - 110 a_{11} b^3 Y^9 + 2 a_{10} b^4 Y^{11} + \frac{9}{4} a_{11} b^4 Y^{12}\right) \chi \left(\frac{Y}{s^{2/7}}\right)\right)\\
&+& \frac{b}{2} \LU (L_V Y)\\
&+& \LU \left( P(s,Y) (1-\bar \chi) \left(\frac{Y}{s^{2/7}}\right) \right),
\end{eqnarray*}
which we now compute.

\begin{lemma}[Computation  of $\mathcal R$]
\label{lem:comp-R0}
The remainder term $\mathcal R$ can be written as
\begin{eqnarray*}
\cR&=& (b_s + b^2)(a_4 Y^4 + 2 a_7 b Y^7+3a_{10} b^2 Y^{10} + 3 a_{11} b^2 Y^{11})\chi \left(\frac{Y}{s^{2/7}}\right)\\
&+&\left(a_{10}b^4 Y^{10} +3 a_{11}b^4/2 Y^{11}\right)\chi \left(\frac{Y}{s^{2/7}}\right)\\
&+ &\frac{b}{2} \LU (L_V Y)+  \frac{1}{2} b^3 a_7 \LU \left(\left(Y^7 V - V_Y \frac{Y^8}{8}\right)\chi \left(\frac{Y}{s^{2/7}}\right)\right)\\
&+& \LU\left[ \left(d_{11} b^4 Y^{11} + d_{12} b^5 Y^{12} + d_{14} b^5 Y^{14} + d_{17} b^6 Y^{17} + d_{18} b^6 Y^{18}\right)\chi \left(\frac{Y}{s^{2/7}}\right)\right]\\
&+&  \LU \left( P(s,Y) (1-\bar \chi) \left(\frac{Y}{s^{2/7}}\right) \right),
\end{eqnarray*}
for some explicit constants $d_{11}, d_{12}, d_{14}, d_{17}, d_{18}\in \R$ .

\end{lemma}

\begin{remark}
Notice that this remainder term is essentially (up to a small error depending on $V$)
$$
\LU \left( \Uapp \p_s \Uapp - \p_Y \Uapp \int_0^Y \Uapp -2 b (\Uapp)^2 + \frac{3b}{2} \p_Y\Uapp\int_0^Y \Uapp - \p_{YY} \Uapp + 1 \right),
$$
and therefore has been computed (up to the application of the operator $\LU$) when the approximate solution $\Uapp$ was defined. However, it is actually easier to do over the computations rather than to apply $\LU$ to the remainder that has already been computed.

\end{remark}

\begin{proof}
We start with
\begin{eqnarray*}
&&\p_s \Uapp - b \Uapp + \frac{b}{2} Y \p_Y \Uapp + \frac{b}{2} Y \\
&=&- a_4\left( b_s + {b^2}\right) Y^4\chi\left(\frac{Y}{s^{2/7}}\right)\\
&+& \chi\left(\frac{Y}{s^{2/7}}\right) \left[  - a_7b \left(2 b_s +\frac{5b^2}{2}\right) Y^7 - a_{10} Y^{10} (3 b_s b^2 + 4 b^4) - a_{11} Y^{11}(3 b_s b^2 + \frac{9}{2} b^4) \right]\\
&+&\chi'\left(\frac{Y}{s^{2/7}}\right) \frac{Y}{s^{2/7}} \left( \frac{b}{2} - \frac{{2/7}}{s}\right)\left[ Y- a_4 b Y^4 - a_7 b^2 Y^7 - a_{10 } b^3 Y^{10} - a_{11} b^3 Y^{11}  \right] \\
&+&b^{-2} \left( b_s + {b^2}\right) \left(\frac{1}{2} Z \Theta'(Z) - \Theta(Z)\right)_{|Z=\sqrt{b} Y} \\
&+& \frac{b}{2} Y (1-\chi) \left(\frac{Y}{s^{2/7}}\right) .
\end{eqnarray*}
The last three terms in the right-hand side are supported in $Y \geq c_1 s^{2/7}$. They can be written as a linear combination of terms of the type $ P(s,Y) (1-\bar \chi) \left(\frac{Y}{s^{2/7}}\right) $. 

We thus focus on the second term, which we group with the other terms in the definition of $\mathcal R$. We first isolate the factor $(b_s+b^2)$ in  $Y^7, Y^{10}, Y^{11}$, which we group with the first term. There remains to study
\begin{multline*}
\frac{1}{2} a_7 b^3 Y^7 \chi\left(\frac{Y}{s^{2/7}}\right) \\+b^3\LU \left(\left(\left(\frac{5}{4} a_7-90 a_{10} \right)  Y^8 - 110 a_{11}  Y^9 + 2 a_{10} b Y^{11} + \frac{9}{4} a_{11} b Y^{12}\right) \chi \left(\frac{Y}{s^{2/7}}\right)\right).
\end{multline*}
We use the same trick as in Lemma \ref{lem:cR} and we write
$$
Y^7\chi\left(\frac{Y}{s^{2/7}}\right)= \LU \left((L_{\Uapp} + L_V) Y^7\chi\left(\frac{Y}{s^{2/7}}\right)\right)
$$
Notice that $L_V Y^7= V Y^7 - V_Y Y^8/8$. On the other hand, a lengthy but straightforward computation yields
$$
L_{\Uapp} Y^7 = \frac{7}{8} Y^8 + \frac{3}{8} Y^9 - \frac{ a_4}{2} b Y^{11} - \frac{a_7}{8} b^2 Y^{14} + \frac{a_{10}}{4} b^3 Y^{17} + \frac{3 a_{11}}{8} b^3 Y^{18}.
$$
Gathering all the terms and recalling the values of $a_{10}$, $a_{11}$ \eqref{def:a10a11}, we obtain the decomposition announced in the Lemma.


\end{proof}

\subsection{Proof of the commutator result (Lemma \ref{lem:commutator})}

We compute separately $[\LU, \p_s]$ and $[\LU, Y \p_Y]$, and then check that cancellations occur between the two commutators.

Using the formulas in Lemma \ref{lem:p3LU} in the Appendix, we have
\begin{eqnarray*}
[\LU, Y\p_Y] W&=&\left( U \int_0^Y \frac{Y\p_Y W}{U^2}\right)_Y - Y \p_Y \LU W \nonumber\\
&=& U_Y \int_0^Y \frac{Y \p_Y W}{U^2} + \frac{Y \p_Y W}{U}- Y \left( \p_Y^2 U \int_0^Y \frac{W}{U^2} + \frac{\p_Y W} U\right)\\
&=&U_Y \int_0^Y \frac{Y \p_Y W}{U^2} -Y\p_Y^2 U \int_0^Y \frac{W}{U^2} .
\end{eqnarray*}
We introduce the quantity
$$
\Gamma:=Y U_Y - 2U,
$$
so that $\Gamma_Y=Y \p_Y^2 U - U_Y$. Then
\begin{eqnarray}
[\LU, Y\p_Y] W&=&-\Gamma_Y \int_0^Y   \frac{W}{U^2} + U_Y \int_0^Y \frac{\p}{\p Y }\left(  \frac{W}{Y}\right)   \frac{Y^2}{U^2}\nonumber\\
&=& -\Gamma_Y \int_0^Y   \frac{W}{U^2} +   \frac{YU_Y}{U^2} W - 2 U_Y\int_0^YW   \frac{U-YU_Y}{U^3}\nonumber\\
&=&2 \LU W - \Gamma_Y \int_0^Y   \frac{W}{U^2} +   \frac{\Gamma}{U^2} W + 2U_Y\int_0^Y W  \frac{\Gamma}{U^3}.\label{eq:comm-cLU}
\end{eqnarray}


We now address the commutator with $\p_s$. To that end, we recall that $U$ satisfies \eqref{eq:U-LU}, so that, with the previous definitions of $f$ and $\cD$,
\be\label{Us}U_s= -\frac{b}{2} \Gamma  + \cD.\ee

It follows that 
\begin{eqnarray*}
[\LU, \p_s] W &=&-\left(U_s \int_0^Y \frac{ W}{U^2}\right)_Y + 2 \left( U \int_0^Y \frac{ W}{U^3} U_s\right)_Y\\
&=&\frac{b}{2}\left(\Gamma \int_0^Y \frac{ W}{U^2}\right)_Y - b\left( U \int_0^Y \frac{ W}{U^3} \Gamma\right)_Y-\left(\cD \int_0^Y \frac{ W}{U^2}\right)_Y + 2 \left( U \int_0^Y \frac{ W}{U^3} \cD\right)_Y.
\end{eqnarray*}
Using \eqref{eq:comm-cLU}, we infer that
\begin{eqnarray*}
&&\left[\LU, \p_s + \frac{b}{2} Y \p_Y \right] W\\ &=& b \LU W-\frac{b}{2} \Gamma_Y \int_0^Y \frac{W}{U^2}+ \frac{b}{2} \frac{\Gamma}{U^2} W + b U_Y\int_0^Y W\frac{\Gamma}{U^3}\\
&+&\frac{b}{2 }\left(\Gamma \int_0^Y \frac{ W}{U^2}\right)_Y - b\left( U \int_0^Y \frac{ W}{U^3} \Gamma\right)_Y-\left(\cD \int_0^Y \frac{ W}{U^2}\right)_Y + 2 \left( U \int_0^Y \frac{ W}{U^3} \cD\right)_Y\\
&=&b \LU W-\left(\cD \int_0^Y \frac{ W}{U^2}\right)_Y + 2 \left( U \int_0^Y \frac{ W}{U^3} \cD\right)_Y.
\end{eqnarray*}

This completes the proof of the commutator Lemma.
\subsection{Structure of the diffusion term}


We will use in several instances the following weighted Hardy inequality (see \cite{masmoudi2011hardy}):

\begin{lemma}
Let $p_1, p_2$ be measurable functions such that $p_1, p_2> 0$ almost everywhere. Let $0<R\leq \infty$ and let
$$
C_H:=4 \sup_{0<r<R}\left(\int_{r}^R p_1\right) \left(\int_0^r \frac{1}{p_2} \right).
$$
Assume that $C_H<+\infty$.
Then for any function $f\in H^1_{loc}(\R_+)$ such that $f(0)=0$, there holds
$$
\int_0^R f^2\; p_1 \leq C_H \int_0^R (\p_Y f)^2 p_2.
$$

\label{lem:Hardy}

\end{lemma}

In this paragraph, we state and prove the coercivity inequality that will be used to control $(E_1, D_1)$ and $(E_2, D_2)$, up to small remainder terms.

\begin{lemma}Let $s_0<s_1$, and let $\delta>0$ be arbitrary.
Assume that $U$ is increasing in $Y$ for all $s\geq s_0$, with $U_{|Y=0}=0, \p_Y U_{|Y=0}=1$, and that there exists constants $M_2, c$ such that
\be\label{hyp:U_YY} \ba
1-M_2bY^2 \leq \p_{YY} U\leq 1 \quad \forall Y \in [0, c s^{1/3}],\ \forall s\in [s_0,s_1],\\
-M_2\leq \p_{YY} U \leq 1 \quad \forall Y\geq c s^{1/3}\ \forall s\in [s_0,s_1].\ea
\ee
Assume furthermore that 
$$
\frac{1}{2s}\leq b\leq \frac{2}{s}\quad \forall s\in [s_0, s_1].
$$

For $a>0, \beta\in ]1/4, 2/7[, m\in \N$, define the weight $w(s,Y):=Y^{-a} (1+ s^{-\beta} Y)^{-m}$.

There exist universal constants $\bar a, \bar c>0$,
independent of $\delta, s,\beta, m$, such that if $s_0$ is large enough (depending on $\delta, \beta, a$), then
for all $f\in W^{1,\infty} (\R_+)$ such that $\p_Y f= O(Y)$ for $Y\ll 1$, for all $0<a\leq \bar a$, for any cut-off function $\chi\in \mathcal C^\infty_0(\R_+)$ such that $\chi\equiv 1$ in $[0,1]$,
\begin{eqnarray*}
-\int_0^\infty \left( \p_{YY}\LU f \right)\;  f \; w &\geq& \bar c \left( \int_0^\infty \frac{(\p_Y f)^2}{U} w + \int_0^\infty \frac{f^2}{U^2} w \right) -\delta b \int_0^\infty f^2 w\\
&&-C\int_{cs^{1/3}}^\infty U \left(\int_0^Y \left(1-\chi\left(\frac{Y}{s^{1/4}}\right)\right) \frac{f}{U^2}\right)^2 w.
\end{eqnarray*}

\label{lem:diff1}
\end{lemma}

\begin{remark}
Note that the estimate we prove here is not as strong as one would like to have. Indeed, we only control $f$ and $\p_Y f$, and not $\LU f$. However, another way of writing the inequality, setting $h= \LU f$, is
$$
-\int_0^\infty \p_{YY} h \; L_U h\; w \geq \bar c \int_0^\infty \frac{(\p_Y (L_U h))^2}{U} w + \int_0^\infty \frac{(L_U h)^2}{U^2} w + \text{ remainder terms.}
$$
But notice that if $h=U_Y$, then $L_U h=0$ while $h\neq 0$, $\p_Y h\neq 0$. Hence it is hopeless to control $h$ and $\p_Y h$ by the left-hand side of the above inequality without any further assumption that would discard the case $h=U_Y$.

\end{remark}

\begin{remark}
The last term in the right-hand side of the inequality, namely
$$
C\int_{cs^{1/3}}^\infty U \left(\int_0^Y \left(1-\chi\left(\frac{Y}{s^{1/4}}\right)\right) \frac{f}{U^2}\right)^2 w
$$
will be handled in Lemma \ref{lem:tails}. Heuristically, since it is supported in a zone where $w$ is strongly decaying, it can be made ``as small as desired'', i.e. $O(s^{-P})$ for any $P>0$, provided $m$ is chosen sufficiently large (depending on $P$ and $\beta$).

\end{remark}

\begin{remark}
Notice that under the assumptions of the Lemma, there exists a constant $C$ such that for $1\leq Y \leq cs^{1/3}$,
$$
\frac{C^{-1}}{Y^2}\leq \frac{1}{U}\leq \frac{C}{Y^2}, \quad \frac{C^{-2}}{Y^4}\leq \frac{1}{U^2}\leq \frac{C^2}{Y^4}.
$$
Hence, for $1\leq Y \leq c s^{1/3}$, the two integrals in the diffusion term
$$
\int_0^\infty \frac{(\p_Y f)^2}{U} w + \int_0^\infty \frac{f^2}{U^2} w 
$$
have the same scaling.

However, if $Y\leq 1$, then
$$
\frac{C^{-1}}{Y}\leq \frac{1}{U}\leq \frac{C}{Y}, \quad \frac{C^{-2}}{Y^2}\leq \frac{1}{U^2}\leq \frac{C^2}{Y^2}.
$$
Hence for $Y\leq 1$, the control given by the first integral is stronger. Indeed, the Hardy inequality implies that
$$
\int_0^1 \frac{(\p_Y f)^2}{Y^{1+a}}\geq C \int_0^1 \frac{ f^2}{Y^{3+a}}.
$$
Therefore
$$
\int_0^{cs^{1/3}} \left( \frac{(\p_Y f)^2}{U} +\frac{f^2}{U^2}\right) w   \geq C \left(\int_1^{cs^{1/3}} \frac{f^2}{Y^4} w + \int_0^1 \frac{f^2}{Y^{3+a}}\right).
$$
We will often use this control close to zero when we deal with some non-local terms.\label{rem:diff-zero}
\end{remark}

\begin{proof}[Proof of Lemma \ref{lem:diff1}]
Starting with a simple integration by parts,
$$
-\int_0^\infty \left( \p_{YY}\LU f \right)\;  f \; w= \int_0^\infty \left(\p_Y \LU f\right)\: (\p_Y f w + f \p_Y w).
$$
Using the formula in lemma \ref{lem:p3LU} and integrating by parts once again, we obtain
\begin{eqnarray*}
-\int_0^\infty \left( \p_{YY}\LU f \right)\;  f \; w&=& \int_0^\infty \frac{(\p_Y f)^2}{U} w - \int_0^\infty \frac{f^2}{U^2}w\\
&&+ \int_0^\infty (U_{YY}-1) \left(\int_0^Y \frac{f}{U^2}\right)\: (f \p_Y w + \p_Y f w)\\
&&+ \int_0^\infty \frac{1}{U} \p_Y f \: f \p_Y w.
\end{eqnarray*}
The first two terms are the main order terms. We now prove the coercivity thanks to a weighted Hardy inequality for which we compute the constant explicitly.
Using the assumptions on $U$, we have
$$
U \geq Y + \frac{Y^2}{2} - \frac{M_2}{12} b Y^4\quad \text{for } Y \in [0, cs^{1/3}].
$$
Therefore, since $U$ is increasing, 
$$
U(s,Y)\geq \frac{K^2}{4} s^{1/2}\quad \forall Y\geq K s^{1/4}
$$
provided $s_0$ is sufficiently large. It follows that if $K$ is chosen large enough (depending on $\delta$),
$$
\int_{Ks^{1/4}}^\infty \frac{f^2}{U^2}w\leq  \frac{16}{K^4} s^{-1} \int_{Ks^{1/4}}^\infty f^2 w\leq \delta b \int_0^\infty f^2 w.
$$
On the set $[0, K s^{1/4}]$, we use a weighted Hardy inequality (see Lemma \ref{lem:Hardy}), namely
$$ \int_0^{Ks^{1/4}}\frac{1}{U^2} f^2 w
\leq  C_a \int_0^{K s^{1/4}} \frac{(\p_Y f)^2}{U}w
$$
where the constant $C_a$ satisfies
$$
C_a\leq 4 \sup_{0<r<Ks^{1/4}} \left(\int_{r}^{Ks^{1/4}} \frac{w}{U^2}\right) \left(\int_{0}^{r} \frac{U}{w}\right).
$$
On the set $[0, K s^{1/4}]$, we have
$$\ba
\mu Y + \frac{Y^2}{2}\leq U \leq Y + \frac{Y^2}{2}\quad \text{with } \mu=1-\frac{M_2}{24} K^3 s^{-1/4},\\
\text{and } Y^{-a} (1-\delta)\leq w\leq Y^{-a}\ea
$$
for any $\delta>0$ provided $s_0$ is sufficiently large. Therefore
$$
C_a \leq \frac{1}{1-\delta} \bar C_{a,\mu} ,
$$
where
$$
\bar C_{a,\mu}:=4\sup_{r>0} \left(\int_{r}^\infty \frac{Y^{-a}}{\left(\mu Y + \frac{Y^2}{2}\right)^2}dY\right) \left(\int_{0}^{r} Y^{a}\left( Y + \frac{Y^2}{2}\right)dY\right).
$$
We then have the following Lemma, which is proved in the Appendix:
\begin{lemma}
There exists universal constants $\bar a>0, \mu_0\in (0,1)$ such that for all $a\in (0, \bar a)$, for all $\mu\in (\mu_0, 1)$,
$$
C_{a,\mu}\leq \frac{9}{10}.
$$
\label{lem:constant-Hardy}
\end{lemma}

Therefore, for any $a\in (0, \bar a)$, provided $\delta$ is small enough (say $\delta<1/50$) and $s_0$ is sufficiently large,
$$
C_a\leq 1- \frac{1}{20}
$$
and thus
\be\label{in:coercivity1}
\int_0^\infty \frac{(\p_Y f)^2}{U} w - \int_0^\infty \frac{f^2}{U^2}w
 \geq \frac{1}{40} \left(\int_0^\infty \frac{(\p_Y f)^2}{U} w + \int_0^\infty \frac{f^2}{U^2}w\right) - \delta b \int_0^\infty f^2 w.
\ee

There remains to estimate the two lower order terms, namely
$$
\int_0^\infty \frac{1}{U} \p_Y f \: f \p_Y w\quad \text{and } \int_0^\infty (U_{YY}-1) \left(\int_0^Y \frac{f}{U^2}\right)\: (f \p_Y w + \p_Y f w).
$$
For the first lower order term, we distinguish once again between the zones $Y \leq K s^{1/4}$ and $Y \geq K s^{1/4}$. Using the Cauchy-Schwartz inequality, we have
$$
\left|\int_0^\infty \frac{1}{U} \p_Y f \: f \p_Y w\right|\leq \left( \int_0^\infty \frac{(\p_Y f)^2}{U} w \right)^{1/2} \left(\int_0^\infty \frac{f^2}{U} \frac{(\p_Y w)^2}{w}\right)^{1/2}.
$$
For $Y\leq  K s^{1/4}$, we have, for $s_0$ large enough (depending on $a, m, \beta,K$)
$$
\left|\p_Yw\right|\leq 2 a Y^{-1} w.
$$
Using a Hardy inequality, we have
\be\label{Hardy-g}
\int_0^{Ks^{1/4}} f^2\frac{1}{Y^{2} U}w \leq C_H \int_0^{K s^{1/4}} \frac{(\p_Y f)^2}{ U} w
\ee
where the constant $C_H$ is defined by
$$
C_H=4\sup_{0<r<K s^{1/4}}\left( \int_{r}^{K s^{1/4}}\frac{1}{Y^{2} U}w \right)\left(\int_0^r  U \frac{1}{w}\right).
$$
As above, we have, provided $s_0$ is sufficiently large (depending on $m$, $\beta$, $K$)
$$
C_H \leq 16 \sup_{r>0} \left( \int_{r}^{\infty}\frac{1}{Y^{2+a} (Y+Y^2)}dY \right)\left(\int_0^r  Y^a (Y + Y^2)dY\right).
$$
Studying separately the cases $r<1$ and $r>1$, 
it can be easily proved that $C_H$ is bounded uniformly in $a$ and $s$, so that there exists a universal constant $\bar C$ such that for $s$ large enough
$$
\int_0^{K s^{1/4}}\frac{f^2}{U} \frac{(\p_Y w)^2}{w}\leq \bar C a\int_0^{K s^{1/4}} \frac{(\p_Y f)^2}{ U} w.
$$
This term is absorbed in the main order diffusion term for $a$ small enough.
Now, for $Y\geq K s^{1/4}$, we have
$$
\left|\p_Yw\right|\leq \left( a  + m \right) Y^{-1} w \leq (a+m) K^{-1} s^{-1/4} w,
$$
and $U(s,Y)\geq \frac{K^2}{4} s^{1/2}$. As a consequence,
$$
\int_{Ks^{1/4}}^\infty \frac{f^2}{U} \frac{(\p_Y w)^2}{w} \leq 4\frac{(a+m)^2}{K^4} s^{-1}\int_0^\infty f^2 w\leq \delta b \int_0^\infty f^2 w
$$
for $K$ large enough (depending on $m, \delta$). We infer that
\be
\label{in:coercivity2}
\left|\int_0^\infty \frac{1}{U} \p_Y f \: f \p_Y w \right| \leq \bar C a^{1/2} \int_0^\infty \frac{(\p_Y f)^2}{U} w + \delta b \int_0^\infty f^2 w.
\ee

We now address the second lower order term. We focus on the term involving $\p_Y f w$, since the one with $f \p_Y w$ can be treated with similar ideas combined with the same estimates as above. Using
  assumption \eqref{hyp:U_YY} on $\p_{YY} U$, we have
\begin{multline*}
\left|\int_0^{cs^{1/3}} (U_{YY}-1) \left(\int_0^Y \frac{f}{U^2}\right)\: \p_Y f w\right|\\\leq M_2 b \left(\int_0^{cs^{1/3}} \frac{(\p_Y f)^2}{ U} w \right)^{1/2}\left( \int_0^{cs^{1/3}} Y^4\left(\int_0^Y \frac{f}{U^2}\right)^2 U w \right)^{1/2}.
\end{multline*}
We separate the last integral in the right-hand side into three zones: $Y\leq 1$, $1\leq Y \leq K s^{1/4}$ for some large constant $K$, and $Y\geq K s^{1/4}$. Notice that if $Y\leq Ks^{1/4}$, then
$$
w(s,Y)\geq (1+ K s_0^{1/4-\beta})^{-m} Y^{-a}\geq \frac{1}{2}Y^{-a}
$$
for $s_0$ large enough (depending on $K$, $m$ and $\beta$).
For $Y\leq 1$, we have, using \eqref{Hardy-g}
$$
\left|\int_0^Y \frac{f}{U^2}\right|\leq \left(\int_0^Y \frac{f^2}{Y^{2+a} U}\right)^{1/2}\left(\int_0^Y \frac{Y^{2+a}}{U^3}\right)\leq \frac{\bar C}{a^{1/2}} \left(\int_0^1\frac{(\p_Y f)^2}{U}w\right)^{1/2}.
$$
And if $1\leq Y \leq K s^{1/4}$, using a simple Cauchy-Schwartz inequality,
$$
\int_0^Y \frac{|f|}{U^2}\leq\frac{\bar C}{a^{1/2}}  \left(\int_0^1\frac{(\p_Y f)^2}{U}w\right)^{1/2} + \bar C \left(\int_1^Y \frac{f^2}{U^2} w\right)^{1/2}.
$$
Therefore, we have
\begin{eqnarray*}
\int_0^{Ks^{1/4}}Y^4\left(\int_0^Y \frac{f}{U^2}\right)^2 U w 
&\leq &\frac{\bar C}{a} \left[\int_0^\infty \frac{(\p_Y f)^2}{U} w + \int_0^\infty \frac{f^2}{U^2}w\right] \int_0^{Ks^{1/4}}Y^{4-a} U\\
&\leq & \frac{\bar CK^7}{a}  \left[\int_0^\infty \frac{(\p_Y f)^2}{U} w + \int_0^\infty \frac{f^2}{U^2}w\right] s^{(7-a)/4}.
\end{eqnarray*}
It follows that
\begin{multline*}
M_2b\left(\int_0^{cs^{1/3}} \frac{(\p_Y f)^2}{ U} w \right)^{1/2}\left( \int_0^{K s^{1/4}} Y^4\left(\int_0^Y \frac{f}{U^2}\right)^2 U w \right)^{1/2}\\
\leq \bar C M_2 \frac{K^{7/2}}{a^{1/2}} s^{-\frac{1+a}{8}} \left( \int_0^\infty \frac{(\p_Y f)^2}{U} w + \int_0^\infty \frac{f^2}{U^2}w \right).
\end{multline*}
Therefore, for $s_0$ sufficiently large (depending on $K$ and $a$), this term can be absorbed in the main order diffusion term. On the other hand, using the same estimates as above,
\begin{eqnarray*}
&& \int_{Ks^{1/4}}^{c s^{1/3}} Y^4\left(\int_0^Y \frac{f}{U^2}\right)^2 U w  \\
&\leq &  \frac{C}{a} \left[\int_0^\infty \frac{(\p_Y f)^2}{U} w + \int_0^\infty \frac{f^2}{U^2}w\right] \int_{Ks^{1/4}}^\infty Y^4 U w  +  \int_{Ks^{1/4}}^{c s^{1/3}}  Y^4 U \left(\int_{Ks^{1/4}}^Y \frac{|f|}{U^2}\right)^2   w\\
&\leq & C_{a,m} s^{(7-a)\beta} \left[\int_0^\infty \frac{(\p_Y f)^2}{U} w + \int_0^\infty \frac{f^2}{U^2}w\right] + C_{m,\beta}(s) \int_{Ks^{1/4}}^{c s^{1/3}}  \frac{f^2}{U^4} w
\end{eqnarray*}
where
\begin{eqnarray*}
 C_{m,\beta}(s)&=& 4 \sup_{r\in (Ks^{1/4}, c s^{1/3})} \left(\int_{r}^{c s^{1/3}}  Y^4 U w \right)\left( \int_{Ks^{1/4}}^r \frac{1}{w}\right)\\
 &\leq & 8 s^{8\beta} \sup_{r\in (K s^{\frac{1}{4} - \beta}, c s^{\frac{1}{3} - \beta})}  \left(\int_{r}^{c s^{\frac{1}{3} - \beta}}  Z^{6-a} (1+ Z)^{-m} \; dZ \right)\left( \int_{Ks^{\frac{1}{4} - \beta}}^r Z^a(1+Z)^m \; dZ\right)\\
&\leq & C_{m} s^{8\beta} s^{8(\frac{1}{3} - \beta)}\leq C_m s^{8/3}.
\end{eqnarray*}
It follows that, if $\beta<2/7$, for $s_0$ sufficiently large,
\begin{eqnarray*}
&&M_2b\left(\int_0^{cs^{1/3}} \frac{(\p_Y f)^2}{ U} w \right)^{1/2}\left( \int_{K s^{1/4}}^{cs^{1/3}} Y^4\left(\int_0^Y \frac{f}{U^2}\right)^2 U w \right)^{1/2}\\
&\leq &  C_{a,m} M_2s^{(7-a)\frac{\beta}{2} -1} \left[\int_0^\infty \frac{(\p_Y f)^2}{U} w + \int_0^\infty \frac{f^2}{U^2}w\right] \\
&&+ \delta  \left[\int_0^\infty \frac{(\p_Y f)^2}{U} w + \int_0^\infty \frac{f^2}{U^2}w\right]\\
&&+ C_{\delta,m } b^2 s^{8/3} K^{-8} s^{-2} \int_0^\infty f^2 w\\
&\leq & \delta \left[\int_0^\infty \frac{(\p_Y f)^2}{U} w + \int_0^\infty \frac{f^2}{U^2}w + b \int_0^\infty f^2 w\right].
\end{eqnarray*}

There remains to consider the part of the second lower order term for $Y\geq c s^{1/3}$. Using the cut-off function $\chi$, we have
\begin{eqnarray*}
&&\left|\int_{cs^{1/3}}^\infty (U_{YY}-1) \left(\int_0^Y \frac{f}{U^2}\right)\: \p_Y f w\right|\\&\leq & \delta \int_0^\infty \frac{(\p_Y f)^2}{U} w + \frac{(M_2+1)^2}{4\delta} \int_{cs^{1/3}}^\infty U\left(\int_0^Y \left( 1 - \chi \left(\frac{Y}{s^{1/4}} \right) + \chi \left(\frac{Y}{s^{1/4}} \right) \right) \frac{f}{U^2}\right)^2 w\\
&\leq &  \delta \int_0^\infty \frac{(\p_Y f)^2}{U} w +\frac{C}{\delta}  \left(\int_0^{K s^{1/4}} \frac{|f|}{U^2}\right)^2 \int_{cs^{1/3}}^\infty U w\\&&+ \frac{C}{\delta}   \int_{cs^{1/3}}^\infty U\left(\int_0^Y \left( 1 - \chi \left(\frac{Y}{s^{1/4}} \right)  \right) \frac{f}{U^2}\right)^2 w.
\end{eqnarray*}
The second term in the right-hand side is estimated as above. Notice that for any $P>0$,
$$
\int_{cs^{1/3}}^\infty U w \leq\int_{cs^{1/3}}^\infty \left(Y+ \frac{Y^2}{2}\right) w \leq s^{-P}
$$
provided $m$ is sufficiently large (depending on $\beta$ and $P$). Therefore we obtain, for any $\delta>0$, provided $s\geq s_0$ with $s_0$ sufficiently large,
\begin{multline}\label{in:coercivity3}
\left| \int_0^\infty (U_{YY}-1) \left(\int_0^Y \frac{f}{U^2}\right)\: (f \p_Y w + \p_Y f w)\right| \\\leq \delta \left( \int_0^\infty \frac{(\p_Y f)^2}{U} w + \int_0^\infty \frac{f^2}{U^2}w + b \int_0^\infty f^2 w \right)+ \frac{C}{\delta}   \int_{cs^{1/3}}^\infty U\left(\int_0^Y \left( 1 - \chi \left(\frac{Y}{s^{1/4}} \right)  \right) \frac{f}{U^2}\right)^2 w.
\end{multline}

Gathering \eqref{in:coercivity1}, \eqref{in:coercivity2} and \eqref{in:coercivity3}, we conclude that there exists a positive universal constant $\bar c$ (say for instance $\bar c=1/50$), and $\bar a>0$, $\bar \delta>0$, such that for all $0<a<\bar a$ and for all $\beta \in (1/4, 2/7)$, $m\geq 1$, $\delta\in (0, \bar \delta)$, there exists $s_0>0$ such that if $s\geq s_0$, then
\begin{multline}\label{est:diff-D_1}
\int_0^\infty (\p_Y^2 \LU f) f w\geq \bar c \left( \int_0^\infty \frac{(\p_Y f)^2}{U}w + \int_0^\infty \frac{f^2}{U^2}w \right) - \delta b \int_0^\infty f^2 w\\
- C  \int_{cs^{1/3}}^\infty U\left(\int_0^Y \left( 1 - \chi \left(\frac{Y}{s^{1/4}} \right)  \right) \frac{f}{U^2}\right)^2 w.
\end{multline}

This completes the proof of Lemma \ref{lem:diff1}.

\end{proof}

\subsection{Structure of the commutator}

We record here some formulas and a few estimates that will be useful in the estimation of commutator terms.
We recall that $\mathcal D= \LU(\p_{YY} U-1)= \cLU V + \LU(\p_{YY} \Uapp-1)$, and that a decomposition of $\p_{YY} \Uapp-1$ is given in Lemma \ref{lem:cR}.
We also recall that the commutator $\cC$ is defined by $\cC=[\LU, \p_s + b/2 Y \p_Y] - b \LU$, and that according to Lemma \ref{lem:commutator}, 
$$
\mathcal C[W]= - \left(\cD\int_0^Y \frac{W}{U^2}\right)_Y + 2  \left(U\int_0^Y \frac{W}{U^3}\cD\right)_Y .
$$
\begin{lemma}
Let $W\in \mathcal C^2(\R_+)$ be arbitrary and such that $W= O(Y^2)$ for $Y\ll 1$.
Then
\begin{eqnarray*}
\mathcal C [W]&=& 2 \LU \left(\frac{\mathcal D}{U} W \right) - \left( \frac{\mathcal D}{U}  \int_0^Y \LU W \right)_Y,\\
\p_{YY} \mathcal C [W]&=& \frac{\mathcal D}{U}  \p_Y^2 \LU W + \p_Y \frac{\mathcal D}{U}  \left[ \p_Y \LU W - 2 \frac{U_Y}{U^2} W - 4 U_{YY} \int_0^Y \frac{W}{U^2}\right]\\
&&+ \p_Y^2 \frac{\cD}{U} \left[- 3 \LU W + 2 \frac{W}{U}\right] -  \p_Y^3 \frac{\cD}{U} \int_0^Y \LU W\\
&&+ 2 \p_Y^3 U \int_0^Y \frac{W \cD}{U^3} - 2 \p_Y^3 U \frac{\cD}{U} \int_0^Y \frac{W}{U^2}.
\end{eqnarray*}

\label{lem:commu-formula}
\end{lemma}

\begin{remark}
In the estimations of $E_1$ and $E_2$, we will use the form of $\p_{YY} \mathcal C[W]$ with $W=\p_Y^2 V$ and $W= \p_Y^2 \cLU V$ respectively. Notice in particular that using Corollary \ref{cor:est-cD}, for any weight $w$ with a strong polynomial decay for $Y\geq s^{\beta}$ for some $\beta\in [1/4, 2/7]$,
\begin{multline*}
\left|\int_0^\infty  \frac{\mathcal D}{U}  \left(\p_Y^2 \LU W\right)\left(\p_Y^2 \LU W\right) w \right|\\\leq C b \int_0^\infty \frac{1}{1+Y}\left(\p_Y^2 \LU W\right)^2 w + \left\|  \frac{\mathcal D}{U}  \right\|_{L^\infty(Y\gtrsim s^{1/3})} \int_{cs^{1/3}}^\infty \left(\p_Y^2 \LU W\right)^2 w.
\end{multline*}
In order to control the tail of the integral, we will use lower order estimates. More precisely, we will use a control of $\p_Y W$ in $L^2_{s,Y}$ (with appropriate weights in $s$ and $Y$). We refer to Lemma \ref{lem:tails} for details.

\label{rem:est:commu}
\end{remark}

\begin{proof}

The first formula follows easily by recalling the definition of $\LU$ and noticing that
$$
\int_0^Y \frac{W}{U^2}=\frac{1}{U}\int_0^Y \LU W.
$$
The second formula is a consequence of Lemma \ref{lem:p3LU}. Notice indeed that
\begin{eqnarray*}
\p_Y^2 \LU  \left(\frac{\mathcal D}{U} W \right)&=& \p_Y^3 U \int_0^Y \frac{\cD W}{U^3} + \p_Y^2 U \frac{\cD W}{U^3} - \frac{U_Y}{U^2}\p_Y \left(\frac{\cD W}{U}\right) + \frac{1}{U} \p_Y^2 \left(\frac{\cD W}{U}\right)\\
&=&  \p_Y^3 U \int_0^Y \frac{\cD W}{U^3} + \frac{\cD}{U}\p_Y^2 W - \p_Y^3 U \frac{\cD}{U}\int_0^Y \frac{W}{U^2}\\
&& - \frac{U_Y}{U^2}\p_Y \left(\frac{\cD }{U}\right) W + \frac{2}{U}\p_Y \left(\frac{\cD }{U}\right)\p_Y W + \frac{1}{U}\p_Y^2 \left(\frac{\cD }{U}\right) W.
\end{eqnarray*}
Then, writing
$$
\frac{\p_Y W}{U}= \p_Y \LU W - U_{YY}\int_0^Y \frac{W}{U^2}
$$
and gathering the terms, we obtain the expression announced in Lemma \ref{lem:commu-formula}

\end{proof}

\subsection{Useful inequalities}

\subsubsection{Evaluation of some remainder terms}

In equations \eqref{eq:cLUV}, \eqref{eq:cLU2V}, several terms behave heuristically like 
$$
b \frac{Y}{U} f,
$$ 
where $f=\cLU V$ in \eqref{eq:cLUV} and $f=\cLU^2 V$ in \eqref{eq:cLU2V}. This is for instance the case for the main order commutator term in \eqref{eq:cLUV} (see Remark \ref{rem:est:commu})
or for the remainder term
$$
\frac{b}{2} \cLU \LU \left( Y V - \frac{Y^2}{2} V_Y\right).
$$
(Recall that $\LU$ behaves like a division by $U$ and $\p_Y$ like a division by $Y$.)
Therefore it is useful to have a systematic way to estimate such remainder terms. To that end, let us first recall that the $L^\infty$ estimates given by Proposition \ref{prop:max-princ} ensure that
$$
Y + \frac{Y^2}{2} - \frac{M_2}{12} b Y^4\leq U \leq Y + \frac{Y^2}{2}\quad \forall Y \in [0, cs^{1/3}],
$$
so that there exist constants $C, c >0$ such that
$$
\frac{Y}{U} \leq \frac{C}{1+Y}\quad \forall Y \leq c s^{1/3}.
$$
We then have the following result:
\begin{lemma}
Assume that
$$
\frac{1}{2s}\leq b \leq \frac{2}{s} \quad \forall s\in [s_0, s_1].
$$
Define the weight $w(s,Y):= Y^{-a} (1+ s^{-\beta} Y)^{-m}$ for some $a>0, m>1, \beta>1/4$.

Let $\delta>0$ arbitrary. Then there exists $s_0>0$, depending  on $\delta, \beta, m$, such that the following properties holds:
For all  $f\in L^\infty (\R_+)$ such that $ f= O(Y)$ for $Y\ll 1$, for all $s\geq s_0$,
$$
b \int_0^\infty \frac{1}{1+Y} f^2 w \leq \delta b \int_0^\infty f^2 w + \delta \int_0^\infty \frac{f^2}{U^2} w.
$$

\label{lem:rem-1+Y}

\end{lemma}

\begin{proof}

We split the integral in two, distinguishing between $Y\leq \delta^{-1}$ and $Y\geq\delta^{-1}$. First, it is clear that
$$
\int_{\delta^{-1}}^\infty  \frac{1}{1+Y} f^2 w_1 \leq \delta \int_{\delta^{-1}}^\infty f^2 w_1.
$$
Thus we focus on the set $Y \leq \delta^{-1}$. On this set, we have
$$
\frac{1}{U^2}\geq\frac{1}{Y^2\left(1+\frac{Y}{2}\right)^2},
$$
so that
$$
s \geq 2 \delta^{-3} (1+ \delta^{-1})\Rightarrow\frac{b}{1+Y} \leq \frac{\delta}{U^2}\quad \forall Y \in [0, \delta^{-1}] .
$$

\end{proof}

\subsubsection{Control of integral tails}

\begin{lemma}
Assume that $U$ satisfies the following assumptions:
$$
\ba
U(s,Y)\leq  Y + \frac{Y^2}{2},\\
U_Y \leq 1+Y,\\
|U_{YY}|\leq C.
\ea
$$

Let $p, p_0$ be positive weights given by
$$
p(s,Y)= Y^{-k} (1+ s^{-\beta} Y)^{-m},\quad p_0(s,Y)= Y^{-k_0} (1+ s^{-\beta_0} Y)^{-m_0},
$$
for some $k,k_0\geq 0$, $m\geq m_0$, $\beta<\beta_0$. Let $\alpha_0>\beta_0$.

Let $P>0$ be arbitrary. Then there exists  $m_P\geq 1$ (depending on $\alpha_0, \beta, \beta_0, k, k_0$) such that if  and $m_0, m-m_0\geq m_P$, then for all $\delta>0$:

\begin{itemize}
\item For $a>0$, for all $W\in W^{3,\infty}$ such that $W=O(Y^2)$ for $Y\ll 1$,
\begin{multline*}
\int_{cs^{\alpha_0}}^\infty \left(\p_Y^2 \LU W\right)^2 p \\\leq \delta \int_0^\infty \frac{(\p_Y^3 \LU W)^2}{U} p + \frac{s^{-P}}{\delta}  \left(\int_0^\infty \frac{(\p_Y W)^2}{U} p_0 + \int_0^\infty W^2 p_0 + C_a\int_0^1 \frac{W^2}{Y^{3+a}}\right);
\end{multline*}

\item Let $\alpha_1\leq \beta$, and let $\chi \in \mathcal C^\infty_0(\R)$ be a cut-off function such that $\chi\equiv 1$ in a neighbourhood of zero. Then
\begin{multline*}
\int_{c s^{\alpha_0}}^\infty U \left( \int_0^Y (1-\chi)\left(\frac{Y}{s^{\alpha_1}}\right) \frac{\p_Y^2 \LU W}{U^2}\right)^2 p \\\leq C s^{-P}\left(\int_0^\infty W^2 p_0+ \int_0^1 \frac{W^2}{Y^{3+a}} + \int_0^\infty \frac{(\p_Y W)^2}{U} p_0 \right).
\end{multline*}

\end{itemize}
, 

\label{lem:tails}
\end{lemma}
\begin{remark}
Notice that a similar estimate also holds for quantities such as
$$
\int_{cs^{\alpha_0}}^\infty \left(\p_Y^3 \LU W\right) \left(\p_Y^2 \LU W\right) p.
$$
Indeed, the integral above can be transformed after a straightforward integration by parts into an quantity similar to the one handled in the Lemma (provided $p(cs^{\alpha_0})=0$, which we can always assume without loss of generality, up to the addition of a cut-off function.)

\end{remark}

\begin{remark}
We will use these estimates in the next section and we will make the following choices
\begin{itemize}
\item $\alpha_0=1/3, \alpha_1=1/4$;
\item $W=\p_Y^2 V$, $p=w_1$, $p_0=w_0= Y^{-a} (1+ s^{-\beta_0})^{-m_0}$,  (estimate on $E_1$ from Proposition \ref{prop:est-cLUV-1});
\item $W=\p_Y^2 \cLU V$, $p=w_2$, $p_0=w_1$  (estimate on $E_2$ from Proposition \ref{prop:est-cLU2V}).

%
\end{itemize}

\end{remark}

\begin{proof}
Let $\chi\in C^\infty_0(\R)$ be a non-negative cut-off function such that $\chi\equiv 1$ in $ \subset [-c/2,c/2]$ and $\mathrm{Supp} \;\chi \subset [-c,c]$ . Then
\begin{eqnarray*}
\int_{cs^{\alpha_0}}^\infty \left(\p_Y^2 \LU W\right)^2 p& \leq & \int_0^\infty (1-\chi) \left(\frac{Y}{s^{\alpha_0}}\right)\left(\p_Y^2 \LU W\right)^2 p\\
&=&- \int_0^\infty \p_Y \LU W \:\p_Y \left(\p_Y^2 \LU W (1-\chi) \left(\frac{Y}{s^{\alpha_0}}\right) p \right).
\end{eqnarray*}
We then estimate each term in the right-hand side separately. For instance,
\begin{multline*}
 \int_0^\infty | \p_Y \LU W | \; \left| \p_Y^3 \LU W\right|  (1-\chi)\left(\frac{Y}{s^{\alpha_0}}\right) p\\
 \leq \delta \int_0^\infty \frac{(\p_Y^3 \LU W)^2}{U} p + \frac{1}{4 \delta} \int_0^\infty {(\p_Y \LU W)^2} U (1-\chi) \left(\frac{Y}{s^{\alpha_0}}\right) p.
\end{multline*}
Then, recalling that 
$$
|\p_Y \LU W | \leq \int_0^Y \frac{|W|}{U^2} + \frac{| \p_Y W|}{U},
$$
we infer that
$$
\int_0^\infty {(\p_Y \LU W)^2} U (1-\chi) \left(\frac{Y}{s^{\alpha_0}}\right) p \leq 2\int_{\frac{cs^{\alpha_0}}{2}}^\infty \frac{(\p_Y W)^2}{U} p + \int_{\frac{cs^{\alpha_0}}{2}}^\infty \left( \int_0^Y \frac{|W|}{U^2} \right)^2 U p .
$$
If $Y\gtrsim s^{\alpha_0}$, write
$$
\int_0^Y \frac{|W|}{U^2}\leq C_a\left(\int_0^1 \frac{W^2}{Y^{3+a}}\right)^{1/2} + C \left(\int_0^Y W^2 p_0\right)^{1/2} \left(\int_0^Y p_0^{-1}\right)^{1/2}.
$$
Using the fact that $\|U(s)\|_{\infty}=O(s)$ together with the assumptions on $p$ and $p_0$, we obtain the desired result. 

As for the other term, we have
\begin{eqnarray*}
&&-\int_0^\infty \p_Y \LU W \: \p_Y^2 \LU W\:\p_Y \left( (1-\chi) \left(\frac{Y}{s^{\alpha_0}}\right) p \right)\\
&=& \frac{1}{2}\int_0^\infty (\p_Y \LU W)^2 \p_{YY} \left( (1-\chi) \left(\frac{Y}{s^{\alpha_0}}\right) p \right)\\
&\leq & C s^{-2/3}\int_{cs^{\alpha_0}/2}^\infty (\p_Y \LU W)^2 p,
\end{eqnarray*}
which is evaluated with the same estimate as above.

%

As for the second estimate, notice that thanks to the cut-off function $\chi$, we can integrate by parts without having to deal with boundary terms, so that
\be\label{tails}
\int_0^Y (1-\chi)\left(\frac{Y}{s^{\alpha_1}}\right) \frac{\p_Y^2 \LU W}{U^2}= - \int_0^Y \p_Y \LU W \p_Y \left[ (1-\chi)\left(\frac{Y}{s^{\alpha_1}}\right) \frac{1}{U^2}\right].
\ee
Recall that
$$
\p_Y \LU W = \frac{\p_Y W}{U} + U_{YY} \int_0^Y \frac{W}{U^2}.
$$
Then the integral in the right-hand side of \eqref{tails} is bounded by
$$
\left(\int_0^\infty W^2 p_0+ \int_0^1 \frac{W^2}{Y^{3+a}} + \int_0^\infty \frac{(\p_Y W)^2}{U} p_0 \right)^{1/2} \left(\int_0^Y Q(s,Y) p_0^{-1}\right)^{1/2},
$$
for some function $Q$ that can be computed explicitely and that has at most polynomial growth in $s$ and $Y$. We conclude as before by choosing $m$ sufficiently large.

\end{proof}

\subsubsection{A special case of Hardy inequalities}

We will often use the weighted Hardy inequality from Lemma \ref{lem:Hardy} in the following case:

\begin{lemma}
Let $k\geq 2$ be arbitrary, and let  $w=w_i$ for $i=1,2$. Then there exists a constant $C=C(k,m_i,a)$, independent of $s$, such that the following inequalities hold: if $s$ is large enough, then for any $f\in H^1_{loc}(\R_+)$ such that $f(0)=0$, 
$$
\ba
\int_0^\infty \frac{1}{(1+Y)^k} f^2 w \leq C\int_0^\infty \frac{1}{(1+Y)^{k-2}} (\p_Y f)^2 w,\\
\int_0^\infty \frac{1}{Y^k} f^2 w \leq C\int_0^\infty \frac{1}{Y^{k-2}} (\p_Y f)^2 w\text{ (provided $f=O(Y^{k/2})$ for $0<Y\ll 1$)}.
\ea
$$
\label{lem:Hardy-1}
\end{lemma}
\begin{remark}
Obviously the Lemma can be extended to weights of the form $Y^{-k}(1+Y)^{-l}w$ with $k+l\geq 2$.
\end{remark}

\begin{proof}
We focus on the first inequality, since the second one goes along the same lines (and is in fact slightly easier). Lemma \ref{lem:Hardy} entails that
$$
\int_0^\infty \frac{1}{(1+Y)^k} f^2 w \leq C_H\int_0^\infty \frac{1}{(1+Y)^{k-2}} (\p_Y f)^2 w
$$
where
$$
C_H=4\sup_{r>0}\left(\int_{r}^\infty \frac{w}{(1+Y)^k}\right) \left(\int_0^r (1+Y)^{k-2} w^{-1}\right). 
$$

We distinguish between the cases $r<1$ and $r>1$. If $Y\leq r<1$, then for $s$ large enough $w^{-1}\leq 2 Y^{a}$, and
$$
\left(\int_{r}^\infty \frac{w}{(1+Y)^k}\right) \left(\int_0^r (1+Y)^{k-2} w^{-1}\right)\leq C_{k,a}.
$$
If $r>1$, then writing $\int_0^r=\int_0^1 + \int_1^r$, we obtain
\begin{eqnarray*}
&&\left(\int_{r}^\infty \frac{w}{(1+Y)^k}\right) \left(\int_0^r (1+Y)^{k-2} w^{-1}\right)\\
&\leq & 2\left(\int_{1}^\infty \frac{Y^{-a}}{(1+Y)^k}\right)\left(\int_0^1 (1+Y)^{k-2} Y^a\right) + 2^{k-2}\left(\int_{r}^\infty \frac{w}{Y^k}\right) \left(\int_1^r Y^{k-2} w^{-1}\right)\\
&\leq & C_{k,a} + C_k \left(\int_{s^{-\beta}r}^\infty Z^{-k-a} (1+Z)^{-m }dZ\right) \left(\int_{s^{-\beta}}^{s^{-\beta}r} Z^{k+a-2} (1+Z)^m dZ \right),
\end{eqnarray*}
so that
$$
C_H \leq C_{k,a} + C_k \sup_{r'>0}\left(\int_{r'}^\infty Z^{-k-a} (1+Z)^{-m }dZ\right) \left(\int_{0}^{r'} Z^{k+a-2} (1+Z)^m dZ \right)\leq C_{m,k,a}.
$$

\end{proof}

\section{Sequence of estimates on \texorpdfstring{$V$}{V}}
\label{sec:proof-statements-energy}

The goal of this section is to prove the energy estimates of Propositions \ref{prop:est-cLUV-1} and \ref{prop:est-cLU2V}. To that end, we rely on the transport/diffusion version of the rescaled Prandtl equation, namely \eqref{eq:V-2}.
We will use extensively the tools introduced in section \ref{sec:proof-statements-algebraic}. Throughout the section, we will assume that $U$ satisfies the following pointwise $L^\infty$ estimates: there exists  constants $c, C$ such that for all $Y\in [0, c s^{1/3}]$, for all $s\in [s_0, s_1]$,
\be\label{hyp:U}\ba
Y + \frac{Y^2}{2} - \frac{M_2}{12} b Y^4\leq U(s,Y)\leq Y + \frac{Y^2}{2},\\
1+Y - \frac{M_2}{4} b Y^3 \leq U_Y(s,Y) \leq 1+Y,\\
-M_2 b Y^2\leq U_{YY} -1\leq 0.
\ea
\ee
Furthermore,  we assume that there exists a constant $M_1$ such that 
$$
-M_1\leq U_{YY}\leq 1\quad \forall Y \geq c s^{1/3}
$$

It follows in particular that there exists a universal constant $\bar C$ such that  for all $Y \in [0, c s^{1/3}]$,
\be\label{hyp:infty-V}
|V| \leq \bar C (1+M_2) b Y^4,\ |V_Y| \leq \bar C (1+M_2) b Y^3,\ |V_{YY}| \leq \bar C (1+M_2) b Y^2.
\ee
We will also assume that
\be\label{hyp:b}
 \frac{1- \bar \eps}{s} \leq b(s) \leq \frac{1+ \bar \eps}{s},
\ee
for some small universal $\bar \eps$ ($\bar \eps=1/50$ is sufficient), 
and that
\be\label{hy:b2}
\int_{s_0}^{s_1}s^{13/4}(b_s+b^2)^2ds\leq J.\ee

The $L^\infty$ estimates \eqref{hyp:infty-V} imply in particular the following  estimates, which will be used repeatedly in the sequel
\be
|\cLU V | \leq \bar C (1+M_2)^2b Y,\quad \left| \int_0^Y \cLU V \right| \leq  \bar C (1+M_2) b Y^2 \quad \forall Y \leq c s^{1/3},
\ee
and $\cLU V$ is bounded by a polynomial in $Y$ for $Y \geq s^{1/3}$.

Let us recall the definition of the notation $O_\alpha$ (see \eqref{O-alpha}): there exists a constant $C>0$ such that
$$\ba
|A(s,Y)|\leq C | B(s,Y)| \quad\text{for } Y \leq s^{\alpha}, \\
|A(s,Y)|\leq Q(s,Y)\quad\text{for } Y \geq s^{\alpha},\ea$$
for some function $Q$ that has at most polynomial growth in $s$ and $Y$.

\begin{remark}
This section contains rather heavy calculations: in particular, terms such as $\p_Y \cLU^2 (\LU (L_V Y))$ can be expressed as a linear combination of derivatives of $V$ from order zero up to order 6, with coefficients that are rational expressions involving $U$ and its derivatives. However, most of the terms in this expression will often have the same scaling, in the sense that they can all be bounded in $L^2$ by the leading order term, i.e. the one that has the largest number of derivatives. To obtain such estimates, we use the weighted Hardy inequalities from the previous section (see Lemma \ref{lem:Hardy-1}) together with the pointwise bounds on $U$ \eqref{hyp:U}. For instance, if $w=w_j$, $j=1,2$, then for any $P>0$, provided $m_j$ is large enough,
$$
\left\| \frac{U_Y}{U} \cLU V \right\|_{L^2(w)} \leq C \left\| \frac{1}{1+Y} \cLU V \right\|_{L^2(w)} + C s^{-P}\leq C \| \p_Y \cLU V \|_{L^2(w)} + C s^{-P}
$$
and
$$
\left\| \frac{\p_{YY} V}{U} \right\|_{L^2(w)} = \left\| \cLU V - \frac{U_{Y}}{U}\int_0^Y \cLU V \right\|_{L^2(w)} \leq C \| \cLU V \|_{L^2(w)} + C s^{-P}.
$$

As a consequence, we adopt the following convention: we will write
$$
f=g+\lot \quad \text{in } L^2(w)
$$
if  $f=g+h$ and there exists a constant $C$ such that $\| h\|_{L^2(w)}\leq C \|g\|_{L^2(w)} + C s^{-P}$ for any $M>0$ provided $m_1,m_2$ are chosen large enough.

\label{rem:lot}
\end{remark}

Eventually, let us recall the definition of the different weights that will be used in the sequel. We will use parameters
$$
\beta_2<\beta_1<\beta_0<\frac{2}{7}
$$
and integers $m_2\gg m_1\gg m_0$. For $i\in \{0,1,2\}$, we set $w_i:= Y^{-a} (1+ s^{-\beta_i} Y)^{-m_i}$. The parameter $a>0$ is a fixed number such that $a\in (0, \bar a)$, where $\bar a$ is the universal constant in Lemma \ref{lem:diff1}. The need for the weight $w_0$ is explained in the following remark:

\begin{remark}[Role of the different estimates]
	In this section, we derive estimates on $E_0$, $E_1$, $E_2$, where, for $i=0,1,2$,
	$$\ba
	E_i(s):=\int_0^\infty (\p_Y^2 \cLU^i V(s) )^2 w_i,\\
	D_i(s):= \int_0^\infty \frac{(\p_Y^3 \cLU^i V(s) )^2 }{U} w_i + \int_0^\infty \frac{(\p_Y^2 \cLU^i V(s) )^2}{U^2}w_i.\ea
	$$
	
	\begin{itemize}
\item The estimate on $D_0$ allows us to have an $L^2$ control on $\p_Y^3 V$, which is useful to bound the integral tails stemming from the estimates of $E_1, D_1$ (see Lemma \ref{lem:tails}). Note that the integral tails in the estimate of $(E_0, D_0)$ are merely handled thanks to the pointwise $L^\infty$ estimates \eqref{hyp:U}.

\item The estimate on $E_1$ will be used in the maximum principle argument (see Proposition \ref{prop:max-princ}). 

\item The estimate on $D_1$ will be used (in a quantitative fashion) when we bound the remainder terms on $(E_2, D_2)$. 

\item Eventually, $E_2, E_1$ and $D_1$ will control  $b_s + b^2$ thanks to Lemmas \ref{lem:trace} and \ref{lem:coerciv}.

	\end{itemize}

\end{remark}

Since the equations on $\p_Y^2 \cLU^k V$  for $k=0,1,2$ have the same structure, the estimates on $E_0, E_1, E_2$ go along the same lines. The differences between the three estimates come from the right-hand side of the equation:
\begin{itemize}
\item Commutator terms may or may not be present;
\item The estimates on the remainder terms are different for each energy estimate.
\end{itemize}
The reader who merely wishes to understand the energetic structure of the equation may go through paragraph \ref{ssec:E0} only.

\subsubsection*{Notation} Throughout this section, we will denote by $H_1$ a constant depending only on $a, m_1, \beta_1, M_1, M_2$, and by $H_2$ a constant depending on the same parameters and also on $m_2, \beta_2$. 

\subsection{Preliminary step: estimates on \texorpdfstring{$(E_0, D_0)$}{the first energy}}\label{ssec:E0}

Let us recall that the purpose of this paragraph is to have an $L^2$ control of $\p_Y^3 V$ through $D_0$. 

First, notice that $\p_Y^2 V$ is a solution of
$$
\p_s \p_Y^2 V + \frac{b}{2} Y \p_Y (\p_Y^2 V) - \p_Y^2 \LU \p_Y^2 V = \p_Y^2 \cR.
$$
Consider the weight $w_0:=Y^{-a} (1+ s^{-\beta_0 } Y)^{-m_0}$, for some $a\in ]0, \bar a]$, where $\bar a$ is defined in Lemma \ref{lem:diff1}, $\beta_0\in ]1/4, 2/7[$, and $m_0\geq 1$.  The diffusion term 
$$
-\int_0^\infty \left(\p_Y^2 \LU \p_Y^2 V\right) \; \p_Y^2 V w_0
$$
is handled by Lemma \ref{lem:diff1}, up to a remainder term which we estimate now: we have, for any $P>0$,
\begin{eqnarray*}
&&\int_{c s^{1/3}}^\infty U \left( \int_0^Y (1-\chi) \left(\frac{Y}{s^{1/4}} \right) \frac{\p_Y^2 V}{U^2}\right)^2 w_0\\
&\leq & \bar C (1+M_1)^2 \int_{c s^{1/3}}^\infty \left(Y+ \frac{Y^2}{2}\right) \frac{Y^2}{s^2} w_0\\
&\leq &  \bar C (1+M_1)^2s^{-P}
\end{eqnarray*}
provided $m_0$ is sufficiently large.

 Hence, according to Lemma \ref{lem:diff1}, setting
$$\ba
E_0(s):= \int_0^\infty (\p_Y^2 V)^2 w_0,\\
D_0(s):= \int_0^\infty \frac{(\p_Y^3 V)^2}{U} w_0 + \int_0^\infty \frac{(\p_Y^2 V)^2}{U^2} w_0
\ea
$$
we have, for any $\delta>0$, provided $s_0$ is large enough,
$$
-\int_0^\infty \left(\p_Y^2 \LU \p_Y^2 V\right) \; \p_Y^2 V w_0\geq \bar c D_0(s) - \delta b E_0(s) -  \bar C (1+M_1)^2s^{-P}
$$

Concerning the transport term, we have
$$
\int_0^\infty (\p_s \p_Y^2 V) \p_Y^2 V w_0 = \frac{1}{2} \frac{d}{ds} \int_0^\infty (\p_Y^2 V)^2 w_0 - \frac{1}{2}\int_0^\infty (\p_Y^2 V)^2 \p_s w_0
$$
and
$$
\int_0^\infty Y \p_Y^3 V\;  \p_Y^2 V w_0= -\frac{1}{2} \int_0^\infty (\p_Y^2 V)^2 (Y w_0)_Y.
$$
Combining the two identities and using the expression of $w_0$, we infer that
\begin{eqnarray*}
\int_0^\infty \left(\p_s \p_Y^2 V+ \frac b 2 Y \p_Y (\p_Y^2 V) \right) \p_Y^2 V w_0 &=& \frac{1}{2} \frac{d}{ds} \int_0^\infty (\p_Y^2 V)^2 w_0 - \frac{1-a}{4} b \int_0^\infty (\p_Y^2 V)^2 w_0\\
&&+m \left( \frac{b}{2} - \frac{\beta_0}{2s} \right)\int_0^\infty(\p_Y^2 V)^2 \frac{s^{-\beta_0} Y}{1 + s^{-\beta_0} Y} w_0.
\end{eqnarray*}
Using assumption \eqref{hyp:b} with $\bar \eps<1/7$ and $\beta_0<2/7$, we infer that the last term is non-negative.

It follows that for all $\delta>0$, if $s_0$ is large enough,
$$
\frac{d E_0}{ds} - \frac{1-a-\delta}{2} b E_0 (s ) + 2 \bar c D_0(s) \leq 2 \int_0^\infty \left( \p_Y^2 \cR\right) \p_Y^2 V w_0 +  \bar C (1+M_1)^2s^{-P}.
$$
We now evaluate $\p_Y^2 \cR$. Using Lemma \ref{lem:comp-R0} together with assumption \eqref{hyp:b}, we have
\begin{eqnarray*}
\p_Y^2 \cR&=& (b_s + b^2)\left(\frac{Y^2}{4} + a_4 b Y^5 + 270 a_{10} b^2 Y^8 + 330 a_{11} b^2 Y^9\right)\chi\left(\frac{Y}{s^{2/7}}\right) \\
&&+ O_{2/7} (b^4(Y^8 + Y^9))+ \p_Y^2 \LU \left(O_{2/7}(bY) V + O_{2/7}(b Y^2)V_Y + O_{2/7}(b^4 Y^{11} )\right),
\end{eqnarray*}
where the $O(\cdot)$ in the last term of the right-hand side must be understood\footnote{We say that ``$f=O_\alpha(g)$ in $W^{k, \infty}$" if $\partial_Y^j f= O_\alpha (\partial_Y^j g)$ for $0\leq j \leq k$.} in $W^{2,\infty}$. Recalling the expression of $\p_Y^2 \LU$ (see Lemma \ref{lem:p3LU}) together with the $L^\infty$ estimates \eqref{hyp:infty-V}, we obtain, for $Y\leq c s^{2/7}$,
$$
\p_Y^2 \cR= O((b_s+b^2)Y^2) +  O (b^4(Y^8 + Y^9)) + O \left(\frac{bY^2 \p_Y^3 V}{U} \right) + \p_Y^3 U \int_0^Y \frac{O(b^2 Y^5 )}{U^2}.
$$
All constants appearing in the $O(\cdot)$ depend on $M_1, M_2$.
Notice that for $Y\leq c s^{2/7}$
$$
\int_0^Y\frac{b^2 Y^5 }{U^2}= O(b^2 Y^2)\ll \frac{b Y^2}{U},
$$
and $\p_Y^3 U = \p_Y^3 \Uapp + \p_Y^3 V= \p_Y^3 V + O(bY)$. Hence the last term is smaller than the first two on $Y \leq c s^{2/7}$.

Now, for all $\delta>0$, for $s\geq s_\delta$ large enough and for $m_0\geq 6-a$,
$$
\ba
\int_0^\infty |b_s + b^2| Y^2 |\p_Y^2 V| w_0\leq \delta b E_0 + \frac{H_1}{\delta} (b_s + b^2)^2 s^{1+ (5-a)\beta_0},\\
\int_0^{cs^{2/7}} b^2 Y^2 | \p_Y^2 V| \; w_0 \leq \delta b E_0 + \frac{H_1}{\delta} b^3 s^{(5-a)\beta_0},\\
\int_0^{cs^{2/7}} b Y^2\frac{|\p_Y^3 V|}{U} | \p_Y^2 V| \; w_0 \leq \delta D_0 + \frac{1}{4\delta} \int_0^{cs^{2/7}} b^2 \frac{Y^4 }{U} |\p_Y^2 V|^2 w_0\leq \delta D_0 + \delta b E_0.
\ea
$$
We now focus on the set $Y\geq s^{2/7}$. Recalling that $\beta_0<2/7$, for any function $\phi=\phi(s,Y)$ such that
$$
\phi(s,Y) \leq C Y^{\gamma_1} s^{\gamma_2}\quad \forall Y\geq 0, \ s\geq s_0,
$$
for some explicit $\gamma_1, \gamma_2\geq 0$, we have, for any $P>0$,
$$
\int_{cs^{2/7}}^\infty \phi(s,Y) w_0 \leq s^{-P}
$$
provided $m_0$ is large enough (depending on $\beta_0$, $\gamma_1$, $\gamma_2$ and $P$). We use  a simple Cauchy-Schwartz inequality and control the terms involving $\p_Y^3 V$ by the dissipation term $D_0$. 
Since we have $L^\infty $ estimates on $V, \p_Y V, \p_Y^2 V$, it follows that for all $\delta>0$
$$
\int_{cs^{2/7}}^\infty |\p_Y^2 \cR| \; |\p_Y^2 V|\; w_0 \leq \delta D_0 + \frac{H_1}{\delta} s^{-P}.
$$

Eventually, choosing $\delta$  small enough (depending on $a$), we obtain
$$
\frac{d E_0}{ds} - \frac{1}{2s} E_0 + \bar c D_0 \leq H_1s^{-3+ (5-a) \beta_0} + H_1 (b_s + b^2)^2 s^{1+ (5-a)\beta_0}.
$$
Thus there exists $S_0$, depending on $M_1, M_2,a,\beta_0$, such that if $s\geq S_0$,
\be
\label{est:D_0}
E_0(s) s^{-1/2} + \int_{s_0}^s t^{-1/2} D_0(t)\:dt \leq E_0(s_0) s_0^{-1/2} + H_1 + H_1 J s_0^{-\frac{11}{4} + (5-a)\beta_0}.
\ee
We will use this inequality in the sequel in order to have a control of $\p_Y^3 V$, which is not given by the $L^\infty$ estimates. This will allow us in particular to control the tail of some commutators in the estimate of Proposition \ref{prop:est-cLUV-1}.

\begin{remark}
Notice that the assumption $|\p_Y^2 V(s_0)| \leq 1+M_1$ implies that
$$
E_0(s_0) \leq (1+M_1)^2 s_0^{(1-a)\beta_0}\leq s_0^{1/2}
$$
is $s_0$ is sufficiently large (depending on $M_1$ and $\beta_0$). Therefore there is no need to include any additional assumption on $E_0(s_0)$ in the bootstrap argument.

\end{remark}

\subsection{Estimate on \texorpdfstring{$\p_Y^2\cLU V$}{g}: proof of Proposition \texorpdfstring{\ref{prop:est-cLUV-1}}{}}

We now turn towards the estimate on $\cLU V$. Notice that the main order term in the right-hand side of \eqref{eq:cLUV} is now $(b_s+ b^2) Y$. The lack of decay of this remainder prompts us to differentiate twice \eqref{eq:cLUV}, in order to cancel the linear term. We therefore perform estimates on  $g:=\p_Y^2 \cLU V$. Using Lemma \ref{lem:commutator}, we have
\be\label{eq:g}
\p_s g + 2 b g + \frac{b}{2} Y \p_Y g  -\p_{YY}\LU g
=\p_Y^2 \cLU \mathcal R + \p_Y^2 \mathcal C[\p_Y^2 V].
\ee

We  multiply \eqref{eq:g}  by $g  w_1$, where $w_1=Y^{-a}(1+ s^{-\beta_1} Y)^{-m_1}$ for some $\beta_1 \in ]1/4, \beta_0[$, $a>0$, and $m_1\gg m_0\gg 1$. Using the same computations as in the previous paragraph, we have
$$
\int_0^\infty \left(\p_s g + 2 b g + \frac{b}{2} Y \p_Y g \right) g w_1
\geq  \frac{1}{2}\frac{d}{ds} E_1(s) + \frac{7+a}{4} b E_1(s),
$$
where
\be\label{def:E_1}
E_1(s):=\int_0^\infty g^2 w_1.
\ee
Using Lemma \ref{lem:diff1} and Lemma \ref{lem:tails}, we also have, for any $P>0$, provided $a\leq \bar a$,  $s_0$  large enough (depending on $a$ and $\beta_1$) and $m_1$ large enough,
$$
-\int_0^\infty  \left(\p_{YY}\LU g\right) g w_1 \geq \bar c D_1(s) - \frac{1}{8} b E_1(s) - s^{-P} D_0(s),
$$
where
\be\label{def:D_1}
D_1(s):= \int_0^\infty \frac{(\p_Y g)^2}{U} w_1 + \int_0^\infty \frac{g^2}{U^2} w_1.
\ee

We now treat independently the other terms, namely
\begin{itemize}
\item The  commutator term $\p_Y^2 \mathcal C[\p_Y^2 V]$;
\item The remainder term $\p_Y^2 \cLU \mathcal R$.
\end{itemize}

\subsubsection*{The commutator term}
\label{commu-E1}

The goal of this paragraph is to prove that
for all $\delta,P>0$, there exists  $s_0>0$ such that is $s\geq s_0$,
$$
\left|\int_0^\infty \p_Y^2 \mathcal C[\p_Y^2 V]\; \p_Y^2 \cLU V w_1 \right|\leq \delta \left( b E_1(s) + D_1(s)\right) + s^{-P }D_0(s) + s^{-P}.
$$

We use the formula in Lemma \ref{lem:commu-formula}. 
We recall that $\mathcal D= \LU(\p_{YY} \Uapp -1) + \cLU V$, so that $\mathcal D$ ``contains'' two derivatives of $V$ through the term $\cLU V$. Hence each term in 
$$
\int_0^\infty \p_Y^2 \mathcal C[\p_Y^2 V]\; \p_Y^2 \cLU V w_1
$$
is a product of three derivatives of $V$.

As a consequence, we arrange the terms in $\p_Y^2 \mathcal C[\p_Y^2 V]$ in the following way:
\begin{itemize}
\item The terms with the highest number of derivatives of $V$ are estimated by $E_1$ or $D_1$;
\item The terms with a number of derivatives lower than or equal to 2 are estimated in $L^\infty$ thanks to \eqref{hyp:infty-V}.
\end{itemize}
This strategy will work as long as we do not end up with a product of three terms of the type
$$
\p_Y^{k_1} V \; \p_Y^{k_2} V \; \p_Y^{k_3} V,
$$
with $k_1, k_2, k_3 \geq 3$. Such terms will need to be re-arranged thanks to an integration by parts. However, a quick look at the formula in Lemma  \ref{lem:commu-formula} shows that this situation occurs only for the second term in the right-hand side of the formula giving $\p_Y^2 \mathcal C[\p_Y^2 V]$, namely
\be\label{cD2}
\int_0^\infty \p_Y\frac{\cD}{U}\; \p_Y \cLU V \; \p_Y^2 \cLU V w_1.
\ee
But as we will see, it is easy to overcome the difficulty raised by this term.

We now examine the terms in $\p_Y^2 \mathcal C [\p_Y^2 V]$ one by one.
\begin{itemize}
\item Using the $L^\infty$ estimates \eqref{hyp:infty-V} together with Corollary \ref{cor:est-cD}, it can be easily checked that
 $$
\frac{\cD}{U}  = O_{1/3}\left(  b \frac{1}{1+Y}\right).
 $$
It follows that
$$
\int_0^{cs^{1/3}} \left| \frac{\cD}{U} \right| (\p_Y^2 \cLU V)^2 w_1 \leq H_1 b \int_0^{cs^{1/3}} \frac{1}{1+Y} (\p_Y^2 \cLU V)^2 w_1 ,
$$
which is estimated thanks to Lemma \ref{lem:rem-1+Y}. The part of the integral for $Y\geq c s^{1/3}$ is estimated thanks to Lemma \ref{lem:tails}, with $p=Y^k w_1$ for some integer $k$, $p_0= w_0$ with  $m_1\gg m_0$. It follows from Lemma \ref{lem:rem-1+Y} that for all $\delta>0$, $P\geq 1$, provided $s_0$ is large enough and $m_1$ is large enough,
\be\label{cD1}
\int_0^\infty\left| \frac{\cD}{U} \right| (\p_Y^2 \cLU V)^2 w_1  \leq \delta b E_1 + \delta D_1 +  s^{-P} D_0(s) + s^{-P}.
\ee

\item We then address the problematic term \eqref{cD2}.
We integrate by parts and obtain
\begin{eqnarray*}
I_1&=&
- \int_0^\infty   \frac{\cD}{U}  (\p_Y^2 \cLU V)^2 w_1  - \int_0^\infty \frac{\cD}{U} \p_Y \cLU V \; \p_Y^3 \cLU V w_1\\
&&- \int_0^\infty \frac{\cD}{U}  \p_Y \cLU V \p_Y^2 \cLU V \p_Y w_1.
\end{eqnarray*}
The first term in the right-hand side is the same as in \eqref{cD1}. In the third term, we use the fact that $|\p_Y w_1| \leq C_{a,m} Y^{-1} w_1$ together with a weighted Hardy inequality from Lemma \ref{lem:Hardy-1}, so that
\begin{eqnarray*}
&&\left| \int_0^\infty \frac{\cD}{U}  \p_Y \cLU V \p_Y^2 \cLU V \p_Y w_1\right|\\&\leq& H_1 b \left(\int_0^\infty \frac{(\p_Y^2 \cLU V)^2}{1+Y} w_1\right)^{1/2}\left(\int_0^\infty \frac{(\p_Y \cLU V)^2}{Y^2(1+Y)} w_1\right)^{1/2} + s^{-P} D_0 + s^{-P}\\
&\leq& H_1 b\int_0^\infty \frac{(\p_Y^2 \cLU V)^2}{1+Y} w_1+ s^{-P} D_0 + s^{-P}.
\end{eqnarray*}
We eventually evaluate the last term in $I_1$. We have, for all $\delta>0$
$$
\left| \int_0^\infty \frac{\cD}{U} \p_Y \cLU V \; \p_Y^3 \cLU V w_1\right| \leq \delta D_1 + \frac{1}{4 \delta}\int_0^\infty \frac{\cD^2}{U} (\p_Y \cLU V)^2 w_1.
$$
Using Corollary \ref{cor:est-cD}, we have 
$$\frac{\cD^2}{U} = O_{1/3}\left(\frac{b^2 Y}{1+Y}\right)= O_{1/3}\left(\frac{b }{Y^2(1+Y)}\right).$$ 
Using a Hardy inequality from Lemma \ref{lem:Hardy-1} together with Lemma \ref{lem:rem-1+Y}, we deduce that the part of the integral bearing on $Y\lesssim s^{1/3}$ is lower than $\delta D_1 + \delta b E_1$ for $s\geq s_0$ large enough. The part of the integral bearing on $Y\gtrsim s^{1/3}$ is handled thanks to Lemma \ref{lem:tails}, recalling the form of $\p_Y \LU$ (see Lemma \ref{lem:p3LU}). It is therefore smaller than $\delta D_1 + s^{-P} D_0 + s^{-P}$, for any $\delta,P>0$, provided $s_0$ is large enough.

\item We then treat simultaneously the next three terms, namely
\be\label{cD3}\ba
\int_0^\infty \p_Y \frac{\cD}{U} \left( - 2 \frac{U_Y}{U^2} \p_Y^2 V - 4 U_{YY} \int_0^Y \frac{\p_Y^2 V}{U^2}\right) \p_Y^2 \cLU V w_1,\\
\int_0^\infty\p_Y^2 \frac{\cD}{U} \left[ - 3 \cLU V + 2 \frac{\p_Y^2 V}{U}\right]  \p_Y^2 \cLU V w_1\text{ and } \int_0^\infty\p_Y^3 \frac{\cD}{U} \left(\int_0^Y \cLU V\right)\p_Y^2 \cLU V w_1.\ea
\ee
The overall idea  is to decompose $\p_Y^k (\cD/U)$ for $k=1,2,3$ into a part that is controlled in $L^\infty$ and a part that involves derivatives of $V$ of order 3 or higher (or equivalently, derivatives of $\cLU V$ of order one or higher). Concerning the part of $\p_Y^k \cD/U$ that is controlled in $L^\infty$, we  use weighted Hardy inequalities to upper-bound $\p_Y^2 V, \cLU V$, etc. by $\p_Y^2 \cLU V$ in $L^2$. As for the part of $\p_Y^k(\cD/U)$ that is not controlled in $L^\infty$, we observe that in the three terms in \eqref{cD3}, $\p_Y^2 V, \cLU V$ and $\int_0^Y \cLU V$ are controlled in $L^\infty$, and we use this $L^\infty$ control to conclude.

Let us now be more specific: it can be easily checked that for $k=1,2,3$,
$$
\p_Y^k \frac{\cD}{U}=O_{1/3}\left(b \frac{1}{Y^k(1+Y)}\right) +  O_{1/3}\left(\frac{1}{Y(1+Y)}\right)\p_Y^k \cLU V+\lot
$$
Since we also have $L^\infty$ estimates on $\p_Y^2 V$, $U_Y, U_{YY}$ and $U^{-1}$, using Lemma \ref{lem:tails}, we infer that the part of the integrals in \eqref{cD3} bearing on $Y\gtrsim s^{1/3}$ is  bounded by $s^{-P} + s^{-P} D_0(s) + \delta D_1$ for some $P>0$ arbitrary provided $m_1$ is large enough. 

We now address the part of the integral bearing on $Y\lesssim s^{1/3}$, and we start with  the part of $\p_Y^k (\cD/U) $ that is bounded in $L^\infty$. We focus on the first integral in \eqref{cD3}, since the other two are treated in a similar fashion. We recall that
$$
\p_Y^2V= U \cLU V - U_Y \int_0^Y \cLU V,\quad \int_0^Y \frac{\p_Y^2 V}{U^2} = \frac{1}{U}\int_0^Y \cLU V,
$$
and therefore
$$
- 2 \frac{U_Y}{U^2} \p_Y^2 V - 4 U_{YY} \int_0^Y \frac{\p_Y^2 V}{U^2}= O_{1/3}(Y^{-1}) \cLU V + \lot
$$
Using several weighted Hardy inequalities (see Lemma \ref{lem:Hardy-1}), it follows that
\begin{multline*}
\left| \int_0^{c s^{1/3}} O_{1/3}\left(b \frac{1}{Y(1+Y)}\right)  \left(- 2 \frac{U_Y}{U^2} \p_Y^2 V - 4 U_{YY} \int_0^Y \frac{\p_Y^2 V}{U^2}\right) \p_Y^2 \cLU V w_1\right|\\
\leq  H_1 b \int_0^{c s^{1/3}} \frac{1}{1+Y} | \p_Y^2 \cLU V|^2 w_1.
\end{multline*}

We then focus on the part of $\p_Y (\cD/U)$ that is not controlled in $L^\infty$, and that involves $\p_Y \cLU V$. For that part, we use the $L^\infty$ estimates on $V$, which entail
$$
\left| - 2 \frac{U_Y}{U^2} \p_Y^2 V - 4 U_{YY} \int_0^Y \frac{\p_Y^2 V}{U^2}\right| \leq H_1 b \quad \forall Y\leq c s^{1/3}.
$$
It follows that 
\begin{multline*}
\left| \int_0^{c s^{1/3}}  O_{1/3}\left(\frac{1}{Y(1+Y)}\right){\p_Y \cLU V}\left(- 2 \frac{U_Y}{U^2} \p_Y^2 V - 4 U_{YY} \int_0^Y \frac{\p_Y^2 V}{U^2}\right) \p_Y^2 \cLU V w_1\right|\\
\leq  H_1 b \int_0^{c s^{1/3}} \frac{1}{Y (1+Y) }| \p_Y \cLU V|\; | \p_Y^2 \cLU V| w_1.
\end{multline*}
We conclude once again using a weighted Hardy inequality.

We then treat the other two integrals in \eqref{cD3} using similar arguments. We conclude that for any $\delta, P>0$, provided $m_1\gg m_0$ and $s_0$ is large enough, the  three integrals in \eqref{cD3} are bounded by
$$
\delta D_1 +\frac{H_1}{\delta}\left( b \int_0^\infty \frac{1}{1+Y} | \p_Y^2 \cLU V|^2 w_1+ s^{-P} + s^{-P} D_0 \right),
$$
where the term $D_1$ stems from the bound of the third integral in \eqref{cD3}, which involves $\p_Y^3 \cLU V$.

\item Eventually, we address
\be\label{cD4}
2\int_0^\infty \p_Y^3 U \left(  \int_0^Y \frac{\p_Y^2 V \cD}{U^3} - \frac{\cD}{U} \int_0^Y \frac{\p_Y^2 V}{U^2} \right) \p_Y^2 \cLU V\; w_1.
\ee
Now, $\cD$ is controlled in $L^\infty$. As in the previous step, we decompose $\p_Y^3 U$ into a part that is controlled in $L^\infty$ and a part over which we have no $L^\infty$ control. More precisely,
$$
\p_Y^3 U = U \p_Y \cLU V + O_{1/3}(b(Y+Y^2)) .
$$
Let us start with the contribution of $ O_{1/3}(b(Y+Y^2))$. For that part, we use the control of $V$ and $\cD$  in $L^\infty$ to prove that the integral tails for $Y\geq c s^{1/3}$ are $O(s^{-P})$, and Hardy inequalities on the set $Y\leq c s^{1/3}$. We infer that the contribution of this part of the integral to \eqref{cD4} is bounded by
$$
 H_1 b E_1^{1/2} \left(\int_0^\infty \frac{1}{1+Y} (\p_Y^2 \cLU V)^2 w_1 +  s^{-P}\right).
$$
We now address the part of \eqref{cD4} where $\p_Y^3 U$ is replaced by $U \p_Y \cLU V$. For that part, we use the $L^\infty$ control of $\p_Y^2  V$, $\cD$ and $U$ in $L^\infty$, together with Lemma \ref{lem:tails} and the control of $D_0$ to estimate the tails. We infer that  this part of \eqref{cD4} is bounded by
$$
\int_0^{s^{2/7}} b^2 Y^2 |\p_Y \cLU V| \; | \p_Y^2 \cLU V | \; w_1 +  s^{-P} D_0 + s^{-P}.
$$
Now, for $Y\leq C s^{2/7}$, we have $b^2 Y^2 \leq C b Y^{-1} (1+Y)^{-1/2}$, so that the integral above is bounded by
$$
H_1 b E_1^{1/2} \left(\int_0^{\infty} \frac{1}{1+Y} (\p_Y^2 \cLU V)^2 w_1\right)^{1/2}.
$$

\end{itemize}

Gathering the estimates above and using Lemma \ref{lem:rem-1+Y}, we conclude that for any $P,\delta>0$, if $m_1, s_0$ are large enough,
the total commutator term satisfies
\be\label{est:commu-E_1}
\left|\int_0^\infty \p_Y^2 \mathcal C[\p_Y^2 V]\; \p_Y^2 \cLU V w_1 \right| \leq \delta b E_1(s) + \delta D_1(s) +  s^{-P} + s^{-P} D_0(s).
\ee

\subsubsection*{The remainder term}

We now evaluate
$$
\int_0^\infty \left(\p_Y^2 \cLU \cR\right) \left(\p_Y^2 \cLU V\right)  w_1.
$$
We claim that for all $\delta>0$, for all $P>0$, provided $m_1$ is large enough and $s_0$ is large enough,
\begin{multline*}
\left|\int_0^\infty \left(\p_Y^2 \cLU \cR\right) \left(\p_Y^2 \cLU V\right)  w_1 \right|\\\leq \delta\left(D_1 + b E_1\right) + \frac{H_1}{\delta}\left(s^{-P } D_0 + s^{-7+(11-a)\beta_1} + (b_s + b^2)^2 s^{-3 + (11-a)\beta_1}\right).
\end{multline*}

We follow the decomposition of Lemma \ref{lem:reste} and write $\cR=\sum_{i=1}^4 \cR_i$, as suggested in remark \ref{rem:decomposition-D}.
\begin{itemize}
\item Recalling that $a_4=1/48,\ a_7=a_4/84$, and $a_{10}, a_{11}$ are defined by \eqref{def:a10a11}, we have
\begin{eqnarray*}
\p_Y^2\cR_1&=& (b_s+ b^2)\left(\frac{1}{4} Y^2 + a_4 b Y^5 + \frac{81}{16} a_7 b^2 Y^8 + \frac{9}{10} a_7 b^2 Y^9\right) \chi\left(\frac{Y}{s^{2/7}}\right) \\&&+ \tilde P_1(s,Y) (1-\chi_1)\left(\frac{Y}{s^{2/7}}\right),
\end{eqnarray*}
for some function $\tilde P_1$ that has at most polynomial growth in $s$ and $Y$, and some cut-off functions $\chi, \chi_1\in \mathcal C^\infty_0(\R_+)$ such that $\chi, \chi_1\equiv 1$ in a neighbourhood of zero.
Using the identity \eqref{LUY2}, we have, up to terms supported in $Y\gtrsim s^{2/7}$,
\begin{eqnarray*}
&&\cLU \cR_1\\&=& \frac{1}{2} (b_s + b^2) Y +\LU \left[ \left(\frac{a_4}{2}(b_s + b^2) b Y^5 \right)\chi\left(\frac{Y}{s^{2/7}}\right)\right]\\
&-&(b_s + b^2)\LU \left( \left(\frac{61}{16} a_7 b^2 Y^8  - \frac{9}{10}a_7 b^2 Y^9+2 a_{10} b^3 Y^{10} + \frac{9}{4} a_4 b^3 Y^{12} \right)\chi\left(\frac{Y}{s^{2/7}}\right) \right)
\\&-& \frac{1}{2}(b_s + b^2) \LU L_V Y - \LU \left[\tilde P_1(s,Y) (1-\chi_1)\left(\frac{Y}{s^{2/7}}\right)\right],
\end{eqnarray*}
and therefore
\begin{eqnarray}\nonumber
&&\p_Y^2\cLU \cR_1 (b_s + b^2)^{-1}\\&=&\p_Y^2\LU \left[ \left(\frac{a_4}{2} b Y^5 -\frac{61}{16} a_7 b^2 Y^8 + \frac{9}{10} a_7 b^2 Y^9-2 a_{10} b^3 Y^{10} - \frac{9}{4} a_4 b^3 Y^{12}\right)\chi\left(\frac{Y}{s^{\frac{2}{7}}}\right)\right]\nonumber\\
&-& \frac{1}{2} \p_Y^2\LU L_V Y - \p_Y^2\LU \left[\tilde P_1(s,Y)  (1-\chi_1)\left(\frac{Y}{s^{\frac{2}{7}}}\right)\right].\label{dec:R1}
\end{eqnarray}
Recalling the expression of $\p_Y^2 \LU$ from Lemma \ref{lem:p3LU}, we infer that for $k\geq 5$,
$$
\p_Y^2 \LU Y^k= \p_Y^3 V \int_0^Y \frac{Y^k}{U^2} + O_{1/3} \left( \left\{ \begin{array}{l}
Y^{k-3}\text{ if } Y \ll1 ,\\
Y^{k-4}\text{ if } 1 \lesssim Y .
\end{array}\right.\right)
$$
In a similar way, we have
\begin{eqnarray}
\label{dec:LV}\p_Y^2 \LU (L_V Y)&=&\p_Y^3 U \int_0^Y \frac{L_V Y}{U^2} - \frac{Y^2}{2}\p_Y\cLU V  \\
\nonumber&&+ V_{YY}\frac{Y^2 U_Y - 2 YU}{2 U^2} + V_Y \frac{2U - Y^2 U_{YY}}{2U^2}+ V \frac{U_{YY} Y - U_Y}{U^2} \\
\nonumber&&+ \frac{Y^2U_{YY}}{2U} \int_0^Y \cLU V.
\end{eqnarray}
We write $\p_Y^3 U = \p_Y^3 \Uapp + \p_Y^3 V$ and replace $\p_Y^3 V$ by the formula \eqref{pkV} in Appendix A. Using the $L^\infty $ estimate on $\p_Y^2 V$,
we obtain
$$
\p_Y^2 \LU (L_V Y)=O_{2/7}(Y^2) \p_Y \cLU V + O_{2/7}(Y) \cLU V + O_{2/7}(1)\int_0^Y \cLU V+\lot
$$
Gathering all the terms, we get, for $Y\gg 1$,
$$
\p_Y^2 \cLU \cR_1= (b_s + b^2)\left(O_{2/7} (b Y) +  \sum_{i=0}^2 O_{2/7}(Y^i) \p_Y^i \int_0^Y  \cLU V +\lot\right).
$$

 Then, noticing that for $Y\leq s^{2/7}$ we have $bY^2 \leq Y^{-1} (1+Y)^{-1/2}$ and using Hardy inequalities from Lemma \ref{lem:Hardy-1}, we infer that for any $P>0$, provided $m_1$ and $S_0$ are large enough,
\begin{eqnarray*}
&&\left| \int_0^\infty \p_Y^2 \cLU \cR_1\; \p_Y^2 \cLU V w_1\right| \\
&\leq & H_1 |b_s+ b^2| E_1^{1/2} \left(b s^{(3-a)\beta_1/2}  + b^2 s^{(11-a)\beta_1/2}\right)\\&&+ H_1 s |b_s + b^2|  E_1^{1/2}\left(\int_0^{s^{2/7} } \frac{1}{Y^2(1+Y)} (\p_Y \cLU V)^2 w_1 +  s^{-P} + s^{-P} D_0\right)^{1/2}\\
&\leq & \delta b E_1 + \frac{H_1}{\delta} (b_s+b^2)^2s^{-3 + (11-a )\beta_1} \\&&+ \frac{H_1}{\delta} s^3 |b_s + b^2|^2 \int_0^\infty \frac{1}{1+Y} (\p_Y^2 \cLU V)^2 w_1 + s^{-P} D_0.
\end{eqnarray*}

\item We use the same type of estimates for the term $\cR_2$, and we find
\begin{eqnarray*}
&&\left| \int_0^\infty \p_Y^2 \cLU \cR_2\; \p_Y^2 \cLU V w_1\right|\\ &\leq& H_1 b^4E_1^{1/2}s^{(11-a)\beta_1/2}  + \delta b E_1\\&\leq& \delta b E_1 + \frac{H_1}{\delta}s^{-7+(11-a) \beta_1}+ s^{-P} D_0.
\end{eqnarray*}

\item We then address the term $\cR_3$. We recall that using \eqref{dec:LV} and \eqref{lem:p3LU}, for $Y\leq c s^{1/3}$, for $k\in \{2, 3, 4\}$
\begin{eqnarray*}
\p_Y\LU(L_V Y)&=& O_{1/3}(bY^2),\quad \LU(L_V Y)= O_{1/3}(bY^3),\\
\p_Y^k \LU(L_V Y)&=& \sum_{i=0}^{k} O_{1/3}(Y^{i-k+2})\p_Y^i \int_0^Y \cLU V + O_{1/3}\left( \frac{1}{Y^{k-1}(1+Y)}\right) V_Y +\lot
\end{eqnarray*}
Notice also that $b^3Y^7\lesssim b Y$  for $Y\leq c s^{1/3}$, so that we can treat $b^3 L_V Y^7$ as a perturbation of $b L_V Y$. In a similar way, for $k\in\{0,\cdots 3\}$,
$$
\p_Y^k\left(-a_4 b Y^4 - a_7 b^2 Y^7 + a_{10} b^3 Y^{10} + a_{11} b^3 Y^{11}\right)= O (bY^{4-k})\quad \text{for } Y \lesssim s^{2/7},
$$
so that
$$
b^3\p_Y^2 \LU L_{-a_4 b Y^4 - a_7 b^2 Y^7 + a_{10} b^3 Y^{10} + a_{11} b^3 Y^{11}} Y^7= O_{2/7}(b^4 Y^7) + O_{2/7}(b^4 Y^8)\p_Y^3 V.
$$
 It follows that
\begin{eqnarray*}
\p_Y^2 \cLU \cR_3 &=& \sum_{i=0}^{4} O_{2/7}(b Y^{i-4})\p_Y^i \int_0^Y \cLU V  \\
&&+ O_{2/7}\left(b^2 (Y+Y^2)\right) \p_Y \cLU V + O_{2/7} (b^2)\int_0^Y \cLU V\\
&&+ O_{2/7}\left(\frac{b}{Y^4(1+Y)^2} \right) V_Y +\lot + O_{2/7}(b^4 Y^3).
\end{eqnarray*}
Using Hardy inequalities together with Lemma \ref{lem:tails}, it is easily proved that for $0\leq i\leq 4$, for any $P>0$, provided $m_1$ is large enough,
$$\ba
\left| \int_0^\infty O_{2/7}(b Y^{i-4})\p_Y^i \int_0^Y \cLU V \; \p_Y^2 \cLU V w_1\right| \leq \delta(bE_1 + D_1)+ s^{-P} D_0(s) + s^{-P},\\
\left| \int_0^\infty O_{2/7}\left(b^2 (Y+Y^2)\right) \p_Y \cLU V  \; \p_Y^2 \cLU V w_1\right| \leq \delta bE_1 + s^{-P} D_0(s) + s^{-P},\\
\left| \int_0^\infty O_{2/7}(b^2)\int_0^Y \cLU V  \; \p_Y^2 \cLU V w_1\right| \leq \delta bE_1 + s^{-P} D_0(s) + s^{-P},
\ea
$$
and
\begin{eqnarray*}
&&\left| \int_0^\infty O_{2/7}\left( \frac{b}{Y^4(1+Y)^2}\right) V_Y\; \p_Y^2 \cLU V w_1 \right| \\&\leq& H_1 b \left(\int_0^\infty \frac{1}{1+Y} (\p_Y^2 \cLU V)^2 w_1\right)^{1/2}\left(\int_0^\infty \frac{1}{Y^8(1+Y)^3} V_Y^2 w_1 +  s^{-P}\right)^{1/2}\\
&\leq &  H_1 b \int_0^\infty \frac{1}{1+Y} (\p_Y^2 \cLU V)^2 w_1+  s^{-P}.
\end{eqnarray*}
Using Lemma \ref{lem:rem-1+Y},  we end up with
$$
\left| \int_0^\infty \p_Y^2 \cLU \cR_3\; \p_Y^2 \cLU V\; w_1\right| \leq \delta\left(D_1 + b E_1\right) +  \left(s^{-P } D_0 + s^{-P}\right).
$$

\item The term $\cR_4$ is easily treated thanks to Lemma \ref{lem:tails}. More precisely, using Appendix A, it can be proved that
$$
\p_Y^2 \cLU \cR_4 = \left(P_1(s,Y) + P_2(s,Y) \p_Y \cLU V + P_3(s,Y)\p_Y^2 \cLU V\right) (1-\bar \chi)\left(\frac{Y}{s^{2/7}}\right),
$$
where $P_1,P_2,P_3$ are functions that have at most polynomial growth in $s$ and $Y$, and $\bar \chi\in \mathcal C^\infty_1$ is identically equal to 1 in a neighbourhood of zero. We infer that for any $P>0$, provided $m_1$ and $s_0$ are large enough,
$$
\left|\int_0^\infty (\p_Y^2 \cLU \cR_4) \p_Y^2 \cLU V w_1\right| \leq C s^{-P} + s^{-P} D_0.
$$

\end{itemize}
We now gather all the terms. Notice that since $\beta_1>1/4$, $5+(3-a) \beta_1<-7+ (11-a) \beta_1$.  We end up with the following estimate: for all $\delta, P>0$, if $m_1$ is large enough,
there exists a constant $H_1$, depending on $M_1, M_2,a,\beta_1$ and $m_1$ and a constant $S_0$, depending on the same parameters and also on $\delta$, such that if $s\geq S_0$,
 then
\begin{multline}
\label{est:rem-g}
\left| \int_0^\infty \left(\p_Y^2 \cLU \cR\right)\: \p_Y^2 \cLU V \; w_1 \right|\\\leq \delta\left(D_1 + b E_1\right) + 
\frac{H_1}{\delta}\left[(b_s + b^2)^2 s^{-3 + (11-a)\beta_1} + s^3 (b_s+b^2)^2 E_1 + s^{-7+(11-a)\beta_1}\right] +s^{-P } D_0.
\end{multline}

\subsubsection*{Conclusion}
Gathering the estimates \eqref{est:commu-E_1} and \eqref{est:rem-g}, we infer that for $a>0$ sufficiently small, for any $P>0$, for $1/4<\beta_1<2/7$ and for $m_1$ sufficiently large, $s_0\leq S_0$, we have, for $s\in [s_0, s_1]$,
$$
\frac{d}{ds} E_1 + \left(\frac{7}{2}b - H_1 s^3 (b_s + b^2)^2 \right)  E_1 + \bar c D_1\leq  H_1(b_s+ b^2)^2 s^{-3+(11-a) \beta_1}+H_1 s^{-7+(11-a)\beta_1} + s^{-P} D_0.
$$
Since $\beta_1>1/4$, we have
$$
6-(11-a)\beta_1<\frac{7}{2},
$$
and therefore the rate of convergence is limited by the size of the right-hand side. 
Let $\phi_1(s):= H_1 \int_{s_0}^s \tau^3 (b_\tau + b^2)^2 d\tau $. The assumptions of the Lemma entail that $0\leq \phi_1(s)\leq H_1 J s_0^{-1/4}$ for all s$\in [s_0, s_1]$. As a consequence, if $s_0\geq J^4$, we have $0\leq \phi_1(s)\leq H_1$ for all $s\in [s_0, s_1]$.
Using a Gronwall type argument and using the preliminary estimate on $D_0$, we obtain, for all $\alpha<6-(11-a)\beta_1$,
\begin{multline*}
E_1(s)s^\alpha \exp(-\phi_1(s))+ \bar c \int_{s_0}^s \tau^\alpha \exp(-\phi_1(\tau))D_1(\tau)d\tau \\
\leq E_1(s_0) s_0^\alpha + H_1 J s_0^{-1/4} + \frac{H_1}{6-(11-a)\beta_1 - \alpha} s_0^{\alpha-6+ (11-a \beta_1)}.
\end{multline*}
Hence, for $s_0\geq \max(S_0, J^4)$ we obtain (up to a new definition of the constant $H_1$)
\[
E_1(s) s^{\alpha} + \int_{s_0}^s D_1(\tau ) \tau^\alpha d\tau \leq H_1 \left(1 + E_1(s_0) s_0^\alpha\right).
\]
This completes the proof of the Proposition.\qed

We also have the following
\begin{corollary}
Under the assumptions of Proposition \ref{prop:est-cLUV-1}, we get the following refined $L^\infty$ estimates on $V$: for all $a\in (0, \bar a)$, for all $\beta_1$ such that
$$
\frac{1}{4}<\beta_1\leq \frac{1}{4}\frac{11}{11-a},
$$
there holds
$$
\ba
|\p_Y \cLU V|= O_{\beta_1}\left( s^{-13/8} Y^{\frac{1+a}2}\right),\\
| \cLU V |=O_{\beta_1}\left( s^{-13/8} Y^{\frac{3+a}2}\right),\quad \int_0^Y \cLU V=O_{\beta_1}\left( s^{-13/8} Y^{\frac{5+a}2}\right),\\
|\p_Y^3 V| =O_{\beta_1}\left(C_1 s^{-13/8}Y^{\frac{3+a}2} (1+Y)\right).
\ea
$$
Note that the constants in the $O_{\beta_1}$ depend on $M_1, M_2, m_1, a, \beta_1$ and $J$.

As a consequence, setting $C_0=E_1(s_0) s_0^\alpha$, where $\alpha$ is such that $13/4<\alpha<6-(11-a)\beta_1$, we infer that there exists a universal constant $\bar H$ and an constant $C_{m_1}$ depending only on $m_1$ such that if $s_0\geq \max (S_0, J^4, C_0^8)$, 
$$
|\p_Y^3 U| \leq \bar H b Y \quad \forall Y\in [0, C_{m_1} s^{\beta_1}].
$$
Furthermore, there exists a constant $C_{a,m_1}$ depending only on $a$ and $m_1$, such that
\be\label{est:non-loc}
\ba
\left|\int_0^Y \frac{\p_Y^2 \cLU V}{U^2} \right|\leq C_{a,m_1} D_1^{1/2} (1+ s^{-\beta_1} Y)^\frac{4+a+m_1}2= O_{\beta_1}(D_1^{1/2}),\\
\left|\p_Y^2 \cLU V \right|\leq C_{a,m_1} D_1^{1/2} (1+ s^{-\beta_1} Y)^\frac{m_1}2(1+Y)^{1/2}Y^{(2+a)/2} = O_{\beta_1}\left(D_1^{1/2} (1+Y)^{1/2}Y^{(2+a)/2} \right).\\
\ea\ee
In particular,
$$
\p_Y \cLU^2 V = \frac{\p_Y^3 \cLU V} U + O_{\beta_1}(D_1^{1/2})
$$

\label{cor:p3U}
\end{corollary}
\begin{proof}
As mentioned in Remark \ref{rem:E_1-13/4} we can pick $a>0$, $\beta_1\in (1/4, 2/7)$ ($\beta_1$ depends on $a$) so that
$$
\frac{13}{4}= 6 - \frac{11}{4}<6-(11-a)\beta_1.
$$
With this choice of $a$ and $\beta_1$, we have
$$
E_1(s)\leq H_1 (1+ C_0) s^{-13/4}.
$$
A simple Cauchy-Schwartz inequality entails
\be\label{CS}
|\p_Y\cLU V| = \left| \int_0^Y \p_Y^2 \cLU V\right| \leq E_1^{1/2}\left(\int_0^Y \frac{1}{w_1}\right)^{1/2}.
\ee
Now, setting $C_{m_1}:= 2^{1/m_1}-1\leq 1$, it is easily checked that for $Y\leq C_{m_1} s^{\beta_1}$, we have
$$
\frac{1}{w_1} \leq 2 Y^a.
$$
The estimates follow, using the formula in equation \eqref{pkV} for the one of $\p_Y^3 V$. Note in particular that for $Y\leq C_{m_1} s^{\beta_1}$,
\begin{eqnarray*}
|\p_Y^3 U|& \leq& | \p_Y^3 \Uapp| + \left(Y+\frac{Y^2}{2} \right) |\p_Y \cLU V + \int_0^Y |\cLU V|\\
&\leq & \bar H\left( b Y + (Y+Y^2)  Y^{\frac{a+1}{2}} H_1^{1/2} (1+C_0)^{1/2} s^{-13/8}\right)\\
&\leq&\bar H b Y \left(1 + H_1^{1/2} (1+C_0)^{1/2}  s_0^{\frac{a+3}{2}\beta_1 - \frac{5}{8}  } \right). 
\end{eqnarray*}
Now, for $\beta_1<2/7$ and $a$ sufficiently small, $\beta_1<1/(3+a)$, so that, if $s_0\geq \max (C_0^8, H_1^8)$,
\[
|\p_Y^3 U|\leq \bar H b Y \left(1 + H_1^{1/2} (1+C_0)^{1/2} s_0^{-1/8} \right)\leq \bar H b Y.
\]
Since $\p_Y^3 V$ has  polynomial growth in $s$ and $Y$ according to \eqref{CS}, we obtain the estimate on $\p_Y^3 U$.

The  two estimates from \eqref{est:non-loc} follow from the Cauchy-Schwartz inequality (see also Remark \ref{rem:diff-zero} ): observe that for $Y \leq 1$,
$$
\left|\int_0^Y\frac{\p_Y^2 \cLU V}{U^2} \right|\leq \left(\int_0^1 \frac{(\p_Y^2 \cLU V)^2}{Y^{3+a}}\right)^{1/2} \left(\int_0^Y \frac{Y^{3+a}}{U^4}\right)^{1/2}\leq C_a D_1^{1/2}.
$$
The same type of estimate holds when $Y\geq 1$. As for the estimate on $\p_Y^2 \cLU V$, we have
$$
\left|\p_Y^2 \cLU V\right|=\left|\int_0^Y \p_Y^3 \cLU V \right|\leq D_1^{1/2} \left(\int_0^Y \frac{U}{w_1}\right)^{1/2},
$$
which leads to the result.

\end{proof}

%
%
%
%
%

We end this paragraph with a short proof for Lemma \ref{lem:trace-cLU2V}: differentiating once equation \eqref{eq:cLUV}, we have
$$
\p_s \p_Y \cLU V + \frac{3b}{2} \p_Y \cLU V + \frac{b}{2} Y \p_Y^2 \cLU V - \p_Y \cLU^2 V = \p_Y \cLU \cR + \p_Y \mathcal C[\p_Y^2 V].
$$
We now take the trace of the above equality at $Y=0$. Since $V=O(Y^7)$ for $Y\ll 1$ by definition of the approximate solution $\Uapp$, we have
$$
\p_Y \cLU V_{|Y=0}= \p_Y^2 \cLU V_{|Y=0}=0,
$$
as well as $\p_Y \mathcal C[\p_Y^2 V]_{|Y=0}=0$. Hence there remains only
$$
\p_Y \cLU^2 V_{|Y=0}= - \sum_{i=1}^4 \p_Y \cLU \cR_{i|Y=0}.
$$
Once again, it can be easily checked that $\p_Y \cLU \cR_{i|Y=0}=0$ for $i=2,3,4$, and that $$\p_Y \cLU \cR_{1|Y=0}= a_4 (b_s+b^2) \p_Y \cLU Y^4\vert_{Y=0}.$$

Now, using \eqref{LUY2}, we have 
$$
\p_Y \cLU Y^4=12 \p_Y \LU Y^2= 24 \p_Y Y + O (bY^3)\text{ for } Y\ll 1.
$$
We obtain eventually that
$$
\p_Y \cLU^2 V_{|Y=0}=-\frac{1}{2} (b_s+b^2).
$$

\begin{remark}
It also follows from  equation \eqref{eq:g} and from the computation of $\p_Y^2 \cR$ that for $Y \ll 1$,
$$\p_Y^2 \cLU^2 V = O((b|b_s+b^2| + | \p_s (b_s + b^2)|) Y^2).$$
In particular, $\cLU^3 V$ is well-defined.
\end{remark}

\subsection{Estimate on \texorpdfstring{$\p_Y^2\cLU^2 V$}{h}: proof of Proposition \texorpdfstring{\ref{prop:est-cLU2V}}{}}

We now tackle the estimates on $\p_Y \cLU^2 V$. The general scheme of proof is the same as the one of Proposition \ref{prop:est-cLUV-1}. There are however a few differences:
\begin{itemize}
\item There are now two commutator terms, namely $\p_Y \cLU \mathcal C[\p_Y^2 V]$ and $\p_Y \mathcal C[\p_Y^2 \cLU V]$;
\item The estimate of the remainder term $\p_Y^2 \cLU^2 \cR$  is more technical since the explicit expression of $\p_Y^2 \cLU^2$ has much more terms than the one of $\p_Y^2 \cLU$.
\end{itemize}

We set  $h=\p_Y^2 \cLU^2 V$ in the rest of the paper, and we have
\be\label{eq:h}
\p_s h + 4b h + \frac{b}{2 } Y \p_Y h - \p_Y^2 \LU h= \p_Y^2 \left(\cLU^2 \cR +  \mathcal C[\p_Y^2 \cLU V] + \cLU \mathcal C[\p_Y^2 V]\right).\ee

 In order to derive  estimate \eqref{est:E_3}, we multiply \eqref{eq:h} by $h w_2$, with $w_2:= Y^{-a} (1+ s^{-\beta_2} Y)^{-m_2}$. We recall that we choose $m_2\gg m_1\gg 1$ and $\beta_2<\beta_1$.

Using the same computations as in the previous paragraphs and recalling the definitions of $E_2,D_2$, we have
\begin{eqnarray*}
&&\frac{1}{2}\frac{dE_2}{ds} + \frac{15}{4} b E_2 + \bar c D_2 \\&\leq& \int_0^\infty\p_Y^2 \left(\cLU^2 \cR +  \mathcal C[\p_Y^2 \cLU V] + \cLU \mathcal C[\p_Y^2 V]\right) h w_2 + \bar C s^{-P}+ s^{-P } D_1.
\end{eqnarray*}

We now state the main intermediate results allowing us to prove the statement, namely a commutator estimate and a remainder estimate. We then prove each of the statements separately.

Concerning the commutators, we have the same type of estimates as in the proof of Proposition \ref{prop:est-cLUV-1}, which leads to
\begin{lemma}
Assume that the assumptions of Proposition \ref{prop:est-cLU2V} are satisfied. 
There exist constants $H_2, S_0$, depending on $a, m_i, \beta_i, M_1, M_2$ ($i=1,2$), such that if $s_0\geq \max (S_0, C_0^8, J^4)$, for all $\delta >0$,
\begin{multline*}
\left|\int_0^\infty \int_0^\infty \left(\p_Y^2 \mathcal C[\p_Y^2 \cLU V] + \p_Y^2\cLU \mathcal C[\p_Y^2 V]\right) h w_2\right| \\
\leq \delta\left( bE_2 + D_2 \right) + \frac{H_2}{\delta } b^2 (b_s+ b^2)^2+  \frac{H_2}{\delta } \left( s^{(7+a)\beta_1} D_1 E_2 + s^3(b_s + b^2)^2 E_2 + s^{-2} D_1 \right).
\end{multline*}

\label{lem:commu-2}
\end{lemma}

We now turn towards the remainder term. We have the following estimate:
\begin{lemma}
Assume that the assumptions of Proposition \ref{prop:est-cLU2V} are satisfied. 
There exist constants $H_2, S_0$, depending on $a, m_i, \beta_i, M_1, M_2$ ($i=1,2$), such that if $s_0\geq \max (S_0, C_0^8, J^4)$, for all $\delta >0$,
$$\int_0^\infty \p_Y^2 \cLU^2 \cR\; \p_Y^2 \cLU^2 V w_2\\
\leq \delta (bE_2 + D_2) +  \frac{H_2}{\delta } b^2 (b_s + b^2)^2 +  \frac{H_2}{\delta } s^{-2} D_1 +  \frac{H_2}{\delta } s^{-7+(3-a)\beta_2} .
$$

\label{lem:reste-E2-E3}
\end{lemma}

Gathering  Lemmas \ref{lem:commu-2} and \ref{lem:reste-E2-E3} and choosing $\delta=\min(1/2,\bar c/2 )$, we obtain 
\begin{eqnarray*}
&&\frac{dE_2}{ds} + 7 b E_2 + \frac{\bar c}{2} D_2 \\&\leq & 
H_2 \left(b^2(b_s+b^2)^2 + s^{(7+a)\beta_1}D_1 E_2 + s^3(b_s + b^2)^2 E_2 + s^{-2} D_1 + s^{-7 + (3-a)\beta_2}\right).
\end{eqnarray*}

We now multiply the above equation by $s^5$ and infer
\begin{eqnarray}\label{est:E_2}
&&\frac{d}{ds}(E_2(s) s^5) - H_2\left(s^{(7+a)\beta_1} D_1 + s^3(b_s + b^2)^2 \right)  (E_2(s) s^5) \\&\leq& H_2\left(s^3(b_s+b^2)^2 + s^3 D_1 + s^{-2+(3-a)\beta_2}\right).
\end{eqnarray}
Define
$$
\phi_2(s):=H_2\int_{s_0}^s \tau^{(7+a)\beta_1} D_1(\tau)+ \tau^3(b_\tau + b^2)^2d\tau.
$$
According to Proposition \ref{prop:est-cLUV-1}, since $\beta_1<2/7<1/3$, we have $(7+a)\beta_1<6-(11-a)\beta_1$ and $7\beta_1<13/4$, and therefore, if $s_0\geq \max (S_0, J^4)$,
$$
0\leq \phi_2(s)\leq H_1H_2\left(1+ E_1(s_0)s_0^{13/4} \right).
$$
Hence, multiplying \eqref{est:E_2} by $\exp(-\phi_2(s))$ and integrating over $[s_0,s]$, we obtain
\begin{eqnarray*}
&&E_2(s)s^5\exp(-\phi_2(s)) 
\\&\leq& E_2(s_0)s_0^5 + H_2 \int_{s_0}^s\left(\tau^3(b_\tau+b^2)^2 + \tau^3 D_1 + \tau^{-2+(3-a)\beta_2}\right)\exp(-\phi_2(\tau))d\tau.
\end{eqnarray*}
Now, using assumption \eqref{hy:b2} together with Proposition \ref{prop:est-cLUV-1} with $\alpha=3$, we have, for $s_0\geq \max (S_0, J^4)$,
$$
\ba
\int_{s_0}^s \tau^3(b_\tau+b^2)^2d\tau\leq J s_0^{-1/4} \leq 1,\\
\int_{s_0}^s \tau^3 D_1\leq H_1(1+C_0) ,\\
\int_{s_0}^s \tau^{-2+(3-a)\beta_2}d\tau\leq \frac{1}{1-(3-a)\beta_2} s_0^{-1+(3-a)\beta_2}\leq 1.
\ea
$$
We infer eventually
$$
E_2(s) s^5\leq \exp(H_2(1+ C_0 )) \left[E_2(s_0)s_0^5 +  H_2H_1(1+C_0)\right] \leq \exp(H_2(1+ C_0 ) ) H_2(1+ C_0 ) .
$$

\vskip3mm

We now turn towards the proofs of the Lemmas.

\subsubsection*{The commutator terms: proof of Lemma \ref{lem:commu-2}}
We start with the computation of the commutator terms. Looking at equation \eqref{eq:h}, the commutator integrals in the differential inequality for $E_2$ are
\be\label{commu-E_2}
\int_0^\infty \left(\p_Y^2 \mathcal C [\p_Y^2 \cLU V] + \p_Y^2 \cLU \mathcal C[\p_Y^2 V]\right) \p_Y^2 \cLU^2 V w_2.
\ee
We recall that the heuristics is that up to some corrector terms,
$$
\left|\eqref{commu-E_2}\right|\leq H_2 b \int_0^\infty \frac{1}{1+Y} (\p_Y^2\cLU^2V)^2w_2,
$$
and the right-hand side of the above inequality is then handled by Lemma \ref{lem:rem-1+Y}. However, there are a few complications, coming from the fact that the trace $\p_Y \cLU^2V_{|Y=0}$ is not zero. In the sequel, we will therefore focus on the difficulties and differences with respect to the treatment of the commutators in Proposition \ref{prop:est-cLUV-1}.

$\bullet$ We first consider the first term in the integral of \eqref{commu-E_2}.
This term has the same type of structure as  the term 
$$
\int_0^\infty \p_Y^2 \mathcal C[\p_Y^2 V] \p_Y^2 \cLU V \; w_1, 
$$
which we treated in the previous section. However, there exists one substantial difference, due to the fact that $\p_Y \cLU^2 V$ does not vanish at $Y=0$, so that we cannot write Hardy inequalities for $\p_Y \cLU^2 V$. To overcome this difficulty, we recall that $\p_Y\cLU^2 V_{|Y=0}= - \frac{1}{2} (b_s+b^2)$, so that $\p_Y^2 \cLU V \sim -\frac{1}{2} (b_s+b^2)L_U Y \sim -\frac{1}{4} (b_s+b^2) Y^2$ for $Y\ll 1$, and we write
\be\label{dec:p2LU2V}
\p_Y^2 \cLU V= \left(\p_Y^2 \cLU V + \frac{1}{2}(b_s+b^2) L_U Y\right) - \frac{1}{2}(b_s+b^2) L_U Y.
\ee

Now, we have $\p_Y \LU \left(\p_Y^2 \cLU V + \frac{1}{2}(b_s+b^2) L_U Y\right)_{|Y=0}=0$ by construction, so that we can apply Hardy inequalities to the first term. 
For instance
$$
\int_0^\infty \frac{1}{Y^2}\left(\p_Y \LU \left(\p_Y^2 \cLU V + \frac{1}{2}(b_s+b^2) L_U Y\right)\right)^2 w_2 \leq H_2 \int_0^\infty (\p_Y^2 \cLU^2 V)^2 w_2.
$$
Using the additional bound on $\p_Y^3 U$ from Corollary \ref{cor:p3U}, the  computations  are almost identical to the ones on page \pageref{commu-E1} . The only difference lies in  the treatment of one non-linear term, for which we do not apply exactly the same strategy (i.e. evaluate the term with the least number of derivatives in $L^\infty$, and the others in $L^2$) and for which
we use the extra information coming from the bound on $E_1$ and $D_1$. More precisely, the only term for which we do not use the same type of estimates as the ones on page \pageref{commu-E1} is
$$
\int_0^\infty \p_Y^3 \frac{\cD}{U} \left(\int_0^Y (\cLU^2 V + \frac{1}{2} (b_s + b^2) Y)\right) \p_Y^2 \cLU^2 V w_2.
$$
For this term, the problem comes from the part of $\p_Y^3 \frac{\cD}{U} $ for which we do not have $L^\infty$ bounds, namely $O_{1/3} (U^{-1}) \p_Y^3 \cLU V + O_{1/3} (Y^{-1} U^{-1}) \p_Y^2 \cLU V$. We first integrate by parts once; the most problematic term is then
$$
-\int_0^\infty \frac{\p_Y^2 \cLU V}{U} \left(\int_0^Y (\cLU^2 V + \frac{1}{2} (b_s + b^2) Y)\right) \p_Y^3 \cLU^2 V w_2.
$$
Here, although the middle integral term has less derivatives than $\p_Y^2 \cLU V$, we choose to evaluate it in $L^2$ thanks to a Hardy inequality because cancellations occur between $\cLU^2 V$ and $\frac{1}{2} (b_s + b^2) Y$. More precisely, the above integral is bounded by
$$
\int_0^\infty \left| \frac{\p_Y^2 \cLU V}{U^{1/2}} Y^3 \right| \; \left| \frac{1}{Y^3} \int_0^Y (\cLU^2 V + \frac{1}{2} (b_s + b^2) Y)\right|\; \frac{|\p_Y^3 \cLU^2 V|}{U^{1/2}} w_2.
$$
Now,
\begin{eqnarray*}
 \frac{\p_Y^2 \cLU V}{U^{1/2}} Y^3&=& \int_0^Y \p_Y \left(  \frac{\p_Y^2 \cLU V}{U^{1/2}} Y^3\right) \\
 &\leq & H_2 D_1^{1/2} \left(\int_0^Y (Y^4+ Y^6) w_1^{-1}\right)^{1/2}=O_{\beta_1} (s^{(7+a) \beta_1/2} D_1^{1/2}).
\end{eqnarray*}
The part of the integral bearing on $Y\gtrsim s^{\beta_1}$ can be bounded by noticing that we have pointwise bounds on $\int_0^Y \cLU^2 V$ since
$$
\left|\int_0^Y \cLU^2 V\right|= U \left|\int_0^Y \frac{\p_Y^2 \cLU V}{U^2}\right| \leq 2U \int_1^Y\frac{|\p_Y \cLU V| U_Y}{U^3} + \frac{|\p_Y \cLU V|}{U}
$$
For the term involving $(b_s + b^2)$ on the set $Y\gtrsim s^{\beta_1}$, we merely control $\p_Y^2 \cLU V$ by $E_1$.

Using Hardy inequalities and recalling that $\beta_2<\beta_1$, $m_2\gg m_1$, we infer that the problematic integral is bounded by
\begin{multline*}
H_2 s^{7\beta_1/2} D_1^{1/2} E_2^{1/2} D_2^{1/2} + \delta D_2 + s^{-P} D_1 + |b_s+b^2|^2 s^{-P} E_1
\\ \leq 2 \delta D_2 + \frac{H_2}{\delta} s^{(7+a)\beta_1} D_1 E_2 +  s^{-P} D_1+ |b_s+b^2|^2 s^{-P} E_1.
\end{multline*}
We will then choose $\beta_1$ so that $ 7 \beta_1<6-(11-a) \beta_1$, which is possible since $\beta_1<1/3$. We conclude that for all $\delta>0$, for $s_0\geq \max (S_0, J^4, C_0^8, \delta^{-p})$ for some $p>0$,
\begin{multline*}
\left| \int_0^\infty \p_Y^2 \cC \left[\p_Y^2 \cLU V + \frac{1}{2}(b_s+b^2) L_U Y\right]  \; \p_Y^2 \cLU^2 V \; w_2\right|\\
\leq \delta (bE_2 + D_2) + + \frac{H_2}{\delta} s^{(7+a)\beta_1} D_1 E_2 +  s^{-P} D_1+ |b_s+b^2|^2 s^{-P} E_1 + s^{-P}.
\end{multline*}

The next term coming from \eqref{dec:p2LU2V} is
$$
\int_0^\infty\p_Y^2 \cC \left[\frac{1}{2}(b_s+b^2) L_U Y\right] \p_Y^2 \cLU^2 V w_2.
$$
It turns out that the main order term in $\p_Y^2 \mathcal C [L_U Y]$ vanishes. More precisely, following the decomposition from Lemma \ref{lem:cR}, we write $\cD= -b Y/2 + \tilde \cD + \cD_{NL}= :\cD_0  + \tilde \cD + \cD_{NL}$, and we decompose the operator $\mathcal C$ into ${\mathcal C}_0 + \tilde{\mathcal C} + \mathcal C_{NL}$ accordingly. Concerning $\cC_0$, an explicit computation in Appendix A shows that
$
\p_Y^2 \cC_0 [L_U Y]= O_{\beta_1} (b^2Y),
$ so that, for $s\geq s_0$ sufficiently large,
\begin{multline*}
\left| \int_0^\infty\p_Y^2 \cC_0 \left[\frac{1}{2}(b_s+b^2) L_U Y\right] \p_Y^2 \cLU^2 V w_2\right|\\ \leq H_2 b^2 |b_s+ b^2| s^{(3-a)\beta_2/2} E_2^{1/2} \leq \delta b E_2 + b^2 (b_s + b^2)^2.
\end{multline*}
Concerning the terms involving the operator $\tilde \cC$, we use the fact that $\p_Y^k \tilde \cD=O_{\beta_1}(b^3 Y^{7-k} + b^4 Y^{10-k})$ for $Y\gg 1$ and for $k=0,1,2$. Therefore
$$
\p_Y^2 \tilde{\mathcal C}[L_U Y]= O_{\beta_1}(b^3 Y^4) - \p_Y^3\frac{\tilde \cD}{U} \frac{Y^2}{2}.
$$
Hence, integrating by parts the term involving $\p_Y^3 \tilde \cD$ and choosing $s_0$ large enough,
\begin{eqnarray*}
&&\left|\int_0^\infty\p_Y^2 \tilde \cC \left[\frac{1}{2}(b_s+b^2) L_U Y\right] \p_Y^2 \cLU^2 V w_2 \right|\\
&\leq&  b^3 |b_s + b^2| \int_0^\infty O_{\beta_1} (Y^4) | \p_Y^2 \cLU^2 V| w_2 +   b^3 |b_s + b^2|\int_0^\infty O_{\beta_1} (Y^5) | \p_Y^3 \cLU^2 V| w_2 \\
&\leq & \delta b E_2 + \delta D_2 + \frac{H_2}{\delta} (b_s + b^2)^2 (s^{-5+ 9 \beta_2} + s^{-6 + 13\beta_2}) \\
&\leq& \delta b E_2 + \delta D_2  +  b^2 (b_s + b^2)^2 .
\end{eqnarray*}
There remains to address the terms involving $\cD_{NL}$. Notice that $\cD_{NL}, \p_Y \cD_{NL}$ are bounded in $L^\infty$ (see Corollary \ref{cor:p3U}). As above, we integrate by parts the term involving $\p_Y^3 \cD_{NL}$. Eventually, we obtain, using Corollary \ref{cor:p3U}
\begin{eqnarray*}
&&\left|\int_0^\infty\p_Y^2  \cC_{NL} \left[\frac{1}{2}(b_s+b^2) L_U Y\right] \p_Y^2 \cLU^2 V w_2 \right|\\
&\leq & \delta b E_2  + \delta D_2 + \frac{H_2}{\delta}b^2 |b_s + b^2|^2 + s^{-P} D_1.
\end{eqnarray*}

Gathering all the estimates, we obtain, if $s_0\geq \max(S_0, J^4, C_0^8)$,
\[
\left|\int_0^\infty \p_Y^2 \mathcal C [\p_Y^2 \cLU V] \p_Y^2 \cLU^2 V  w_2\right|\leq \delta\left( bE_2 + D_2 \right) + \frac{H_2}{\delta } b^2 (b_s+ b^2)^2+  \frac{H_2}{\delta } s^{7\beta_1} D_1 E_2 + s^{-P} D_1 .
\]

We then address the second term in \eqref{commu-E_2}, for which we use the same type of decomposition as above, writing
$$
\p_Y^2 V = \left(\p_Y^2 V +\frac{1}{48} (b_s+b^2) L_U Y^4 \right) -  \frac{1}{48} (b_s+b^2) L_U Y^4=:\p_Y^2 \tilde V -  \frac{1}{48} (b_s+b^2) L_U Y^4.
$$
Using the formula for $\LU(Y^2)$ \eqref{LUY2} together with the bounds on $V$ stemming from $E_1$, notice that
$$
\p_Y \cLU^2 \tilde V_{|Y=0}=\p_Y \cLU^2  V_{|Y=0} + \frac{1}{4}\left(\p_Y \LU Y^2\right)_{|Y=0}=0,
$$
and that $$\p_Y^2 \cLU^2 \tilde V = \p_Y^2 \cLU^2 V + \frac14 (b_s + b^2) \p_Y^2 \LU Y^2 =\p_Y^2 \cLU^2 V  + O_{\beta_1} ((b_s + b^2) bY).$$

Concerning the term
\be\label{commu-E3-1}
\int_0^\infty  \p_Y^2 \cLU \mathcal C[\p_Y^2 \tilde V]\p_Y^2 \cLU^2 V w_2,
\ee
the computations are very similar to the ones  above, using the formula for $\p_Y^2 \mathcal C[\cdot]$ from Lemma \ref{lem:commu-formula} on the one hand, and the formula for $\p_Y^2 \LU$ from Lemma \ref{lem:p3LU} on the other hand. 
We leave the details of the estimates to the reader since they do not raise any additional difficulty. We end up with
$$
\eqref{commu-E3-1} \leq \delta \left( D_2 + b E_2\right) + b^2 (b_s + b^2)^2  + \frac{H_2}{\delta}\left(s^{-3} D_1 + s^{(7+a)\beta_1} D_1 E_2+ s^3(b_s + b^2)^2 E_2 + s^{-P}\right).
$$
We then consider
\be
\label{commu-E3-2}
\frac{1}{48} (b_s+b^2)\int_0^\infty  \p_Y^2 \cLU \mathcal C[ L_U Y^4]\p_Y^2 \cLU^2 V w_2.
\ee
A straightforward computation leads to
$$
\p_Y^2 \cLU \mathcal C[ L_U Y^4] = O_{\beta_1} \left(\frac{b}{(1+Y) U} \right) + \sum_{i=0}^4 O_{\beta_1} ((1+Y)^{i-1})\p_Y^i \int_0^Y \cLU^2 V. 
$$
We have
$$
\left| \int_0^\infty  O_{\beta_1} \left(\frac{b(b_s+b^2)}{(1+Y) U} \right)\p_Y^2 \cLU^2 V w_2\right| \leq \delta D_2 + \frac{H_2}{\delta} b^2 (b_s + b^2)^2.
$$
Expressing $\p_Y\cLU^2 V$ and $\cLU^2 V$ in terms of $\cLU V$ and using the expressions of $E_1$, $D_1$, it can be easily proved that the  terms of the form $O_{\beta_1} ((1+Y)^{i-1})\p_Y^i \int_0^Y \cLU^2 V$ give rise to integrals that can be bounded by
$$
\delta b E_2 + \delta D_2+  H_2 s^{-2} D_1 +  s^{-P}.
$$
Gathering all the terms, we obtain the estimate announced in  Lemma \ref{lem:commu-2}.

\subsubsection*{The remainder terms: proof of Lemma \ref{lem:reste-E2-E3}}

We now consider the remainder terms occurring in the right-hand side of the differential inequality for $E_2$, namely
$$
\int_0^\infty \p_Y^2 \cLU^2 V\; \p_Y^2 \cLU^2 \cR\; w_2.
$$
Following  the decomposition of $\cR$ from Remark \ref{rem:decomposition-D}, we will write $\cR$ as $\sum_{i=1}^4 \cR_i$, and study the contribution of each $\cR_i$ separately. We emphasize that the most important terms are $\cR_1$ and $\cR_2$: indeed, they dictate the  final convergence rate, whereas the terms involving $\cR_3$ and $\cR_4$ are small  perturbations of the main dissipation terms $D_2 + b E_2$.

$\bullet$\textbf{ Remainder stemming from $\cR_1 $:}

Using the decomposition \eqref{dec:R1} of the previous paragraph and using the fact that for $k\geq 8$ and $Y\gg 1$,
$$
\p_Y^2 \LU \p_Y^2 \LU \left[Y^k \chi\left( \frac{Y}{s^{2/7}}\right)\right] = O_{\beta_1}(Y^{k-8})+ O_{\beta_1}(Y^{k-4}) \p_Y^2 \cLU V + O_{\beta_1} (Y^{k-3}) \p_Y \cLU V,
$$
we find that
\begin{eqnarray}
\nonumber&&\p_Y \cLU^2 (\cR_1) \\
\label{eq:cR1}&=& (b_s + b^2) b\p_Y^2 \LU \p_Y^2 \LU \left[ \left(\frac{a_4}{2}Y^5 - \frac{61}{16} a_7 b Y^8\right)\chi\left(\frac{Y}{s^{2/7}}\right) \right]\\
\nonumber		&+& O_{\beta_1} (b^4  ) + O_{\beta_1}(b^5 Y^4) + O_{\beta_1} (b^4Y^5 + b^5 Y^9) \p_Y^2 \cLU V + O_{\beta_1} (b^4 Y^6 + b^5 Y^{10}) \p_Y \cLU V\\
\nonumber	&-& \frac{1}{2} (b_s+b^2) \p_Y \LU \p_Y^2 \LU L_V Y + \p_Y^2 \LU \left( \tilde P_1(s,Y) (1-\chi)\left(\frac{Y}{s^{2/7}}\right)\right).
\end{eqnarray}
Define
$$
\varphi(s,Y):= \p_Y^2 \LU \p_Y^2 \LU \left[ Y^5\chi\left(\frac{Y}{s^{2/7}}\right) \right].
$$
Then
\begin{eqnarray*}
\varphi(s,Y)&=& \p_Y^3 U\int_0^Y \left(\frac{1}{U^2}\p_Y^2 \LU \left[ Y^5\chi\left(\frac{Y}{s^{2/7}}\right) \right]\right) + \frac{U_{YY}}{U^2} \p_Y^2 \LU \left[ Y^5\chi\left(\frac{Y}{s^{2/7}}\right) \right]\\
&&- \frac{U_Y}{U^2} \p_Y^3 \LU \left[ Y^5\chi\left(\frac{Y}{s^{2/7}}\right) \right]+ \frac{1}{U}\p_Y^4 \LU \left[ Y^5\chi\left(\frac{Y}{s^{2/7}}\right) \right].
\end{eqnarray*}
Since we focus on the value of the above quantities for $Y\leq s^{\beta_1}$, we can replace $\chi(Y/s^{2/7})$ by 1; following by now usual arguments, the part of the integral bearing on $Y\geq s^{\beta_1}$ will be smaller than $s^{-P}$ for  $P>0$ arbitrarily large, provided $m_2\gg m_1$ is chosen large enough. Using Lemma \ref{lem:p3LU}, we have, for $Y\leq s^{\beta_1}$ and for $s\geq s_0$ large enough,
\begin{eqnarray}
\p_Y^2 \LU \left[ Y^5\chi\left(\frac{Y}{s^{2/7}}\right) \right]&=& \p_Y^3 U \int_0^Y \frac{Y^5}{U^2}\nonumber\\
&&+ \frac{Y^5 - 5(1+Y) Y^4 + 20 Y^3 (Y+ \frac{Y^2}{2}) + O (bY^7)}{U^2} \nonumber\\
&=& \frac{6Y^5 +15 Y^4 + O(bY^7)}{U^2} \geq\frac{5Y^5 +14 Y^4 }{U^2}>0\label{varphi-1}
\end{eqnarray}
while
\begin{eqnarray}
\p_Y^3 \LU \left[ Y^5\chi\left(\frac{Y}{s^{2/7}}\right) \right] &=&\p_Y^4 U \int_0^Y  \frac{Y^5}{U^2} + 2 \left(\frac{\p_Y^3 U}{U^2} - \frac{U_Y U_{YY}}{U^3}\right) Y^5\nonumber\\
&&+ 10 Y^4 \frac{U_Y^2}{U^3} - 40Y^3\frac{U_Y}{U^2} + 120\frac{Y^2}{U}\nonumber\\
&=&\frac{2Y^4 (9 Y^2 + 29 Y +45 )}{U^3} + O(bY^2)+ O(Y^2U ) \p_Y^2 \cLU V.\label{varphi-2}
\end{eqnarray}
A similar formula holds for $\p_Y^4 \LU [Y^4]$. It follows that
\[
\varphi(s,Y)=O_{\beta_1} ((1+Y)^{-3}) + O_{\beta_1}(Y) \p_Y^2 \cLU V + O_{\beta_1}(Y^2) \p_Y^3 \cLU V + O_{\beta_1}(bY)
\]
and
\begin{eqnarray*}
|\p_Y^2\cLU^2 \cR_1 |&=&|b_s + b^2| b | \varphi(s,Y)|+O_{\beta_1} (b^4 +b^5 Y^4 ) \\
&&+ O_{\beta_1}(b^4Y^6 + b^5 Y^{10} +(b_s+b^2) Y^2 ) \p_Y \cLU^2 V +\lot
\end{eqnarray*}
As a consequence, since $\beta_2<\beta_1$,
\begin{eqnarray*}
&&\left| \int_0^\infty\left[ (b_s + b^2) b\varphi(s,Y)\right]\; \left[ \p_Y^2 \cLU^2 V\right]\; w_2 \right|\\
&\leq & H_2 b |b_s+b^2|D_2^{1/2}\left( \int_0^\infty \frac{U^2}{(1+Y)^{6}} w_2\right)^{1/2} + H_2 b^2 |b_s + b^2| E_2^{1/2} \left(\int_0^\infty Y^{2} w_2\right)^{1/2}\\
&&+ H_2 b|b_s+b^2| E_2^{1/2} \left(s^{6/7} D_1^{1/2}+ s^{-P}\right).
\end{eqnarray*}
For any $\delta>0$, if $s\geq s_0$ large enough (depending on $\delta$), the right-hand side is bounded by
\begin{multline*}
\delta \left(b E_2 + D_2\right) + \frac{H_2}{\delta}b^2 |b_s+b^2|^2 +  \frac{H_2}{\delta} b |b_s+b^2|^2 s^{12/7} D_1 +  s^{-P} \\
\leq \delta \left(b E_2 + D_2\right) + \frac{H_2}{\delta}b^2 |b_s+b^2|^2 + s^{-2} D_1+  s^{-P}.
\end{multline*}
Gathering all the terms, it follows that
$$
\int_0^\infty  \left(\p_Y^2 \cLU^2 \cR_1\right) \left(\p_Y^2 \cLU^2 V \right) Y^{-a} w_2
\leq  \delta \left(b E_2 + D_2\right) + \frac{H_2}{\delta}b^2 |b_s+b^2|^2 +s^{-2}  D_1(s)+ s^{-P} 
$$

$\bullet$ \textbf{Remainder stemming from $\cR_2$:}
As announced in Remark \ref{rem:decomposition-D}, the second remainder $\cR_2$ will partly dictate the total size of the remainder. More precisely, we have, for $Y\gg 1$,
$$
\ba
\cLU^2 \cR_2=O_{\beta_1}(b^4 Y^3  ) ,\\
\p_Y \cLU^2 \cR_2= O_{\beta_1}(b^4 Y^2  ) + O_{\beta_1}(b^4Y^5)\cLU^2 V,\\
\p_Y^2 \cLU^2 \cR_2= O_{\beta_1}(b^4 Y  ) + O_{\beta_1}(b^4Y^5)\p_Y\cLU^2 V + O_{\beta_1}(b^4Y^4)\cLU^2 V .
\ea
$$
We infer that for all $\delta >0$ and for $s\geq s_0$ large enough (depending on $\delta$),
$$
\left|\int_0^\infty \p_Y^2 \cLU^2 \cR_2\;  \p_Y^2 \cLU^2 V w_2\right| \leq \delta b E_2 + \frac{H_2}{\delta } s^{-7+(3-a)\beta_2} +\frac{H_2}{\delta } s^{-17/4} D_1  .
$$

$\bullet$ \textbf{Remainder stemming from $\cR_3$:}

An easy computation leads to
$$
\ba
\cLU^2 (\LU L_V Y)&= O_{\beta_1}\left(1\right) \p_Y \cLU^2 V + \lot + U_Y \int_0^Y \frac{\p_Y^2\cLU\LU (L_V Y)}{U^2},\\
\p_Y \cLU^2 (\LU L_V Y)&= O_{\beta_1}\left(1\right) \p_Y^2 \cLU^2 V + O_{\beta_1}\left(\frac{1}{1+Y}\right) \p_Y \cLU^2 V + \lot\\ &+ U_{YY} \int_0^Y \frac{\p_Y^2\cLU\LU (L_V Y)}{U^2}.
\ea
$$

In order to estimate $\p_Y^2 \cLU^2 (\LU L_V Y)$, we use the same trick as in the commutator estimate and we replace $V$ by its asymptotic expansion close to $Y=0$. More precisely, we write
$$
V= \left( V + \frac{1}{2} (b_s + b^2) \cLU^{-2} Y\right) - \frac{1}{2}  (b_s + b^2) \cLU^{-2} Y=: V_0 - \frac{1}{2}  (b_s + b^2) \cLU^{-2} Y,
$$
with the convention $\p_Y^{-1}=\int_0^Y$.
Now, by definition $\p_Y \cLU^2 V_{0|Y=0}=0$ and $\p_Y^2 \cLU^2 V_0=\p_Y^2 \cLU^2 V$. Moreover, it can be easily checked that
\begin{eqnarray*}
&&\p_Y^2 \cLU^2 \LU \left(Y \cLU^{-2} Y - \frac{Y^2}{2} \p_Y\cLU^{-2} Y\right)\\&=& O_{\beta_1}\left(\frac{1}{Y(1+Y)^2}\right) + O_{1/3}(Y^5 ) \p_Y^2 \cLU^2 V + O_{1/3}(Y^4 )\p_Y\cLU^2 V + \lot,
\end{eqnarray*}
while
$$
\p_Y^2 \cLU^2 \LU \left(L_{V_0} Y\right)=  O_{\beta_1}\sum_{i=1}^3 O_{1/3} (Y^{3+i}) \p_Y^i \cLU^2 V+ \lot
$$
It follows that
\begin{eqnarray*}
&&\left| \int_0^\infty b \p_Y^2 \cLU^2 (\LU L_V Y) \p_Y^2 \cLU^2 V w_2\right|\\
&\leq & \delta D_2 + \frac{C}{\delta}b^2\int_0^\infty U O_{\beta_1}\left(\frac{Y^4}{U^2}\right) (\p_Y^2 \cLU^2 V)^2 w_2 + H_2b \int_0^\infty \frac{1}{1+Y} (\p_Y^2 \cLU^2 V)^2 w_2\\
&&+ H_2 b|b_s + b^2| \int_0^\infty O_{\beta_1}\left(\frac{1}{Y(1+Y)^2}\right) \left|\p_Y^2 \cLU^2 V\right| w_2 \\&& + H_2 b|b_s + b^2| \sum_{i=1}^3\int_0^\infty O_{1/3}(Y^{3+i} )|\p_Y^i\cLU^2 V|\;  \left|\p_Y^2 \cLU^2 V\right| w_2.
\end{eqnarray*}
Using the estimate on $\p_Y \cLU^2 V$ from Corollary \ref{cor:p3U}, we infer that the right-hand side is bounded by
$$\delta (D_2 + b E_2) + \frac{H_2}{\delta} b^2(b_s+b^2)^2 + H_2 s^{-3} D_1.$$

The same method and estimates apply to $b^3 \LU (L_V Y^7)$. At last, we consider
$$
b^3 \LU\left(\chi\left( \frac{Y}{s^{2/7}}\right) L_{-a_4 b Y^4 - a_7 b^2 Y^7 + a_{10} b^3 Y^{10} + a_{11} b^3 Y^{11}} Y^7\right)=: \LU (\zeta(s,Y)). 
$$
Note that $\zeta(s,\cdot)\in \mathcal C^\infty(Y)$ and that
$$
\p_Y^k \zeta(s,Y)=O(b^4 Y^{11-k})\quad \forall k \leq 11. 
$$
It follows that
\begin{eqnarray*}
\cLU^2 \LU \zeta&=& O(b^4 Y\ln Y) + O_{1/3}(s^{-10/3} Y^4 + s^{-4} Y^8) \p_Y \cLU^2 V + \lot \\
\p_Y \cLU^2 \LU \zeta&=& O(b^4 \ln Y) + O_{1/3}(s^{-10/3} Y^4 + s^{-4} Y^8) \p_Y^2 \cLU^2 V \\&&+  O_{1/3}(s^{-10/3} Y^3 + s^{-4} Y^7) \p_Y \cLU^2 V+\lot\\
\p_{YY}\cLU^2 \LU \zeta&= &O\left(\frac{b^4}{1+Y} \right)+ O_{1/3}(s^{-10/3} Y^2 + s^{-4} Y^6) \p_Y^3 \cLU^2 V\\&& +  O_{1/3}(s^{-10/3} Y^3 + s^{-4} Y^7) \p_Y \cLU^2 V+ \lot 
\end{eqnarray*}

We obtain, since $\beta_2<\beta_1<1/3$,
\begin{multline*}
\left| \int_0^\infty \left( \cLU^2 \LU \zeta \; \cLU^2 V \; \p_{YY} w_2 + \p_Y \cLU^2 \LU \zeta\; \p_Y \cLU^2 V \; w_2\right) \right|\\
\leq \delta b E_2 + \frac{H_2}{\delta} s^{-27/4} \left(\ln s\right)^2 + \frac{H_2}{\delta} s^{-7+(16+a)\beta_2} E_2
\end{multline*}
and, writing $\p_Y \cLU^2 V=\left( \p_Y \cLU^2 V + \frac{1}{2} (b_s + b^2)\right) - \frac{1}{2}(b_s + b^2)$,
$$
\left| \int_0^\infty\p_Y^2 \cLU^2 \LU \zeta\; \p_Y^2 \cLU^2 V w_2 \right|\\
\leq \delta b E_2 + \delta D_2+ \frac{H_2}{\delta}s^{-7}  + \frac{H_2}{\delta } s^{-7 + (13-a)\beta_2} (b_s + b^2)^2 + s^{-M } D_1.
$$

$\bullet$ \textbf{Remainder stemming from $\cR_4$:}

We recall that
$$
\cR_4 = P_1(s,Y) (1-\chi_1)\left(\frac{Y}{s^{2/7}}\right) + \LU \left( P_2(s,Y) (1-\chi_2)\left(\frac{Y}{s^{2/7}}\right) \right),
$$
and that for any $P>0$, $k\geq 0$, $i=1,2$, choosing $\beta_1<2/7$,
$$
\p_Y^k (P_i (1-\chi_i)(\frac{Y}{s^{2/7}}))= O_{\beta_1}(s^{-P}).
$$

Using the same type of computations as above, it follows that for any $M>0$,
$$
\ba
\cLU^2 \cR^4= O_{\beta_1}(s^{-P}) + O_{\beta_1} (s^{-P})\p_Y \cLU^2 V + \lot \text{ in } L^2(\p_{YY} w_2),\\
\p_Y \cLU^2 \LU \psi= O_{\beta_1}(s^{-P})+O_{\beta_1}(s^{-P}) \p_Y^2 \cLU^2 V + O_{\beta_1}(s^{-P})\p_Y \cLU^2 V+ \lot \text{ in } L^2(w_2)\\
\p_{YY}\cLU^2 \LU \psi= O_{\beta_1}(s^{-P})+ O_{\beta_1}(s^{-P}) \p_Y^3 \cLU^2 V +  O_{\beta_1}(s^{-P}) \p_Y \cLU^2 V+ \lot \text{ in } L^2(w_2).
\ea
$$
Thus, choosing $m_2$ and $m_2$ large enough,
$$
\ba
\left| \int_0^\infty (\cLU^2 \cR_4 \cLU^2 V \p_{YY} w_2 + \p_Y \cLU^2 \cR_4 \p_Y \cLU^2 V w_2)\right| \leq  \delta b E_2 +  s^{-P}  + s^{-P} E_2+ s^{-P} D_1,\\
\left| \int_0^\infty\p_Y^2 \cLU^2 \cR_4 \; \p_Y^2\cLU^2 V w_2 \right| \leq \delta b E_2 + \delta D_2 +  s^{-P}+ s^{-P} D_1.
\ea
$$

Gathering all the terms, we obtain the estimates announced in Lemma \ref{lem:reste-E2-E3}.

\section*{Acknowledgements}

A.-L. Dalibard is partially supported by the Agence Nationale de la Recherche, project Dyficolti, Grant ANR-13-BS01-0003-01. This project has received funding from the European Research Council (ERC) under the European Union's Horizon 2020 research and innovation program Grant agreement No 637653, project BLOC ``Mathematical Study of Boundary Layers in Oceanic Motion''. This work was initiated during A.-L. D.'s stay at the Courant Institute in 2013-2014, and she warmly thanks the Courant Institute for its hospitality.

\appendix

\section{Useful formulas}

$\bullet$  We will often need to transform derivatives of $V$ into quantities involving $\cLU V, \cLU^2 V$ and their derivatives. In order to do so, we start from the following observation
\be\label{d2V}
\p_Y^2 V = L_U \cLU V = U \cLU V - U_Y\int_0^ Y \cLU V,
\ee
from which it follows, applying the same idea to $\p_Y^2 \cLU V$,
\be\label{pkV}
\ba\p_Y^3 V = U \p_Y \cLU V - U_{YY}\int_0^Y \cLU V,\\
\p_Y^4 V = U^2 \cLU^2 V - U U_Y \int_0^Y \cLU^2 V + U_Y \p_Y \cLU V - U_{YY} \cLU V - (\p_Y^3 \Uapp + \p_Y^3 V) \int_0^Y \cLU V.\ea
\ee
Notice that for derivatives of $U$ of higher greater than or equal to two, we decompose $\p_Y^k U$ into $\p_Y^k \Uapp + \p_Y^k V$. This is related to the fact that we have pointwise estimates on $\p_Y^2 U$, but not on higher derivatives. Now, in the formula giving $\p_Y^4 V$, we can write $\p_Y^3 V$ in terms of $\cLU V$. Obviously, we can iterate this procedure. As a consequence, for any $k\geq 2$, we can express $\p_Y^k V$ in terms of $\cLU^l V$ for $l\leq k/2$.

$\bullet$ We will also need the explicit expression of $\p_Y^k \LU$, which is given in the following Lemma, whose proof is straightforward and left to the reader.
\begin{lemma}
For any function $W$ which  vanishes at a sufficiently high order near $Y=0$, we have
\begin{eqnarray*}
\p_Y \LU W&=&U_{YY}\int_0^Y  \frac{W}{U^2}+\frac{\p_Y W}{U},\\
\p_Y^2 \LU W &=& \p_Y^3 U \int_0^Y \frac{W}{U^2} + \p_Y^2 U \frac{W}{U^2} - \p_Y U\frac{\p_Y W}{U^2} + \frac{\p_Y^2 W}{U}\\
&=&  \p_Y^3 U \int_0^Y \frac{W}{U^2} + \p_Y^2 U \frac{W}{U^2} + \p_Y \frac{\p_Y W}{U}
\end{eqnarray*}
and
$$
\p_Y^3 \LU W= \p_Y^4 U \int_0^Y \frac{W}{U^2} + 2 \left(\frac{ U\p_Y^3 U-U_{YY} U_Y }{U^3} \right) W + 2 \frac{U_Y^2}{U^3} \p_Y W - 2 \frac{U_Y}{U^2} \p_{YY } W + \frac{\p_Y^3 W}{U}.
$$

\label{lem:p3LU}
\end{lemma}

$\bullet$ Eventually, setting $\cD_0:=-bY/2$ and 
\[
\cC_0[W]:=2 \LU \left(\cD_0\frac{W}{U}\right) - \p_Y \left(\frac{\cD_0}{U} \int_0^Y \LU W\right),
\]
we need to compute $\p_Y^2 \cC_0[L_U Y]$. By definition, $L_U Y = Y U - U_Y Y^2/2$, so that
\[
\cC_0[L_U Y]= -\frac{b}{2}\left\{2 \LU (Y^2) - \LU \left(\frac{Y^3 U_Y}{U}\right) - \p_Y \left( \frac{Y^3}{2U}\right) \right\}.
\]
Now, integrating by parts,
\[
\LU \left(\frac{Y^3 U_Y}{U}\right)= \p_Y \left[ U \int_0^Y \frac{Y^3 U_Y}{U^2}\right]= - \p_Y \left( \frac{Y^3}{2U}\right) + \frac{3}{2} \LU (Y^2).
\]
Therefore
\[
\cC_0[L_U Y]= - \frac{b}{4} \LU (Y^2).
\]
Using the formula above for $\p_Y^2 \LU$, we obtain
\[
\p_Y^2 \cC_0[L_U Y]=- \frac{b}{4} \left\{\p_Y^3 U \int_0^Y \frac{Y^2}{U^2} + \frac{Y^2 U_{YY} - 2 Y U_Y + 2 U^2}{U^2}\right\} .
\]

\section{Estimate on the modulation rate}

\begin{lemma}
	Let $\gamma\in (0,5)$, and let $\varphi:[s_0, s_1]\to \R_+$ be such that
	$$
	\int_{s_0}^{s_1} s^\gamma \varphi(s)\:ds <+\infty.
	$$
	Assume that there exists a constant $\eps>0$ 
	such  that for all  $s\in [s_0, s_1]$ 
	$$
	\ba
	\left| b_s + b^2\right| \leq \sqrt{\varphi},\\
	\frac{1-\eps}{s}\leq b(s)\leq \frac{1+\eps}{s}.\ea
	$$

	Then for all $s\in [s_0, s_1]$,
	$$
	\left| b(s)- \frac{1}{s} \right| \leq \frac{1+\eps}{1-\eps}\left| \frac{1}{s_0} - b(s_0)\right| \frac{s_0^2}{s^2} + \frac{1+\eps}{(1-\eps)^2\sqrt{5-\gamma}} s^{\frac{1-\gamma}{2}} \left( \int_{s_0}^{s_1} t^\gamma \varphi(t) dt\right)^{1/2}.
	$$

	In particular:
	\begin{itemize}
		\item If $\varphi(s)= J s^{-4-\eta}$, the inequality becomes, with $\gamma=3+ \eta/2$,
		$$	\left| b(s)- \frac{1}{s} \right| \leq \frac{1+\eps}{1-\eps}\left| \frac{1}{s_0} - b(s_0)\right| \frac{s_0^2}{s^2} + \frac{J^{1/2}}{\sqrt{\eta}} \frac{1+\eps}{(1-\eps)^2} s^{-1 - \frac{\eta}{4}}.
		$$
		
		\item More generally, if $\gamma>3$, then there exists a constant $\eta>0$  ($\eta=(\gamma-3)/2$) such that for all $s\in [s_0, s_1]$,
		$$
		\left| b(s)- \frac{1}{s} \right| \leq \frac{1+\eps}{1-\eps}\left| \frac{1}{s_0} - b(s_0)\right| \frac{s_0^2}{s^2} +  \frac{(1+\eps)J^{1/2}}{(1-\eps)^2 \sqrt{2-2\eta}} s^{-1-\eta},
		$$
		where $J=\int_{s_0}^{s_1} t^\gamma \varphi(t) dt$.
		
	\end{itemize}

	\label{lem:b}
	
\end{lemma}

\begin{proof}
	The  assumption on $b$ entails
	$$
	\left| \frac{b_s}{b^2} + 1 \right| = \left| \frac{d}{ds}\left(s-\frac{1}{b}\right) \right| \leq \frac{\sqrt{\varphi}}{ b^2} \leq \frac{1}{(1-\eps)^2 } s^2 \sqrt{\varphi}.
	$$
	Integrating the above inequality between $s_0$ and $s$ and using a Cauchy-Schwarz inequality yields
	\begin{eqnarray*}
		\left| s -\frac{1}{b(s)} \right|& \leq& \left| s_0 -\frac{1}{b(s_0)} \right| + \frac{1}{(1-\eps)^2 } \left( \int_{s_0}^{s_1} t^\gamma \varphi(t) dt\right)^{1/2} \left(\int_{s_0}^s t^{4-\gamma}\:dt\right)^{1/2}\\
		&\leq&  \left| s_0 -\frac{1}{b(s_0)} \right| + \frac{1}{(1-\eps)^2 } \left( \int_{s_0}^{s_1} t^\gamma \varphi(t) dt\right)^{1/2} \frac{s^{\frac{5-\gamma}{2}}}{\sqrt{5-\gamma}}.
	\end{eqnarray*}
	Now, multiplying the above inequality by $b/s\leq (1+\eps) s^{-2}$, we obtain
	$$
	\left| b(s)- \frac{1}{s} \right| \leq (1+\eps) \frac{s_0}{b(s_0)}\left| \frac{1}{s_0} - b(s_0)\right| \frac{1}{s^2}+ \frac{1+\eps}{(1-\eps)^2}\left( \int_{s_0}^{s_1} t^\gamma \varphi(t) dt\right)^{1/2}  \frac{s^{\frac{1-\gamma}{2}}}{\sqrt{5-\gamma}}.
	$$
	Since $s_0/b(s_0)\leq (1-\eps)^{-1} s_0^2$, we obtain the inequality announced in the Lemma.
\end{proof}

Lemma \ref{lem:b} has in particular the following consequence:
\begin{corollary}
	Assume that $b$ satisfies the assumptions of Lemma \ref{lem:b} for some $\gamma \in ]3,4[$. For $s\geq s_0$, define $\tb$ by
	$$
	\tb_s + b \tb=0,\quad \tb_{|s=s_0}= \frac{1}{s_0}.
	$$
	Then there exists a universal constant $\bar C$ such that if $s_0\geq \bar C (J\eps^{-2})^{1/(\gamma-3)}$, for all $s\geq s_0$,
	$$
	\frac{1-2\eps}{s}\leq \tb(s)  \leq \frac{1+2\eps}{s}.
	$$
	\label{cor:tb}
\end{corollary}
\begin{proof}
	Since $\tb$ satisfies a linear ODE, we have simply
	$$
	\tb(s)= \frac{1}{s_0} \exp\left(-\int_{s_0}^s b(s')ds'\right).
	$$
	According to Lemma \ref{lem:b}, for all $s\geq s_0$,
	$$
	\ln \frac{s}{s_0} - C_{\eps, \gamma,s_0}\leq \int_{s_0}^s b(s')ds'\leq  \ln \frac{s}{s_0} + C_{\eps, \gamma, s_0},$$
	where
	$$
	C_{\eps, \gamma,s_0}=\frac{1+\eps}{(1-\eps)^2}\eps+\frac{1+\eps}{(1-\eps)^2}J^{1/2}  (5-\gamma)^{-1/2} \frac{2 s_0^{\frac{3-\gamma}{2}}}{\gamma-3}.
	$$
	Therefore
	$$
	\frac{e^{-C_{\eps, \gamma, s_0}}}{s}\leq \tb (s)\leq \frac{e^{C_{\eps, \gamma, s_0}}}{s}.
	$$
	Now, if $s_0\geq \bar C (J\eps^{-2})^{1/(\gamma-3)}$, we have
	$$
	e^{C_{\eps, \gamma, s_0}} \leq 1+2\eps, \quad e^{-C_{\eps,\gamma,  s_0}}\geq 1-2\eps,
	$$
	which completes the proof.
\end{proof}

\section{Proof of Lemma \ref{lem:sub-U}}

As much as possible, we treat  sub-solutions and super-solutions simultaneously. For any $A_\pm>0$, $k_\pm>2$, we consider the function
$$
W_\pm(s,\psi):=\frac{(6\psi)^{4/3}}{4} \pm A_\pm \psi^{k_\pm} \tb^{\frac{3k_\pm-2}{4}}.
$$
We will choose $k_-=7/3$ for sub-solutions and $k_+=10/3$ for super solutions. 

We claim that $W_\pm$ satisfy the following properties: 
\begin{itemize}
\item Choosing $k_-=7/3$, there exists $\bar C$  such that if $C_-\geq \bar C$, then for all $A_->0$,
\be\label{in:W-}
\p_s W_- - 2 b W_-+ \frac{3b}{2} \psi \p_\psi W_- - \sqrt{ W_-} \p_{\psi \psi} W_- + 2 \leq 0
\ee
on the domain $\{\psi \geq C_- \tb^{-3/4}\} \cap \{W_->0\}$.

Similarly, choosing $k_+=10/3$, there exists $\bar C>0$ such that if $C_-\geq \bar C$, for all $A_+>0$
\be\label{in:W+}
\p_s W_+ - 2 b \uW + \frac{3b}{2} \psi \p_\psi W_+ - \sqrt{ W_+} \p_{\psi \psi} W_+ + 2 \geq 0 \quad \forall \psi \geq C_- \tb^{-3/4}.
\ee

\item There exists a  constants $\bar A$ such that if $A_\pm\geq \bar A$  and if  $s_0$ is large enough  (depending on $A_\pm, C_1, a$ and $C_-$),
\be\label{in:W-bottom}
W_-(s,C_- \tb^{-3/4}) \leq W(s, C_- \tb^{-3/4})\leq W_+(s,C_- \tb^{-3/4}).
\ee

\item There exists a constant $A_0$ , depending on $M_0$ such that if $A_\pm \geq A_0$,
\be\label{in:W-initial}
W_-(s_0,\psi) \leq W(s_0, \psi)\leq W_+(s_0,\psi) \quad \forall \psi \geq C_- s_0^{3/4}.
\ee

\end{itemize}

\subsubsection*{Proof of inequalities \eqref{in:W-} and \eqref{in:W+}}

We first compute the transport term. We have
		$$
		\p_s W_\pm - 2 b W_\pm + \frac{3b}{2} \psi \p_\psi W_\pm = \pm A_\pm \frac{3(k_\pm-2)}{4} b\tb^{\frac{3k_\pm-2}{4}}\psi^{k_\pm} \gtrless 0
		$$
		provided $k_\pm>2$. We now address the computation of the diffusion term, which we treat a bit differently for the sub- and for the supersolution. The heuristic is that if $\psi\geq C s^{3/4}$ for some $C$ large enough (i.e. $Y>C' s^{1/4}$ for some large $C'$), transport dominates the diffusion term. We prove this fact by distinguishing between two different zones:
		\begin{itemize}
			\item When $\psi^{4/3}\gg \psi^{k_\pm} \tb^{\frac{3k_\pm-2}{4}} $, i.e. $\psi\ll \tb^{-\frac{3}{4} \frac{3k_\pm-2}{3k_\pm-4}}$, we can perform asymptotic expansions of $\sqrt{W_\pm}$ and $\p_{\psi \psi} W_\pm$. We have
			$$
			\ba
			\sqrt{W_\pm}= \frac{(6\psi)^{2/3}}{2} \left(1\pm\frac{2 A_\pm}{6^{4/3}} \psi^{k_\pm-\frac{4}{3}} \tb^{\frac{3k_\pm-2}{4}} + O(\psi^{2k_\pm-\frac{8}{3}} \tb^{\frac{3k_\pm-2}{2}}) \right),\\
			\p_{\psi \psi} W_\pm = \frac{6^{4/3} \psi^{-2/3}}{9}\left(1\pm\frac{9 A_\pm k_\pm(k_\pm-1)}{6^{4/3}} \psi^{k_\pm-4/3} \tb^{\frac{3k_\pm-2}{4}}  \right),
			\ea
			$$
			so that
			$$
			-\sqrt{W_\pm}\p_{\psi \psi} W_\pm + 2 =\mp \frac{2 A_\pm}{6^{4/3} }\left(9k_\pm^2 -9k_\pm +2\right) \psi^{k_\pm-\frac{4}{3}} \tb^{\frac{3k_\pm-2}{4}} + O\left(\psi^{2k_\pm-\frac{8}{3}} \tb^{\frac{3k_\pm-2}{2}}\right) .
			$$
			Therefore, if
			$$
			\frac{3(k_\pm-2)}{4} b >\frac{4}{6^{4/3}}\left(9k_\pm^2 -9k_\pm +2\right)\psi^{-4/3}\quad \forall \psi \geq C_- \tb^{-4/3},
			$$
			then
			$$
			\p_s W_\pm - 2 b W_\pm + \frac{3b}{2} \psi \p_\psi W_\pm-\sqrt{W_\pm}\p_{\psi \psi} W_\pm + 2 \gtrless 0.
			$$
			Hence we define
			$$
			C_{k_\pm}:= 2 \left(\frac{16(9k_\pm^2 -9k_\pm+2)}{3\times 6^{4/3} (k_\pm-2)}\right)^{3/4},
			$$
			and recalling Lemma \ref{lem:tb}, we obtain inequality \eqref{in:W-} on $C_{k_\pm}\tb^{-3/4}\leq \psi \ll \tb^{-\frac{3}{4} \frac{3k_\pm-2}{3k_\pm-4}}$ provided $\eps$ is sufficiently small.
			
			\item We now consider larger values of $\psi$. This is where we treat separately the sub- and the supersolution. Concerning the subsolution, we can  use the weaker estimate
			\be\label{in:W--weak}
			\sqrt{W_-}\leq \frac{(6\psi)^{2/3}}{2} ,
			\ee
			which leads to, taking $k_-=7/3$,
			\begin{multline*}
			\p_s W_- - 2 b W_- + \frac{3b}{2} \psi \p_\psi W_--\sqrt{W_-}\p_{\psi \psi} W_- + 2\\
			\leq -A_- \frac{1}{4} b\tb^{5/4}\psi^{7/3} + \frac{6^{2/3}}{2} A_-\frac{28}{9} \psi \tb^{5/4}  +2.
			\end{multline*}
			The right-hand side above is negative if $\eps$ is small enough and
			$$
			\psi^{7/3}>\frac{16}{A_-} \tb^{-9/4}.
			$$
			Noticing that $27/28<5/4$,
			we infer that if $s_0$ is large enough,
			$\tb^{-\frac{3k_-+2}{4}}\ll \tb^{-\frac{3}{4} \frac{3k_--2}{3k_--4}},
			$
			and thus \eqref{in:W-} is proved on the set $[C_- \tb^{-3/4}, \infty)\cap \{W_->0\}$ if $s_0$ is sufficiently large.

			We now consider the supersolution, choosing $k_+:=10/3$. We replace estimate \eqref{in:W--weak} by
			$$
			\sqrt{W_+} \leq \left(\frac{(6\psi)^{2/3}}{2} + \sqrt{ A_+} \psi^{5/3} \tb \right),
			$$
			so that
			\begin{eqnarray*}
				&&\p_s W_+ - 2 b W_+ + \frac{3b}{2} \psi \p_\psi W_+ - \sqrt{ W_+} \p_{\psi \psi} W_+ + 2 \\
				&\geq &  A_+  b\tb^2\psi^{10/3} - \bar C \left(1+ A_+^{3/2} \psi^3 \tb^3\right),
			\end{eqnarray*}
		for some universal (and computable) constant $\bar C$.
			It is easily checked that the right-hand side is positive as soon as $\psi\gg s^{9/10}$. Furthermore, since $k_+=10/3$, $\frac{3}{4} \frac{3k_+-2}{3k_+-4}=1$, and therefore $s^{9/10}\ll \tb^{-\frac{3}{4} \frac{3k_+-2}{3k_+-4}}	$. Thus inequality \eqref{in:W+} is proved for $\psi \geq C_{10/3} \tb^{-3/4}$ provided $s_0$ is sufficiently large.

		\end{itemize}

		\subsubsection*{Proof of inequality \eqref{in:W-bottom}}
		
	Inequality \eqref{in:W-bottom} is an immediate consequence of the asymptotic expansion \eqref{asympt-W}. Indeed, we need to choose $A_\pm$ such that
		$$
		-A_- C_-^{7/3}\tb^{-1/2} \leq \tb^{-1/2} \left( \frac{(6C_-)^{2/3}}{2} - \frac{3}{5} a_4 \frac{b}{\tb} (6C_-)^2\right) + O(s^{\frac{a+2}{8}})\leq A_+ C_-^{10/3}\tb^{-1/2} .
		$$
		It is clear that once $C_-$ is fixed, we can pick $A_\pm$ sufficiently large (depending only on $k_\pm$ and $C_-$) so that the above inequality is satisfied provided $\eps$ is small (e.g. $\eps<1/2$, recalling Lemma \ref{lem:tb}) and $s_0$ is sufficiently large, so that the $ O(s^{\frac{a+2}{8}})$ term can be neglected.

		\subsubsection*{Proof of inequality \eqref{in:W-initial}}
	
Since $-M_0 \inf(s_0^{-1} Y^2, 1)\leq U_{YY}(s_0) -1 \leq 0$, we infer that
$$
 Y + \frac{Y^2}{2} - \frac{M_0}{12} s_0^{-1} Y^4\leq U(s_0, Y)\leq  Y + \frac{Y^2}{2} \quad \forall Y\geq 0.
$$
	Then for $1\ll \psi \ll s_0^{3/2}$, performing the same computations as the ones leading to \eqref{asympt-W},
$$
W(s_0,\psi)= \frac{(6\psi)^{4/3}}{4} \left( 1+2 (6\psi)^{-2/3} + O(s_0^{-1} \psi^{2/3} + \psi^{-1})\right)
$$
and therefore there exists a constant $M$, depending only on $M_0$, such that
$$
\frac{(6\psi)^{4/3}}{4} + \frac{(6\psi)^{2/3}}{2}- M s_0^{-1} \psi^2 -M \psi^{1/3}\leq W(s_0, \psi) \leq \frac{(6\psi)^{4/3}}{4} + \frac{(6\psi)^{2/3}}{2}+ M s_0^{-1} \psi^2 + M \psi^{1/3}.
$$
The estimate follows on the set $C_- s_0^{3/4}\leq  \psi \ll s_0^{3/2}$, provided $A_-, A_+$ are large enough (depending on $M_0$). Furthermore, recall that $W_-(s_0, C s_0^{\frac{3}{4}\frac{3k-2}{3k-4}})=0$ and $ s_0^{\frac{3}{4}\frac{3k-2}{3k-4}}\ll s_0^{3/2}$ since $k>2$ (if $s_0$ is large). Thus the inequality $W_-(s_0, \psi)\leq W(s_0,\psi)$ is valid on the domain of definition of $\psi$. On the other hand, for $\psi\geq c s_0^{6/5}$, $W_+(s_0, \psi)\geq \frac{A_+}{2} c^{10/3} s_0^2$. Therefore, since $W(s_0, \psi)\leq \lim_{Y\to \infty}U(s_0, Y)^2\lesssim s_0^2$, we infer that $W(s_0, \psi)\leq W_+(s_0, \psi)$ for $\psi \geq c s_0^{6/5}$. Since $s_0^{6/5}\ll s_0^{3/2}$, we infer that $W(s_0,\psi) \leq W_+(s_0, \psi)$ on the domain of definition of $W_+$.

	\subsubsection*{Conclusion}
	
	Putting together inequalities \eqref{in:W-}, \eqref{in:W+}, \eqref{in:W-bottom} and \eqref{in:coerciv} and applying the maximum principle on the domain $\{s\in[s_0, s_1], \psi \geq C_- \tb^{-3/4}\}$, we deduce that $W_-(s,\psi)\leq W(s,\psi)\leq W_+(s,\psi)$ within this parabolic domain.

%
%
%
%
%
%
%
%

\section{Proof of Lemma \ref{lem:U_YY-final}}

	As in  paragraph \ref{ssec:subsupW}, the real issue is to control $U_{YY}-1$ in the zone $Y\gtrsim s^{1/4}$ (or equivalently, $\psi\gtrsim s^{3/4}$). To that end, 
	we rely on the equation in von Mises variables, and we use the computations in the proof of Lemma \ref{lem:U_YY-prelim}. We set 
	$$
	F(s,\psi):=\sqrt{W} \p_{\psi\psi} W - 2
	$$
	and we recall that $F$ satisfies \eqref{eq:F}.
	We now construct a function $\uF$ such that  $|- M_2 b Y(s,\psi)^2 \leq \uF(s,\psi)\leq 0 $ for some constant $M_2 $ and for $\psi =O(s)$, and such that
	$$\ba
	\p_s \uF - \frac{1}{2W} \uF(\uF+2) + \frac{3b}{2}\psi\p_\psi \uF - \sqrt{W}\p_{\psi\psi }\uF\leq 0\text{ in } s>s_0, \psi>C\tb^{-3/4},\\
	\uF(s,\psi)\leq F(s_0,\psi) \quad \text{on } \{s_0\}\times(Cs_0^{3/4}, \infty)\cup \{\psi=C \tb^{-3/4}, \ s>s_0\} \cup \{\psi=\infty, \ s>s_0\} .\ea
	$$
	Let us postpone for a moment the construction of $\uF$ and explain why the estimate of the Lemma follows. First, notice that $(\uF - F)$ satisfies
	$$
	\p_s (\uF-F) - \frac{1}{2W} (\uF-F)(\uF+ F+2) + \frac{3b}{2}\psi\p_\psi (\uF-F) - \sqrt{W}\p_{\psi\psi } (\uF-F)\leq 0 \text{ in } s>s_0, \psi>C\tb^{-3/4},
	$$
	and $(\uF-F)_{+}=0$ on the parabolic boundary of the domain $\{s\geq s_0, \psi \geq C\tb^{-3/4} \}$. We then multiply the above inequality by $(\uF-F)_+$, integrate in $\psi$ over $[C \tb^{-3/4}, +\infty)$ and use the same argument as in Lemma \ref{lem:U_YY-prelim}.  It follows that $(\uF-F)_{+}\equiv 0$, and thus $F\geq \uF$ for all $s\geq s_0, \psi\geq C\tb^{-3/4}$. In particular, for $Y\leq c s^{1/3}$, 
	$$
	U_{YY}(s,Y)-1\geq -\frac{M_2 }{2}b  Y^2,
	$$
	and the estimate announced in the statement of the Lemma follows.

	We now turn towards the construction of $\uF$. According to Lemma \ref{lem:sub-U}, there exists $A_->0$ and $C_->0$ such that if $\psi\geq C_- \tb^{-3/4}$,
	$$
	W(s,\psi)\geq \frac{(6\psi)^{4/3}}{4} - A_- \psi^{7/3} \tb^{5/4}
	$$
	Let us construct $\uF$ by treating separately the intervals $(C_- \tb^{-3/4}, c\tb^{-5/4})$ and $(c\tb^{-5/4}, +\infty)$, for some small constant $c>0$ to be determined.

	$\bullet$  For $\psi\in (C_- \tb^{-3/4}, c\tb^{-5/4})$ we take $\uF(s,\psi)=-\tilde b \alpha \underbrace{(\psi^{2/3} -  \psi^{1/3})}_{=:g(\psi)}$, for some large constant $\alpha $ to be determined.
	Then
	\begin{eqnarray*}
		&&\p_s \uF - \frac{1}{2W} \uF(\uF+2) + \frac{3b}{2}\psi\p_\psi \uF - \sqrt{W}\p_{\psi\psi }\uF\\
		&=&-b\tilde b \frac{\alpha   }{2} \psi^{1/3} - \frac{1}{2W}\uF^2+ \tilde b \alpha \sqrt{W}\left[ \frac{g}{W^{3/2}} + g''\right].\\
	\end{eqnarray*}
	Let us evaluate the term in brackets in the right-hand side. 
	On the set  $ (C_- \tb^{-3/4}, c\tb^{-5/4})$, we have
	$$
	\frac{1}{W^{3/2}}\leq \left(\frac{(6\psi)^{4/3}}{4} - A_- \psi^{7/3} \tb^{5/4}\right)^{-3/2}= \frac{2}{9\psi^2} \left(1 + \frac{A_-}{6^{1/3}}  \psi \tb^{5/4} + O(\psi^{2} s^{-{5/2}})\right).
	$$
	Therefore
	\begin{eqnarray*}
		&&\frac{g}{W^{3/2}} + g''\\&\leq &  \frac{2}{9\psi^2} \left(1 +  \frac{A_-}{6^{1/3}} \psi\tb^{5/4}+ O(\psi^2 s^{-5/2})\right)(\psi^{2/3} -  \psi^{1/3}) - \frac{2}{9} \psi^{-4/3} + \frac{2}{9} \psi^{-5/3}\\
		&\leq& \frac{2A_-}{9 \cdot 6^{1/3}}  \psi^{-1/3} \tb^{5/4} + O(\psi^{2/3} \tb^{5/2} + \psi^{-2/3} \tb^{5/4}).
	\end{eqnarray*}
	Using Lemma \ref{lem:sub-U}, we see that $W=O(\psi^{4/3})$ for $\psi \leq c \tb^{-1}$.
	Therefore, for any $\alpha >0$, provided $c$ is small enough and $s_0$ is large enough, we have
	$$
	\left| \tilde b \alpha \sqrt{W}\left[ \frac{g}{W^{3/2}} + g''\right]\right| \leq  b\tilde b \frac{\alpha   }{4} \psi^{1/3} ,
	$$ 
	whence
	$$
	\p_s \uF - \frac{1}{2W} \uF(\uF+2) + \frac{3b}{2}\psi\p_\psi \uF - \sqrt{W}\p_{\psi\psi }\uF\leq 0 \text{ on } s\in (s_0, s_1),\ \psi \in (C_- \tb^{-3/4}, c\tb^{-1}).
	$$

	We also need to check that $\uF\leq F$ on $\{s=s_0, \psi\in (C_- s_0^{3/4}, cs_0^{5/4})\}\cup \{\psi= C_-\tb^{-3/4}\}.$ According to Lemma \ref{lem:approx-E1}, we  have, for $\psi=C_- \tb^{-3/4}$,
	$$
	F(s, \psi)= -\frac{6^{2/3}}{2}b\psi^{2/3} + O(s^{\frac{a-6}{8}}).
	$$
	 Therefore it is sufficient to take $\alpha\geq 6^{2/3}$ and $s_0$ large enough.
	On the other hand, on the set $\{s=s_0, \psi\in (C_- s_0^{3/4}, cs_0^{5/4})\}$, since we know that 
	$$
	Y + \frac{Y^2}{2} - M_0 s_0^{-1} \frac{Y^4}{12}\leq U(s_0, Y) \leq Y + \frac{Y^2}{2} \quad \forall Y\geq 0,
	$$
	 we have $Y= (6\psi)^{1/3} + O(1)$, and therefore
	 assumption \eqref{hyp:UYY-1} implies
	$$
	F(s_0, \psi)\geq - 2 M_0 6^{2/3} s_0^{-1} \psi^{2/3}.
	$$
	Therefore, choosing $\alpha=\max (6^{2/3}, 12M_0)$, we have  $\uF\leq F$ on $\{s=s_0, \psi\in (C_- s_0^{3/4}, cs_0^{5/4})\}\cup \{\psi= C_-\tb^{-3/4}\}.$ Note that since $Y\sim (6\psi)^{1/3}$ for $1\ll \psi \lesssim s^{1/3}$, this choice of $\alpha$ amounts to taking $M_2=\bar M \max (1, M_0)$.

	$\bullet$  We now define $\uF$ for $\psi\geq c\tb^{-5/4} $. On that interval, we choose $\uF=-f(s, \psi \tilde b^{5/4})$, for some function $f$ to be determined. Since $\uF, \p_{Y}\uF$ should be continuous at $\psi= c \tb^{-5/4}$, we require that
	\[\ba
	f(s,c)= \alpha \left[c^{2/3} \tb^{1/6} - c^{1/3} \tb^{7/12} \right]=:g_1(s),\\
	\p_\zeta f(s,c)= \frac{\alpha}{3} \left[ 2 c^{-1/3} \tb^{1/6} - c^{-2/3} \tb^{7/12} \right]=:g_2(s).
\ea
	\]
	As a consequence, we choose
	\[
	f(s,\zeta):=\left( g_1(s) + g_2(s)(\zeta-c)\right)\chi(\zeta) + H(\zeta),
	\]
	where $H\in \mathcal C^2\cap W^{2,\infty}(\R)$ is strictly increasing on $[c, + \infty[$ and such that $H(c)=H'(c)=0$, and $\chi\in \mathcal C^\infty_0(\R)$ is a cut-off function. We make the following additional  assumptions: there exists $c''>c'>c$ such that $\chi\equiv 1$ on $[c, c']$, $\chi\equiv 0$ on $[c'', + \infty[$, and $H''(\zeta)\leq -1$ for $\zeta \in [c', c'']$, $H''(\zeta)\leq 0$ and $2\leq H(\zeta)\leq 4$ for $\zeta \geq c''$. With this choice, and recalling that $W\geq \bar C c^{4/3} \tb^{-5/3}$ for $\psi\geq c \tb^{-5/4}$ for some universal constant $\bar C$, we have
	\begin{eqnarray*}
&&	\p_s \uF - \frac{1}{2W} \uF(\uF+2) + \frac{3b}{2}\psi\p_\psi \uF - \sqrt{W}\p_{\psi\psi }\uF\\
&\leq & -\left[g_1'(s) + g_2'(s)(\zeta-c) + \frac{1}{4} b \zeta g_2(s)\right]\chi(\zeta)\\
&&+ \bar C c^{-4/3}\tb^{5/3} \left( \left( g_1(s) + g_2(s)(\zeta-c)\right)\chi(\zeta) + H(\zeta)\right)\\
&&- \frac{1}{4} b \zeta H'(\zeta) + \sqrt{ W} \tb^{5/2} \p_\zeta^2 H\\
&& +\left( g_1(s) + g_2(s)(\zeta-c)\right)\left(- \frac{1}{4} b \zeta\chi'(\zeta) +   \sqrt{ W} \tb^{5/2} \chi''(\zeta)\right) + 2 g_2(s) \sqrt{ W} \tb^{5/2} \chi'(\zeta).
	\end{eqnarray*}
	We now prove that the right-hand side of the above inequality is non-positive by looking separately at the zones $(c'', + \infty)$, $[c', c'']$ and $[c, c']$:
	
	\begin{itemize}
	\item For $\zeta\geq c''$, we have $\chi(\zeta)=0$, and therefore
	\begin{eqnarray*}
		&&	\p_s \uF - \frac{1}{2W} \uF(\uF+2) + \frac{3b}{2}\psi\p_\psi \uF - \sqrt{W}\p_{\psi\psi }\uF\\
		&=  &  \frac{1}{2W} H(\zeta)(2-H(\zeta)) - \frac{1}{4} b \zeta H'(\zeta) + \sqrt{W}\tb^{5/2} H''.
\end{eqnarray*}
	The assumptions $H'(\zeta)\geq 0$, $H''(\zeta)\leq 0$, $H(\zeta)\geq 2$ on $(c'', + \infty)$ ensure that the right-hand side is non-positive on this interval.
	
	\item For $\zeta\in [c', c'']$, we have $H''(\zeta)\leq -1$, and without loss of generality, we also assume that $H'(\zeta)\geq 1$ on this interval. It follows that
	\begin{eqnarray*}
		&&	\p_s \uF - \frac{1}{2W} \uF(\uF+2) + \frac{3b}{2}\psi\p_\psi \uF - \sqrt{W}\p_{\psi\psi }\uF\\
		&\leq & \bar C\left( \alpha b \tb^{1/6} c^{-1/3} c'  +c^{-4/3}\tb^{5/3}  \sup_{[c',c'']} H\right)\\
		&&-\frac{1}{4} c' b  - \bar C c^{2/3} \tb^{5/3}\\
		&&+ \bar C \|\chi\|_{W^{2,\infty}} \alpha \tb^{11/6}  c^{1/3} \max (c',1).
	\end{eqnarray*}
	It is clear that the term $-1/4 c' b $ dominates all others for $s_0$ sufficiently large.
	
	\item For $\zeta\in[c,c']$, the computation is slightly more complicated because we expect that $H''(\zeta)\geq 0$ in a vicinity of $\zeta=c$, and $H'(c)=0$, so we cannot use the good sign of $H'$ in a vicinity of $\zeta=c$. However, using the formulas for $g_1,g_2$, we have
	\begin{eqnarray*}
		&&g_1'(s) + g_2'(s)(\zeta-c) + \frac{1}{4} b \zeta g_2(s)\\
		&=&\frac{\alpha b}{18} \left[{\tb^{1/6} c^{-1/3}} (\zeta -c) + {\tb^{7/12}} c^{-2/3} \left(2 \zeta + 7 c \right)  \right]\geq 0.
	\end{eqnarray*}
	Noticing that $\tb^{5/3}\ll b \tb^{7/12}$ and $ \sqrt{ W} \tb^{5/2}= O(\tb^{5/3})$ on the interval $\zeta\in [c,c']$, we infer that all terms are easily dominated by the above quantity, so that
	$$	\p_s \uF - \frac{1}{2W} \uF(\uF+2) + \frac{3b}{2}\psi\p_\psi \uF - \sqrt{W}\p_{\psi\psi }\uF\leq 0$$
	in this region as well.
		\end{itemize}

	 The assumptions on the initial data also ensure that $\uF(s_0,\psi)\leq F(s_0,\psi)$ for $\psi\geq cs_0^{5/4}$. The result follows.

\section{Proof of Lemma \ref{lem:constant-Hardy}}

Recall that
$$
\bar C_{a,\mu}:=4\sup_{r>0} \varphi(r,a,\mu),
$$
where
$$
\varphi(r,a,\mu):=\left(\int_{r}^\infty \frac{Y^{-a}}{\left(\mu Y + \frac{Y^2}{2}\right)^2}dY\right) \left(\int_{0}^{r} Y^{a}\left( Y + \frac{Y^2}{2}\right)dY\right).
$$
For all $\mu\in (0,1)$, for all $a>0$, $r>0$
\begin{eqnarray*}
\varphi(r,a,\mu)&\leq & \left(\int_{r}^\infty \frac{4}{Y^{4+a}}dY\right) \left(\int_{0}^{r} \left( Y^{1+a} + \frac{Y^{2+a}}{2}\right)dY\right)\\
&\leq& \frac{2}{(3+a)^2} + \frac{4}{(3+a)(2+a) }\; \frac{1}{r}\leq \frac{2}{9} + \frac{2}{3r}.
\end{eqnarray*}
Let $K\geq 1$ such that
$$
 \frac{8}{9} + \frac{8}{3K} \leq \frac{9}{10}.
$$
Then $4 \sup_{r\geq K}\varphi(r,a,\mu) \leq 9/10$, so that for all  $\mu\in (0,1)$, for all $a>0$,
$$
\bar C_{a,\mu} \leq \max \left( \frac{9}{10}, 4\sup_{0<r<K} \varphi(r,a,\mu)\right).
$$
Now, for all $r\in [0,K]$, for $a,\mu>0$,
\begin{eqnarray*}
\p_a \varphi (r,a,\mu) & =&  -\left(\int_{r}^\infty \ln Y\frac{Y^{-a}}{\left(\mu Y + \frac{Y^2}{2}\right)^2}dY\right) \left(\int_{0}^{r} Y^{a}\left( Y + \frac{Y^2}{2}\right)dY\right)\\
&&+  \left(\int_{r}^\infty \frac{Y^{-a}}{\left(\mu Y + \frac{Y^2}{2}\right)^2}dY\right) \left(\int_{0}^{r} \ln Y \; Y^{a}\left( Y + \frac{Y^2}{2}\right)dY\right).
\end{eqnarray*}
There exists a constant $C_K$ such that for all $a\in (0, 1/2)$, for all $\mu \in (1/2, 1)$,
$$
\int_K^\infty (1+\ln Y)\frac{Y^{-a}}{\left(\mu Y + \frac{Y^2}{2}\right)^2}dY,\quad \int_0^K (1+ |\ln Y| )Y^a \left( Y + \frac{Y^2}{2}\right)dY\leq C_K.
$$
It follows that for all $r\in [0,K]$, for all $a\in (0, 1/2)$, for all $\mu \in (1/2, 1)$,
\begin{eqnarray*}
|\p_a \varphi (r,a,\mu) | &\leq &C_K + C_K \sup_{0\leq r \leq K} \left(\int_r^K \frac{|\ln Y|}{Y^{2+a} }dY\right) \left(\int_0^r Y^{1+a} dY\right) \\&&+ C_K\sup_{0\leq r \leq K} \left(\int_r^K \frac{1}{Y^{2+a} }dY\right) \left(\int_0^r |\ln Y|Y^{1+a} dY\right).
\end{eqnarray*}
Notice that if $Y\in [0,K]$, then $|\ln Y| \leq \ln (2K)-\ln Y$. Then, performing explicit integrations by part in the integrals, we observe that there exists a constant $C_K$ such that
$$
\sup_{a\in [0,1/2]}\sup_{\mu \in [1/2, 1]}\sup_{r\in [0, K]}|\p_a \varphi (r,a,\mu) |\leq C_K.
$$
We infer that for all $r\in [0,K]$, for all $a\in [0,1/2]$, for all $\mu \in[1/2, 1]$,
$$
\varphi(r,a,\mu)\leq \varphi(r,0,\mu) + C_K a.
$$
Let us now compute explicitely $\varphi(r,0,\mu) $. 
We have
$$
 \frac{1}{\left(\mu Y + \frac{Y^2}{2}\right)^2}=\frac{1}{\mu}\left(\frac{1}{2\mu +Y}-\frac{1}{Y} \right)+ \frac{1}{Y^2} +   \frac{1}{(2\mu+Y)^2},
$$
so that
$$
\varphi(r,0,\mu) =\frac{1}{\mu^2}\left(\frac{1}{\mu}\ln \left(\frac{r}{2\mu+r}\right) + \frac{1}{r} + \frac{1}{2\mu+r}\right)\left(\frac{r^2}{2} + \frac{r^3}{6}\right).
$$
The function $(r,\mu)\mapsto \varphi(r,0,\mu)$ is $W^{1,\infty}$ in $[0,K]\times [\frac{1}{2}, 1]$, and therefore $|\varphi(r,0,\mu) - \varphi(r,0,1)|\leq  C_K |\mu-1|$.

A careful study of the function 
$$
r\mapsto \varphi(r,0,1)= \left(\ln \left(\frac{r}{2+r}\right) + \frac{1}{r} + \frac{1}{2+r}\right)\left(\frac{r^2}{2} + \frac{r^3}{6}\right)
$$
shows that it is increasing and converges towards $2/9$ as $r\to \infty$. Eventually, we obtain
\begin{eqnarray*}
\bar C_{a,\mu} &\leq & \max \left( \frac{9}{10}, 4\sup_{0\leq r\leq K } \varphi(r,0,1) + C_K a + C_K |\mu-1|\right)\\
 &\leq & \max \left( \frac{9}{10}, \frac{8}{9} + C_K a + C_K |\mu-1|\right).
\end{eqnarray*}

Therefore, choosing $a$ sufficiently small and $\mu$ sufficiently close to 1, we obtain
$$
\bar C_{a,\mu} \leq \frac{9}{10}.
$$

\section{Proofs of Lemmas \ref{lem:trace} and \ref{lem:coerciv}}

\subsection*{Proof of the trace  result (Lemma \ref{lem:trace})}

For any $Y\in [0,L]$, write
$$
\left|g(Y)-g(0)\right|=\left|\int_0^Y \p_Y g\right| \leq\frac{1}{\sqrt{1+a}} \left( \int_0^L (\p_Y g)^2 Y^{-a} dY\right)^{1/2} Y^\frac{1+a}{2}.
$$
It follows that
$$
|g(0)|^2 \leq 2 g(Y)^2 +  2 \left( \int_0^L (\p_Y g)^2 Y^{-a} dY\right) Y^{1+a}.
$$
Mutliply the above equation by $(Y+Y^2) Y^{-a}$ and integrate over $[0,L]$. We obtain
$$
L^{3-a} |g(0)|^2\leq \bar C \left(\int_0^L |g(Y)|^2 (Y+Y^2) Y^{-a}\: dY + L^{4}\left( \int_0^L (\p_Y g)^2 Y^{-a} dY\right) \right),
$$
which leads to the desired inequality.\qed

\subsection*{Proof of the coercivity result (Lemma \ref{lem:coerciv})}

Let us consider the quantity 
\be\label{D1bis}
\int_0^\infty U (\p_Y \cLU^2 V)^2 \tilde w_1,
\ee
where $\tilde w_1= Y^{-a} (1+ s^{-\beta_1} Y)^{-m_1-2} = w_1 (1+ s^{-\beta_1} Y)^{-2}$.
In order to prove the coercivity result, we go back to the diffusion term that is bounded from below by $D_1$ (plus some lower order terms), namely
$$
-\int_0^\infty \p_Y^2 \cLU^2 V \; \p_Y^2 \cLU V\; w_1,
$$
or rather, to the same integral where $w_1$ is replaced by $\tilde w_1$. We set
$$
\tilde D_1(s):= \int_0^\infty \frac{(\p_Y^2 \cLU V)^2}{U^2} \tilde w_1 + \int_0^\infty \frac{(\p_Y^3 \cLU V)^2}{U} \tilde w_1.
$$
Note that we obviously have $\tilde D_1\leq D_1$.

We recall (see the proof of Lemma \ref{lem:diff1} with $f=\p_Y^2 \cLU V$) that for any $\delta>0, P>0$, provided $m_1$ and $s_0$ are sufficiently large,
$$
\bar c \tilde D_1 - \delta b \tilde E_1 - s^{-P} - s^{-P} D_0\leq -\int_0^\infty \p_Y^2 \cLU^2 V \; \p_Y^2 \cLU V\; \tilde w_1\leq \bar c^{-1} \tilde D_1 + b \tilde E_1 + s^{-P} + s^{-P} D_0.
$$

On the other hand, set $h:=\cLU^2 V$. Then 
$
\p_Y^2 \cLU V= L_U h
$, so that, using the identity $\p_Y L_U= U \p_Y  - U_{YY}\int_0^Y$ and performing several integrations by parts,
\begin{eqnarray*}
	&&-\int_0^\infty \p_Y^2 \cLU^2 V \; \p_Y^2 \cLU V\; \tilde w_1= \int_0^\infty \p_Y h \left(\p_Y L_U h \right) \tilde w_1 + \int_0^\infty \p_Y h \: L_U h \p_Y \tilde w_1\\
	&=&\int_0^\infty U (\p_Y h)^2 \tilde w_1 - \int_0^\infty \p_Y h \left(\int_0^Y h\right) \tilde w_1+ \int_0^\infty (1-U_{YY}) \p_Y h \left(\int_0^Y h\right) \tilde w_1\\
	&& -\frac{1}{2} \int_0^\infty h^2 (U\p_Y \tilde w_1)_Y - \int_0^\infty \p_Y h  \left(\int_0^Y h\right) U_Y \p_Y \tilde w_1\\
	&=& \int_0^\infty  \left[U (\p_Y h)^2  +h^2\right] \tilde w_1\\
	&&+ \int_0^\infty (1-U_{YY}) \p_Y h \left(\int_0^Y h\right) \tilde w_1 - \frac{1}{2} \int_0^\infty \left(\int_0^Y h\right)^2 \p_{YY} \tilde w_1\\
	&&+\int_0^\infty h^2 U_Y \p_Y \tilde w_1 + \int_0^\infty h \left(\int_0^Y h\right) (U_{YY}-1) \p_Y \tilde w_1\\
	&&-\frac{1}{2}\int_0^\infty \left(\int_0^Y h\right)^2\left((1+ U_{YY}) \p_{YY} \tilde w_1 + U_Y \p_Y^3 \tilde w_1\right).
\end{eqnarray*}
The first term in the right-hand side is precisely \eqref{D1bis}. The two terms with $(U_{YY}-1)$ in the integrand can be bounded in the same fashion as the analogous remainder terms in Lemma \ref{lem:diff1}, and therefore satisfy, for any $\delta,P>0$,
\begin{multline*}
\left|  \int_0^\infty (1-U_{YY}) \p_Y h \left(\int_0^Y h\right) \tilde w_1 \right| + \left|  \int_0^\infty h \left(\int_0^Y h\right) (U_{YY}-1) \p_Y \tilde w_1 \right| \\
\leq \delta \left[D_1 + b E_1 + \int_0^\infty  \left[U (\p_Y h)^2  +h^2\right] \tilde w_1\right]+ s^{-P} + s^{-P} D_0.
\end{multline*}
Using the bound $|U_{YY}|\leq M_2$ and noticing that $\p_Y^3 \tilde w_1<0$, there remains to upper-bound
$$
\int_0^\infty h^2 U_Y |\p_Y \tilde w_1 | + \int_0^\infty \left(\int_0^Y h\right)^2\p_{YY} \tilde w_1.
$$
Notice that $|\p_Y \tilde w_1| \leq C_{m_1, a} Y^{-1} \tilde w_1$, and $\p_{YY} \tilde w_1 \leq C_{m_1,a} Y^{-2} \tilde w_1$. Hence, using a Hardy inequality together with the bounds on $U_Y$, it is sufficient to upper-bound
$$
\int_0^\infty \left(1 + Y^{-1}\right) h^2 \tilde w_1.
$$
Let us first consider the integral between 0 and 1. By a Hardy inequality, we have
$$
\int_0^1 \left(1 + Y^{-1}\right) h^2 \tilde w_1 \leq 4 \int_0^1 h^2 \frac{dY}{Y^{1+a}}\leq C_a \int_0^1 Y^{1-a} (\p_Y h)^2 dY.
$$
Since
$$
\p_Y h = U_{YY}\int_0^Y \frac{\p_Y^2 \cLU V}{U^2} + \frac{\p_Y^3 \cLU V}{U},
$$
we have, for $Y\in (0,1)$,
$$
| \p_Y h|^2 \leq C_a \tilde D_1 Y^a + \frac{(\p_Y^3 \cLU V)^2}{Y^2}, 
$$
and therefore
$$
\int_0^1 \left(1 + Y^{-1}\right) h^2 \tilde w_1 \leq C_a \tilde D_1.
$$
There remains to control the integral for $Y\geq 1$. To that end, we write $$h=\LU (\p_Y^2 \cLU V)= U_Y\int_0^Y \frac{\p_Y^2 \cLU V}{U^2}+ \frac{\p_Y^2 \cLU V}{U}.$$ Once again, a simple Cauchy-Schwartz inequality yields
$$
\left( \int_0^Y \frac{\p_Y^2 \cLU V}{U^2}\right)^2 \leq \left(\int_0^\infty \frac{(\p_Y^2 \cLU V)^2}{Y^2 U}  w_1\right)\left( \int_0^Y \frac{Y^2}{U^3  w_1}\right)\leq C_{m_1,a} \tilde D_1 (1+ Y^{-3}  w_1^{-1}). 
$$
It follows that
\begin{eqnarray*}
	\int_1^\infty h^2 \tilde w_1&\leq & 2 \int_1^\infty \frac{(\p_Y^2 \cLU V)^2}{U^2} \tilde w_1 + C_{m_1, a }  D_1 \int_1^\infty (1 + Y^{-3} { w_1}^{-1}) (1+Y)^2 \tilde w_1\\
	&\leq & 2 \tilde D_1 + C_{m_1,a} s^{(3-a)\beta_1 } D_1 \leq C_{m_1,a }s^{(3-a)\beta_1}  D_1 .
\end{eqnarray*}
Eventually, we infer that for any $P>0$, provided $m_1$ and $s_0$ are large enough, for any $s\geq s_0$,
\be\label{in:E32}
\int_0^\infty  \left[U (\p_Y \cLU V)^2  +(\cLU V)^2\right] \tilde w_1 \leq  C_{m_1,a }s^{(3-a)\beta_1}  D_1  + b E_1 + s^{-P} + s^{-P} D_0.
\ee

\bibliography{prandtl-separation}
\end{document}